\documentclass[12pt]{amsart}
\usepackage{format}

\renewcommand{\codim}{\operatorname{codim}}

\makeatletter
\newcommand*\bigcdot{\mathpalette\bigcdot@{.7}}
\newcommand*\bigcdot@[2]{\mathbin{\vcenter{\hbox{\scalebox{#2}{$\m@th#1\bullet$}}}}}
\makeatother

\newcommand{\adjo}{\mathrel{\vcenter{\offinterlineskip \ialign{##\cr$\rightarrow$\cr\noalign{\kern 0pt}$\leftarrow$\cr}}}}
\newcommand{\cocolon}{\nobreak\mskip8mu plus1mu \mathpunct{}\nonscript\mkern-\thinmuskip{:}\mskip2mu\relax}

\newcommand{\mf}[1]{{\mathfrak{#1}}}
\newcommand{\mb}[1]{{\mathbb{#1}}}
\newcommand{\mc}[1]{{\mathcal{#1}}}
\newcommand{\mrm}[1]{{\mathrm{#1}}}

\DeclareMathOperator{\mhm}{MHM}

\DeclareMathOperator{\Gr}{Gr}
\DeclareMathOperator{\coh}{Coh}
\DeclareMathOperator{\coker}{coker}
\DeclareMathOperator{\pro}{Pro}

\DeclareMathOperator{\supp}{Supp}

\title[Hodge theory, intertwining functors, and the Orbit Method]{Hodge theory, intertwining functors, and the Orbit Method for real reductive groups}

\author{Dougal Davis}

\author{Lucas Mason-Brown}

\thanks{DD was supported by the ARC grant FL200100141.}

\begin{document}


\begin{abstract}
We study the Hodge filtrations in the sense of Schmid and Vilonen on unipotent representations of real reductive groups. We show that for various well-defined classes of unipotent representations (including, for example, the oscillator representations of metaplectic groups, the minimal representations of all simple groups, and all unipotent representations of complex groups) the Hodge filtration coincides with the quantization filtration predicted by the Orbit Method. We deduce a number of longstanding conjectures about such representations, including a proof that they are unitary and a description of their $K$-types in terms of co-adjoint orbits. The proofs rely heavily on certain good homological properties of the Hodge filtrations on weakly unipotent representations, which are established using a Hodge-theoretic upgrade of the Beilinson-Bernstein theory of intertwining functors for $\mathcal{D}$-modules on the flag variety. The latter consists of an action of the affine Hecke algebra on a category of filtered monodromic $\mathcal{D}$-modules, which we use to compare Hodge filtrations coming from different localizations of the same representation. As an application of the same methods, we also prove a new cohomology vanishing theorem for mixed Hodge modules on partial flag varieties.
\end{abstract}

\maketitle

\tableofcontents

\section{Introduction}

The classification of irreducible unitary representations of a real reductive group $G_{\RR}$ is one of the major unsolved problems in representation theory. Two  conceptual approaches to this problem have emerged: the Orbit Method philosophy of Kostant and Kirillov and the Hodge theory approach of Schmid and Vilonen. In this paper, we combine these two approaches, gaining new insight into both.

The Orbit Method is an idea (dating back to around 1960) that irreducible unitary representations should naturally arise as quantizations of co-adjoint orbits. While it has been difficult to make precise in the context of reductive groups, the Orbit Method offers useful predictions about the overall structure of the unitary dual (that is, the set of all irreducible unitary representations). In particular, the Orbit Method suggests that the unitary dual is generated (see, e.g., \S\ref{subsec:conjecture} below) by a small set of building blocks called `unipotent representations'. Some such representations were defined in \cite{LMBM} using filtered quantizations of symplectic singularities, but it has remained a conjecture that these representations are unitary in general.

Hodge theory, on the other hand, provides a powerful set of tools for studying representations, and in particular for determining which are unitary. This is based on the observation that every irreducible representation of $G_{\RR}$ admits a canonical `Hodge' filtration coming from its Beilinson-Bernstein realization as a (twisted) $\mathcal{D}$-module on the complex flag variety. It was conjectured in \cite{SchmidVilonen2011}, then proved in \cite{DV}, that the Hodge filtration can be used to detect unitarity. In general, the Hodge filtration is an extremely complicated invariant, containing fine information about the geometry of $K$-orbits on the flag variety. However, we show in this paper that the Hodge filtration on a \emph{unipotent} representation is of a very simple form: its associated graded is identified with the sections of a vector bundle on a nilpotent co-adjoint $K$-orbit. In other words, for unipotent representations, the filtrations arising from Hodge theory and the Orbit Method coincide. We exploit this remarkable fact to prove a number of longstanding conjectures about unipotent representations, including their unitarity.

The main ingredient in our analysis of the Hodge filtration on a unipotent representation is a general theorem that, assuming only certain properties of the annihilator, the Hodge associated graded must be a Cohen-Macaulay module. From here the description as a vector bundle quickly follows. The Cohen-Macaulay property is proved using a new Hodge-theoretic upgrade of the Beilinson-Bernstein intertwining functor theory, which is likely to be of independent interest. As an application of the same methods, we also prove a cohomology vanishing theorem for Hodge modules on partial flag varieties, generalizing a result of Vilonen and the first named author \cite[Theorem 1.3]{DV}.

\subsection{Unipotent representations and the Orbit Method}

Originating in the work of Kostant (\cite{Kostant1959}) and Kirillov (\cite{Kirillov1962}), the Orbit Method proposes a correspondence between irreducible unitary $G_{\RR}$-representations and orbits for the co-adjoint action of $G_{\RR}$ on $\fg_{\RR}^*$. By the Jordan decomposition, every such orbit can be constructed from a \emph{nilpotent} co-adjoint orbit for a Levi subgroup of $G_\mb{R}$ and a character of its Lie algebra. Accordingly, the Orbit Method predicts a finite set of `unipotent representations', conjecturally attached to nilpotent co-adjoint orbits, from which the entire unitary dual can be constructed by various forms of induction. See \S\ref{subsec:conjecture} below for a precise conjecture.

The problem of defining the set of unipotent representations has a long and rich history, which we will not summarize here. In the monograph \cite{LMBM}, Losev, Matvieievskyi, and the second-named author gave a precise definition of unipotent representations in the case of \emph{complex} groups. To a complex reductive group $G$ and a finite connected cover $\widetilde{\OO}$ of a nilpotent co-adjoint orbit, they attach a two-sided ideal $I(\widetilde{\OO})\subset U(\fg)$ in the universal enveloping algebra. This ideal is defined to be the kernel of a ring homomorphism $U(\fg) \to \cA_0(\widetilde{\OO})$ to a certain canonical quantization of the ring of regular functions $\CC[\widetilde{\OO}]$. A \emph{unipotent ideal} is a two-sided ideal which arises in this fashion; a \emph{unipotent representation} is an irreducible Harish-Chandra $U(\fg)$-bimodule (thought of as the Harish-Chandra module of a representation of $G$) which is annihilated on both sides by a unipotent ideal. Such representations were classified in \cite{LMBM} and \cite{MBM} and shown to arise as quantizations of $G$-equivariant vector bundles on nilpotent co-adjoint orbits. It was conjectured in \cite{LMBM} (and verified by hand for classical $G$) that all such representations are unitary. Our first contribution is a proof of this conjecture in general.

\begin{thm}[Theorem \ref{thm:complexunitary}] \label{thm:intro complex unipotent}
Let $G$ be a complex reductive group, and let $M$ be a unipotent Harish-Chandra $U(\mf{g})$-bimodule. Then $M$ is unitary.
\end{thm}

In \cite{AdamsBarbaschVogan}, Adams, Barbasch, and Vogan, following ideas of Arthur, defined a class of representations called \emph{special unipotent}. It is an important longstanding conjecture that all such representations are unitary. On the other hand, it was shown in \cite{LMBM} that all special unipotent representations are unipotent in the sense defined above. So Theorem \ref{thm:intro complex unipotent} implies the unitarity of special unipotent representations in the case of complex groups.

\begin{cor}
All special unipotent representations of complex reductive groups are unitary.
\end{cor}

Now suppose that $G_{\RR}$ is a \emph{real reductive group}. For us, this will mean a (possibly disconnected) Lie group equipped with a finite map onto a finite-index subgroup of a real form of a complex reductive group $G$. This class of groups of course includes all real linear groups, but also certain non-linear groups such as $\mathrm{Mp}(2n,\RR)$. Choose a maximal compact subgroup $K_{\RR}$ of $G_{\RR}$ with complexification $K$. Following \cite{LMBM} and \cite{MBM}, we consider the following two classes of unipotent $G_{\RR}$-representations (modeled as Harish-Chandra $(\mf{g}, K)$-modules).

\begin{definition}[Definitions \ref{def:birigidunipotent} and  \ref{def:smallbdryunipotent}]\label{def:unipotentintro}
Let $M$ be an irreducible $(\mf{g}, K)$-module. Assume that $M$ is annihilated by a unipotent ideal $I(\mb{O})$ attached to a nilpotent $G$-orbit $\mb{O} \subset \mf{g}^*$. We say that $M$ is
\begin{itemize}
\item \emph{birationally rigid unipotent} if $\mb{O}$ is birationally rigid (see \S\ref{subsec:unipotent}), and
\item \emph{small-boundary unipotent} if the associated variety of $M$ (i.e., the support of $\gr(M)$ with respect to a good filtration) is the closure of a $K$-orbit $\OO_{\theta} \subset (\fg/\fk)^*$ such that
\[\codim(\partial \OO_{\theta},\overline{\OO}_{\theta}) \geq 2.\]
\end{itemize}
\end{definition}

Although there is substantial overlap between these two classes of representations, neither is contained in the other. We note that most interesting `small' representations (including the oscillator representations and all minimal representations outside of type $A$) are unipotent in both senses. Although Definition \ref{def:unipotentintro} should not be regarded as the most general definition of unipotent representations, we conjecture that it is sufficient for the purposes of describing the unitary dual, see \S \ref{subsec:conjecture} below. We also note that the restriction to unipotent ideals attached to orbits (rather than covers) is mostly for convenience: our methods also work for many nilpotent covers (see Remark \ref{rmk:covers}).

Our second contribution is a proof that all unipotent representations (as in Definition \ref{def:unipotentintro}) are unitary.

\begin{theorem}[Corollary \ref{cor:unitaryunipotent}]\label{thm:unitarityunipotentintro}
Let $G_{\RR}$ be a real reductive group, and let $M$ be a birationally rigid or small-boundary unipotent $(\fg,K)$-module. Then $M$ is unitary.
\end{theorem}

In fact, we prove a much more general result (from which Theorem \ref{thm:unitarityunipotentintro} is deduced). 
 
\begin{theorem}[Theorem \ref{thm:unitaritycriterion}]\label{thm:unitaritygeneralintro}
Let $G_{\RR}$ be a real reductive group, and let $M$ be an irreducible $(\fg,K)$-module which satisfies the following conditions:
\begin{enumerate}[label=(\roman*), ref=\roman*]
    \item \label{itm:unitaritygeneralintro 1} $M$ is Hermitian (i.e., admits a non-degenerate invariant Hermitian form).
    \item \label{itm:unitaritygeneralintro 2} The associated cycle of $M$ is irreducible (see Definition \ref{def:irreducibleAC}). 
    \item \label{itm:unitaritygeneralintro 3} The annihilator of $M$ is very weakly unipotent (see Definition \ref{def:weaklyunipotent}).
    \item \label{itm:unitaritygeneralintro 4} The annihilator of $M$ is a maximal two-sided ideal in $U(\fg)$.
\end{enumerate}
Then $M$ is unitary.
\end{theorem}
We remark that the notion of a \emph{very weakly unipotent} ideal in Theorem \ref{thm:unitaritygeneralintro} is a further weakening of the notion, due originally to Vogan \cite[Definition 8.16]{Vogan1984}, of a weakly unipotent ideal. Roughly, an ideal is very weakly unipotent if it is weakly unipotent for the adjoint form of $G$ (although the actual definition is slightly weaker still).

To verify that all unipotent representations satisfy properties \eqref{itm:unitaritygeneralintro 1}-\eqref{itm:unitaritygeneralintro 4} above is nontrivial, but not particularly deep. Property \eqref{itm:unitaritygeneralintro 4} is \cite[Theorem 5.01]{MBM}. For \eqref{itm:unitaritygeneralintro 3}, we use the explicit calculations of infinitesimal characters in \cite{LMBM,MBM} to check case by case. For \eqref{itm:unitaritygeneralintro 1}-\eqref{itm:unitaritygeneralintro 2}, we use $\mathcal{W}$-algebra techniques, in particular the dagger functors of Losev (\cite{Losev3}).

Just as for complex groups, our results for real groups have implications for special unipotent representations. If $\OO^{\vee}$ is a nilpotent co-adjoint orbit for the Langlands dual group $G^{\vee}$ such that the Barbasch-Vogan dual $d(\OO^{\vee})$ is birationally rigid, then all special unipotent representations attached to $\OO^{\vee}$ (in the sense of \cite{AdamsBarbaschVogan}) are birationally rigid unipotent (in the sense of Definition \ref{def:unipotentintro}). So Theorem \ref{thm:unitarityunipotentintro} implies:

\begin{cor}
Let $G_{\RR}$ be a real reductive group, and let $\OO^{\vee}$ be a nilpotent co-adjoint orbit for the Langlands dual group $G^{\vee}$ such that the Barbasch-Vogan dual $d(\OO^{\vee})$ is birationally rigid. Then all special unipotent representations of $G_{\RR}$ attached to $\OO^{\vee}$ are unitary.
\end{cor}

\subsection{A conjectural description of the unitary dual}\label{subsec:conjecture}

The new results above suggest a simple conjectural description of the unitary dual of a real reductive group. The template for our conjecture is already contained in the writings of Vogan on the Orbit Method (see \cite{Vogan1987},\cite{Voganorbit}). The essential new idea is our definition of `unipotent' (Definition \ref{def:unipotentintro}) and our proof of unitarity (Theorem \ref{thm:unitaritygeneralintro}).

Let  $\Pi_u(G_{\RR})$ (resp. $\Pi_0(G_{\RR})$) denote the set of equivalence classes of irreducible unitary (resp. birationally rigid unipotent) $G_{\RR}$-representations. Let $L_{\RR}$ be a Levi subgroup of $G_{\RR}$ (i.e., the centralizer in $G_{\RR}$ of a semisimple element of $\mathfrak{g}_{\RR}$). There are three basic procedures for producing unitary representations of $G_{\RR}$ from unitary representations of $L_{\RR}$, namely \emph{real parabolic induction}, \emph{unitarity-preserving cohomological induction}, and the formation of \emph{complementary series}. The first and third procedures require that $L_{\RR}$ is the Levi factor of a real parabolic subgroup, whereas the second procedure requires that $L_{\RR}$ is the real points of the Levi factor of a $\theta$-stable parabolic. For the conjecture below, it is important to consider a certain generalization (involving filtered quantizations of nilpotent covers) of the classical complementary series construction. However, for the purposes of this discussion, we will leave this notion imprecise. If $S \subset \Pi_u(L_{\RR})$, let $I(S) \subset \Pi_u(G_{\RR})$ denote the set of irreducible unitary $G_{\RR}$-representations obtained from representations in $S$ by applying some sequence of the three procedures above (and then extracting direct summands).

\begin{conj}\label{conjintro}
There is an equality
$$\Pi_u(G_{\RR}) = \bigcup_{L_{\RR} \subset G_{\RR}} I(\Pi_0(L_{\RR})),$$
where the union runs over all $G_{\RR}$-conjugacy classes of Levi subgroups of $G_{\RR}$. 
\end{conj}
We note that Theorem \ref{thm:unitarityunipotentintro} gives an inclusion of the right hand side into the left. It remains to establish the opposite inclusion.

\subsection{Hodge theory and unipotent representations}

Let us now turn to our second conceptual approach to unitarity: the Hodge theoretic approach of Schmid and Vilonen. While the Orbit Method gives \emph{predictions} of what unitary representations should look like, the purpose of the Hodge-theoretic approach is to provide \emph{tools} to prove theorems about them. The basis for doing so is a conjecture of Schmid and Vilonen \cite{SchmidVilonen2011}, now a theorem of Vilonen and the first named author \cite{DV}, which relates unitary representations to deep structures in complex geometry.

The starting point is the theory of Beilinson-Bernstein localization for $U(\mf{g})$-modules \cite{BB1}, a (by now) standard tool that allows one to apply results about the topology of algebraic varieties (such as the Decomposition Theorem) to the representation theory of Lie groups and Lie algebras. Recall the Harish-Chandra isomorphism
\begin{align*}
\mf{h}^*/W &\cong \Hom(\mathfrak{Z}(\mf{g}), \mb{C}) \\
\lambda &\mapsto \chi_\lambda,
\end{align*}
where $\mf{h}^*$ is the dual of the Lie algebra $\mf{h}$ of the abstract Cartan $H$ of $G$, $W$ is the abstract Weyl group, and $\mf{Z}(\mf{g}) = Z(U(\mf{g}))$. For each $\lambda \in \mf{h}^*$, Beilinson and Bernstein define a sheaf $\mc{D}_\lambda$ of twisted differential operators on the flag variety $\mc{B}$ and a global sections functor
\[ \Gamma \colon \Mod(\mc{D}_\lambda) \to \Mod(U(\mf{g}))_{\chi_\lambda},\]
where $\Mod(U(\mf{g}))_{\chi_\lambda}$ is the category of $U(\mf{g})$-modules with infinitesimal character $\chi_\lambda$. When $\lambda$ is integrally dominant, this functor is a Serre quotient. In particular, every irreducible $(\mf{g}, K)$-module $M$ of infinitesimal character $\chi_\lambda$ is of the form $M = \Gamma(\mc{M})$ for a unique irreducible $(\mathcal{D}_{\lambda},K)$-module $\mathcal{M}$.

Now, just as the classical Hodge theory of complex algebraic varieties is an enhancement of the theory of de Rham cohomology, the theory of $\mc{D}$-modules has a Hodge-theoretic enhancement called \emph{mixed Hodge modules} \cite{Saito1988, Saito1990} (see also \cite{MHMproject}). Roughly speaking, mixed Hodge modules are $\mc{D}$-modules equipped with Hodge-like structures satisfying deep rigidity and functoriality properties with respect to Grothendieck's six operations. The point of departure in \cite{SchmidVilonen2011} is the observation that, if $\lambda \in \mf{h}^*_\mb{R} := \mb{X}^*(H) \otimes \mb{R}$ is real (we assume this from now on), then there is a good theory of mixed Hodge $\mc{D}_\lambda$-modules such that the irreducible $(\mc{D}_\lambda, K)$-module $\mc{M}$ can be lifted to this category in an essentially unique way. In particular, $\mc{M}$ carries a canonical good filtration $F_\bullet^H \mc{M}$, called the \emph{Hodge filtration}, unique up to a shift in indexing. Taking global sections, we obtain a Hodge filtration $F_\bullet^H M = \Gamma(F_\bullet^H\mc{M})$ on the $(\mf{g}, K)$-module $M$.

As formulated above, this construction suffers from a slight ambiguity: given a $(\fg,K)$-module $M$ with real infinitesimal character $\chi$, we must first choose $\lambda \in \mf{h}^*_\mb{R}$ such that $\chi = \chi_\lambda$. The filtration one obtains depends seriously on this choice. Unless otherwise specified, we will always define the Hodge filtration on $M$ by choosing $\lambda$ to be \emph{dominant}. With this convention, the main theorem in \cite{DV} is that the unitarity of $M$ can be read off of the associated graded module $\Gr^{F^H}M$; see \S\ref{subsec:HCmodules} for a precise statement.

Unfortunately, calculating the Hodge associated graded $\Gr^{F^H} M$ is still a difficult problem in general, although it has been carried out in certain cases (see, e.g., \cite[Theorems 1.5 and 1.10]{DV}). In the setting of unipotent representations, however, we show that $\Gr^{F^H}M$ satisfies an additional homological property, which simplifies matters considerably.

\begin{thm}[Theorem \ref{thm:CM}]\label{thm:intro CM}
Let $M$ be an irreducible $(\mf{g}, K)$-module with Hodge filtration $F_\bullet^H M$. Assume the annihilator of $M$ is a maximal and very weakly unipotent ideal in $U(\fg)$. Then $\Gr^{F^H} M$ is a Cohen-Macaulay $(S(\mf{g}), K)$-module.
\end{thm}

Since Cohen-Macaulay modules are completely determined by their localizations outside codimension $2$ subsets of their support, Theorem \ref{thm:intro CM} reduces the computation of $\Gr^{F^H} M$ to a computation in codimension $0$ and $1$. This observation implies Theorem \ref{thm:unitaritygeneralintro} almost immediately, and gives a complete description of $\Gr^{F^H} M$ for small-boundary unipotent representations.

\begin{theorem}[Theorem \ref{thm:Kspectraunipotent}]\label{thm:unipotentquantizationintro}
Let $M$ be a small-boundary unipotent $(\fg,K)$-module with Hodge filtration $F_\bullet^HM$. Then there exists a $K$-orbit $\OO_{\theta} \subset (\fg/\fk)^*$ and an irreducible $K$-equivariant vector bundle $\mathcal{V}$ on $\OO_{\theta}$ such that
\[ \gr^{F^H}(M) \xrightarrow{\sim} \Gamma(\OO_{\theta},\mathcal{V})\]
as $(S(\fg),K)$-modules. In particular, if we choose $x \in \OO_{\theta}$ and write $\rho$ for the corresponding irreducible representation of the stabilizer $K_{x}$ of $x$, then
\[M \simeq \Ind^K_{K_{x}}\rho\]
as representations of $K$.
\end{theorem}

The final statement about $K$-multiplicities in Theorem \ref{thm:unipotentquantizationintro} is a conjecture of Vogan \cite[Conjecture 12.1]{Vogan1991}, based on the expectation that unipotent representations should quantize nilpotent co-adjoint orbits. A pleasing feature of our Hodge-theoretic approach is that this property of unipotent representations is proved together with their unitarity, in keeping with the underlying philosophy of the Orbit Method.

\subsection{Vanishing theorems for Hodge filtrations}

Let us now discuss the proof of Theorem \ref{thm:intro CM}, which in fact takes up the bulk of this paper. Recall that a coherent sheaf on a variety (or module over a commutative ring) is Cohen-Macaulay if and only if its Serre dual is concentrated in a single cohomological degree. By a fundamental result of Saito (e.g., \cite[Lemme 5.1.13]{Saito1988}), the associated graded $\Gr^{F^H}\mc{M}$ of the mixed Hodge module $\mc{M} \in \mhm(\mc{D}_\lambda)$ is a Cohen-Macaulay sheaf on $T^*\mc{B}$, with Serre dual given by
\[ \mb{D}_{\mrm{Serre}} \Gr^{F^H}\mc{M} = \Gr^{F^H}\mb{D}\mc{M}\]
up to a cohomological shift. Here $\mb{D}\mc{M} \in \mhm(\mc{D}_{-\lambda})$ is the dual Hodge module. Unfortunately, it does not {\it a priori} follow that the associated graded $\Gr^{F^H} M$ of the global sections is also Cohen-Macaulay: to draw this conclusion, we need to assume that
\begin{equation} \label{eq:intro CM vanishing}
\mrm{H}^i(\mc{B}, \Gr^{F^H}\mc{M}) = 0 \quad \text{and} \quad \mrm{H}^j(\mc{B}, \Gr^{F^H}\mb{D}\mc{M}) = 0, \quad \text{for $i \neq 0$ and $j \neq d$},
\end{equation}
for some $d$. The first of these conditions is \cite[Theorem 1.3]{DV}, which holds since $\lambda$ is dominant. Since $-\lambda$ is \emph{not} dominant, however, this result does not apply to the dual $\mb{D}\mc{M}$: we instead use the following theorem to reduce \eqref{eq:intro CM vanishing} to the vanishing of $\mrm{H}^i(\mc{B}, \mb{D}\mc{M})$, which can be proved by an algebraic argument using maximality of the annihilator.

\begin{thm}[Theorem \ref{thm:strictness}] \label{thm:intro strictness}
Assume $\mc{M} \in \mhm(\mc{D}_\lambda)$ is very weakly unipotent (Definition \ref{def:CandD}). Then, with no dominance assumption on $\lambda$, the spectral sequence
\[ \mrm{E}_1^{p, q} = \mrm{H}^{p + q}(\mc{B}, \Gr^{F^H}_{-p}\mc{M}) \Rightarrow \mrm{H}^{p + q}(\mc{B}, \mc{M}) \]
degenerates at $\mrm{E}_1$.
\end{thm}

Theorem \ref{thm:intro strictness} also implies the following generalization of \cite[Theorem 1.3]{DV} to partial flag varieties. If $\mc{P} = G/P$ is a partial flag variety, we have a family of sheaves $\mc{D}_{\mc{P}, \lambda}$ of twisted differential operators on $\mc{P}$ indexed by $\lambda \in \mrm{Pic}^G(\mc{P}) \otimes \mb{C} = \mb{X}^*(P) \otimes \mb{C} \subset \mf{h}^*$. We fix the indexing so that $\mc{D}_{\mc{P}, \rho(\mf{g}/\mf{p})} = \mc{D}_\mc{P}$ is the ordinary sheaf of untwisted differential operators, where $\rho(\mf{g}/\mf{p})$ is half the sum of the positive roots in $\mf{g}/\mf{p}$ (or equivalently, half the class of the anti-canonical bundle of $\mc{P}$).

\begin{thm}[Theorem \ref{thm:partial exactness}] \label{thm:intro partial vanishing}
Let $\mc{P}$ be a partial flag variety, let $\lambda \in \mrm{Pic}^G(\mc{P}) \otimes \mb{R} \subset \mf{h}^*_\mb{R}$ and let $\mc{M} \in \mhm(\mc{D}_{\mc{P}, \lambda})$. If $\lambda$ is dominant, then
\[ \mrm{H}^i(\mc{P}, F_p^H \mc{M}) = 0 \quad \text{for $i > 0$ and all $p$}.\]
\end{thm}

Theorem \ref{thm:intro partial vanishing} is the Hodge-theoretic analog of the Beilinson-Bernstein vanishing for partial flag varieties (see, e.g., \cite[Theorem 6.3]{Bien1990}). Note that the pullback of a $\mc{D}_{\mc{P}, \lambda}$-module to $\mc{B}$ is a $\mc{D}_{\lambda + \rho - \rho(\mf{g}/\mf{p})}$-module: even if $\lambda$ is dominant, $\lambda + \rho - \rho(\mf{g}/\mf{p})$ need not be. 

\subsection{Intertwining functors}

Finally, Theorem \ref{thm:intro strictness} is proved using a theory of \emph{intertwining functors} for Hodge modules to reduce the study of Hodge filtrations at non-dominant twists to the dominant case treated in \cite{DV}.

The classical intertwining functors for $\mc{D}$-modules on the flag variety were defined in \cite{BeilinsonBernstein1983}, see also \cite{HMSW2}. For the purposes of this introduction, these are functors
\[ \mc{I}_w \colon \mrm{D}^b\coh(\mc{D}_\lambda) \to \mrm{D}^b\coh(\mc{D}_{w\lambda}) \]
for $w \in W$ and $\lambda \in \mf{h}^*$ such that the diagram
\[
\begin{tikzcd}[column sep=1em]
\mrm{D}^b\coh(\mc{D}_\lambda) \ar[rd, "\mrm{R}\Gamma"'] \ar[rr, "\mc{I}_w"] & & \mrm{D}^b\coh(\mc{D}_{w\lambda}) \ar[dl, "\mrm{R}\Gamma"] \\
& \mrm{D}^b\mrm{Mod}_{fg}(U(\mf{g}))_{\chi_\lambda}
\end{tikzcd}
\]
commutes. Here $\mrm{Mod}_{fg}(U(\mf{g}))_{\chi_\lambda}$ denotes the category of finitely generated $U(\mf{g})$-modules with infinitesimal character $\chi_\lambda = \chi_{w\lambda}$ and $\coh(\mc{D}_\lambda)$ the category of coherent $\mc{D}_\lambda$-modules.

Now, the rings $\mc{D}_\lambda$ are constructed (see \S\ref{subsec:localization}) as quotients of a single sheaf $\tilde{\mc{D}}$ on $\mc{B}$ (a central extension of $\mc{D}$ by $S(\mf{h})$); thus, we may consider all infinitesimal characters together by embedding the various categories $\Coh(\mathcal{D}_{\lambda})$ into $\Coh(\tilde{\mathcal{D}})$. We lift the intertwining functors to the filtered derived category $\mrm{D}^b\coh(\tilde{\mc{D}}, F_\bullet)$ of coherent $\tilde{\mc{D}}$-modules with good filtration:

\begin{thm}[Corollary \ref{cor:affine hecke} and Theorem \ref{thm:intertwining}] \label{thm:intro filtered intertwining}
For $w \in W$ and $\mu \in \mb{X}^*(H)$, there are explicit functors
\[ \mc{I}_w^!, \mc{I}_w^* = (\mc{I}_{w^{-1}}^!)^{-1}, \ t_\mu = \mc{O}(\mu) \otimes - \colon \mrm{D}^b\coh(\tilde{\mc{D}}, F_\bullet) \to \mrm{D}^b\coh(\tilde{\mc{D}}, F_\bullet), \]
where $\mc{O}(\mu)$ is the line bundle on the flag variety corresponding $\mu \in \mb{X}^*(H)$, such that:
\begin{enumerate}
\item the induced operators on the Grothendieck group $\mrm{K}(\mrm{D}^b\coh(\tilde{\mc{D}}, F_\bullet))$ generate an action of the (extended) affine Hecke algebra in its Bernstein presentation, and
\item for all $w \in W$, the diagram
\[
\begin{tikzcd}[column sep=1em]
\mrm{D}^b\coh(\tilde{\mc{D}}, F_\bullet) \ar[rr, "\mc{I}_{w}^!"] \ar[dr, "\mrm{R}\Gamma"'] & & \mrm{D}^b\coh(\tilde{\mc{D}}, F_\bullet) \ar[dl, "\mrm{R}\Gamma"] \\
& \mrm{D}^b\mrm{Mod}_{fg}(U(\mf{g}), F_\bullet)
\end{tikzcd}
\]
commutes.
\end{enumerate}
\end{thm}

The filtered intertwining functors of Theorem \ref{thm:intro filtered intertwining} are constructed Hodge-theoretically as follows. First, we consider the equivariant derived category $\mrm{D}^b_G\mhm(\mc{D}_{\widetilde{0}} \boxtimes \mc{D}_{\widetilde{0}})$ of mixed Hodge $\tilde{\mc{D}}\boxtimes \tilde{\mc{D}}$-modules on $\mc{B} \times \mc{B}$ on which the center $S(\mf{h}) \boxtimes S(\mf{h})$ acts nilpotently. This is a Hodge-theoretic version of the \emph{monodromic Hecke category}, which has been studied extensively in geometric representation theory. Following a well-trodden path (see, e.g., \cite{bezrukavnikov-yun}), we complete this to a category $\widehat{\mrm{D}}^b_G\mhm(\mc{D}_{\widetilde{0}} \boxtimes \mc{D}_{\widetilde{0}})$ of well-behaved pro-objects on which the center may now act freely (see \S\S\ref{subsec:pro mhm} and \ref{subsec:intertwining construction}). This comes equipped with a `forgetful' functor
\[\widehat{\mrm{D}}^b_G\mhm(\mc{D}_{\widetilde{0}} \boxtimes \mc{D}_{\widetilde{0}}) \to \mrm{D}^b_{\mrm{St}} \coh^G(\tilde{\mc{D}} \boxtimes \tilde{\mc{D}}, F_\bullet)\]
 to an appropriate derived category $\mrm{D}^b_{\mrm{St}} \coh^G(\tilde{\mc{D}} \boxtimes \tilde{\mc{D}}, F_\bullet)$ of filtered $\tilde{\mc{D}} \boxtimes \tilde{\mc{D}}$-modules, which acts on $\mrm{D}^b\coh(\tilde{\mc{D}}, F_\bullet)$ by convolution. We define $\mc{I}_w^!$ by convolution with the image of a \emph{free-monodromic standard Hodge module} $\tilde{\Delta}_w^{(0)}$ under this functor. We remark that the filtered $\tilde{\mc{D}}\boxtimes \tilde{\mc{D}}$-module $\tilde{\Delta}_w^{(0)}$ has a simple concrete description (see \cite{BIR} and \cite[Remark 6.18]{DV}), although we will not use this here.

This appearance of the affine Hecke algebra is closely related to Springer theory. Taking associated gradeds defines an algebra homomorphism
\[  \Gr \colon \mrm{K}(\mrm{D}^b_{\mrm{St}}\coh^G(\tilde{\mc{D}} \boxtimes \tilde{\mc{D}}, F_\bullet)) \to \mrm{K}(\coh^{G \times \mrm{\mb{C}}^\times}(\mrm{St})),\]
where $\mrm{St} = \tilde{\mf{g}}^* \times_{\mf{g}^*} \tilde{\mf{g}}^*$ is the Steinberg variety. The target is isomorphic to the affine Hecke algebra, and it is well-known to experts that $\Gr$ sends the objects $\tilde{\Delta}_w^{(0)}$ to its Bernstein generators (cf., e.g., \cite{Tanisaki1987}).

Now, with appropriate conventions, one can show (see, for example, the proof of Corollary \ref{cor:intertwining}) that the filtered intertwining functors $\mc{I}_w^!$ restrict to the usual intertwining functors $\mc{I}_w$ for $\mc{D}_\lambda$-modules upon specializing the infinitesimal character and forgetting the filtration. Unfortunately, $\mc{I}_w^!$ does \emph{not} restrict to a functor between mixed Hodge module categories in a way which is compatible with the Hodge filtrations. However, by convolving with a free-monodromic standard object
\[ \tilde{\Delta}_w^{(\lambda)} \in \widehat{\mrm{D}}^b_G\mhm(\mc{D}_{\widetilde{w\lambda}} \boxtimes \mc{D}_{\widetilde{-\lambda}}) \]
in a monodromic Hecke category with non-trivial twist, we obtain functors
\[ \mc{I}_w^{!(\lambda)} \colon \mrm{D}^b\coh(\tilde{\mc{D}}, F_\bullet) \to \mrm{D}^b\coh(\tilde{\mc{D}}, F_\bullet) \quad \text{and} \quad \mc{I}_w^{!H} \colon \mrm{D}^b\mhm(\mc{D}_\lambda) \to \mrm{D}^b\mhm(\mc{D}_{w\lambda})\]
such that
\[ (\mc{I}_w^{!H} \mc{M}, F_\bullet^H) \cong \mc{I}_{w}^{!(\lambda)} (\mc{M}, F_\bullet^H).\]
The difference between $\mc{I}_w^{!(\lambda)}$ and $\mc{I}_w^!$ measures the failure of $\mc{I}_w^{!H}$ to commute with global sections. This difference can be quantified precisely. For example, we show:

\begin{thm} \label{thm:intro categorified hecke}
Let $\alpha$ be a simple root with corresponding simple reflection $s_\alpha \in W$ and let $\lambda \in \mf{h}^*_\mb{R}$. If $\langle \lambda, \check\alpha \rangle \geq 0$, then for all $(\mc{M}, F_\bullet) \in \mrm{D}^b\coh(\tilde{\mc{D}}, F_\bullet)$, there is a distinguished triangle
\[ \mc{I}_{s_\alpha}^{!(\lambda)}(\mc{M}, F_\bullet) \to \mc{I}_{s_\alpha}^!(\mc{M}, F_\bullet) \to \mc{K} \to \mc{I}_{s_\alpha}^{!(\lambda)}(\mc{M}, F_\bullet)[1]\]
such that the cone $\mc{K}$ lies in the full triangulated subcategory generated by the objects
\[ t_{-n\alpha}(\mc{M}, F_\bullet) \quad \text{for $n \in \mb{Z}$ with $0 < n \leq \langle  \lambda, \check\alpha \rangle$}\]
and their filtration shifts. A similar statement holds for $\langle \lambda, \check\alpha \rangle < 0$.
\end{thm}

In fact, we prove a more precise statement (Theorem \ref{thm:deformation relations}) relating the Hodge-filtered $\tilde{\mc{D}}\boxtimes \tilde{\mc{D}}$-modules $\tilde{\Delta}_{s_\alpha}^{(\lambda)}$ and $\tilde{\Delta}_{s_\alpha}^{(0)}$. For example, when $\lambda \in \mb{X}^*(H)$, Theorem \ref{thm:deformation relations} categorifies the Bernstein relation
\[ t_{s_\alpha\lambda} T_{s_\alpha} t_{-\lambda} = T_{s_\alpha} +  (u - 1) (t_{-\alpha} + \cdots + t_{-\langle \lambda, \check\alpha\rangle\alpha})\]
in the affine Hecke algebra after applying the identity $\mc{I}_w^{!(\lambda)} = t_{w\lambda} \mc{I}_w^! t_{-\lambda}$.

Finally, for very weakly unipotent modules, we are able to show that the error term $\mc{K}$ in Theorem \ref{thm:intro categorified hecke} has vanishing cohomology in all degrees and hence that $\mc{I}_w^{!H}$ commutes with global sections for such modules after all. This implies Theorem \ref{thm:intro strictness}, which in turn implies Theorem \ref{thm:intro CM}, and hence our main results on unipotent representations.

\subsection{Outline of the paper}

In \S\ref{sec:hodge generalities}, we briefly recall the basics of mixed Hodge modules and set up some general categorical machinery for talking about mixed Hodge module structures on certain non-holonomic $\mc{D}$-modules (such as free-monodromic standards), similar to \cite[Appendix]{bezrukavnikov-yun}. In \S\ref{sec:intertwining}, we recall the basics of Beilinson-Bernstein localization and apply the machinery of \S\ref{sec:hodge generalities} to the construction of intertwining functors for Hodge modules and filtered $\mc{D}$-modules on the flag variety. Here we also establish the main properties of the intertwining functors (e.g., Theorems \ref{thm:intro filtered intertwining} and \ref{thm:intro categorified hecke}). In \S\ref{sec:CM}, we introduce the notion of a very weakly unipotent module and apply the intertwining functor theory to prove vanishing theorems for their Hodge filtrations (Theorems \ref{thm:intro CM}, \ref{thm:intro strictness} and \ref{thm:intro partial vanishing}). Finally, in \S\ref{sec:unipotent}, we recall necessary aspects of the theory of unipotent ideals and representations and the Hodge-theoretic unitarity criterion of \cite{DV} and prove Theorems \ref{thm:intro complex unipotent}, \ref{thm:unitarityunipotentintro}, \ref{thm:unitaritygeneralintro}, and \ref{thm:unipotentquantizationintro}.

\subsection*{Acknowledgements}
We would like to thank Jeff Adams, Ivan Losev, Jia-Jun Ma, Kari Vilonen, David Vogan and Shilin Yu for helpful conversations. We especially thank Shilin Yu for explaining to us why Theorem \ref{thm:intro complex unipotent} follows from our other results, and for pointing out a mistake in our original proof that unipotent ideals are very weakly unipotent.

\section{Generalities on mixed Hodge modules and completion} \label{sec:hodge generalities}

In this section, we collect some technical results on mixed Hodge modules and their pro-completions, which we use in the study of intertwining functors for Hodge modules. We briefly recall some key aspects of the theory of mixed Hodge modules in \S\ref{subsec:twisted mhm}. In \S\ref{subsec:pro mhm}, we define a completion of the category of monodromic mixed Hodge modules on a torus bundle over a smooth variety, and a forgetful functor to a category of filtered modules on the base. In \S\ref{subsec:free-monodromic}, we consider the special case of free-monodromic local systems and classify Hodge structures on them. Finally, in \S\ref{subsec:semi-continuity}, we give refinement of the deformation and wall-crossing theory of \cite[\S 3]{DV} and apply this to extensions of free-monodromic local systems.

\subsection{Twisted and monodromic mixed Hodge modules} \label{subsec:twisted mhm}

The theory of mixed Hodge modules was first developed by Morihiko Saito in \cite{Saito1988, Saito1990}. We will use a version of the theory with complex coefficients, due to Sabbah and Schnell \cite{MHMproject}, which is better adapted for our purposes. We refer the reader to \cite[\S 2]{DV} for a short summary of this theory, including the extension to twisted and monodromic $\mc{D}$-modules. We will follow the notation and conventions therein, with the exception that we denote the Hodge filtration by $F_\bullet^H$ instead of $F_\bullet$ to distinguish it from other filtrations that arise in this paper. We briefly recall here the most important aspects of the theory.

Recall that a \emph{(complex) mixed Hodge structure} is a finite dimensional complex vector space $V$ equipped with finite increasing filtrations $W_\bullet V$ (the weight filtration), $F_\bullet^H V$ (the Hodge filtration), and $\bar{F}_\bullet^H V$ (the conjugate Hodge filtration) such that for all $w \in \mb{Z}$, we have
\[ \Gr^W_w V = \bigoplus_{p + q = w} (\Gr^W_w V)^{p, q} \quad \text{where} \quad (\Gr^W_w V)^{p, q} = F^H_{-p} \Gr^W_w V \cap \bar{F}^H_{-q} \Gr^W_w V.\]
A basic example is the \emph{Tate structure} $\mb{C}(1)$, defined to be the $1$-dimensional complex vector space $\mb{C}$ with $W_\bullet$ jumping in degree $-2$ and $F_\bullet^H = \bar{F}_\bullet^H$ jumping in degree $1$. The set of mixed Hodge structures forms a $\mb{C}$-linear abelian category $\mrm{MHS}$.

The theory of mixed Hodge modules defines an analogous abelian category $\mhm(X)$ for every smooth complex variety $X$, such that $\mhm(\mrm{pt}) = \mrm{MHS}$. A mixed Hodge module $\mc{M} \in \mhm(X)$ consists of a regular holonomic $\mc{D}_X$-module (which we also denote by $\mc{M})$ equipped with a weight filtration, Hodge filtration and conjugate Hodge filtration. The weight filtration is a finite filtration $W_\bullet\mc{M}$ by $\mc{D}_X$-submodules. The Hodge filtration is a good filtration $F_\bullet^H \mc{M}$ compatible with the filtration $F_\bullet \mc{D}_X$ by order of differential operator. The conjugate Hodge filtration (which we will not discuss in detail) is defined not on $\mc{M}$ but on an auxiliary `complex conjugate' $\mc{D}$-module---see, e.g., \cite[\S 2.1]{DV} for details. The quadruple $(\mc{M}, W_\bullet, F_\bullet^H, \bar{F}_\bullet^H)$ is required to satisfy a complicated set of additional conditions, which we do not recall here. Morally, the conditions amount to requiring that the specialization of a mixed Hodge module to any point should be a mixed Hodge structure.

The basic example of a mixed Hodge module is a polarizable variation of Hodge structure: this is a mixed Hodge module $\mc{M}$ such that the underlying $\mc{D}$-module is a vector bundle and $W_\bullet \mc{M}$ jumps at a single degree. As a special case, any unitary local system $\mc{V}$ on $X$ defines a mixed Hodge module with `trivial' Hodge structure
\[
\begin{gathered}
 W_w \mc{V} = \begin{cases} \mc{V}, & \text{if $w \geq \dim X$} \\ 0, & \text{otherwise,}\end{cases} \quad
 F_p^H\mc{V} = \bar{F}_p^H\mc{V}= \begin{cases} \mc{V}, &\text{if $p \geq 0$}, \\ 0, &\text{otherwise.}\end{cases}
\end{gathered}
\]
Note that for a variation of Hodge structure, the conjugate Hodge filtration can be described explicitly as a filtration on the smooth vector bundle underlying $\mc{V}$ by anti-holomorphic sub-bundles as we have done above. More generally, a mixed Hodge module in which the underlying $\mc{D}$-module is a vector bundle with possibly non-trivial weight filtration is the same thing as a graded-polarizable admissible variation of mixed Hodge structure; these can be described in a similarly explicit way.

The main point of mixed Hodge module theory is that, in addition to containing the basic objects above, the derived categories $\mrm{D}^b\mhm(X)$ are equipped with Grothendieck's six operations. We will use the conventions of \cite[Notation 2.2]{DV} for these functors and other sheaf operations. In practice, one constructs complicated examples of mixed Hodge modules by applying the six operations (especially direct image functors) to the basic examples coming from variations of Hodge structure.

An important consequence of the six operations is that, for $\mc{M}, \mc{N} \in \mhm(X)$, the complex vector space $\Hom_{\mc{D}_X}(\mc{M}, \mc{N})$ comes equipped with a canonical mixed Hodge structure, such that the evaluation map
\[ \Hom_{\mc{D}_X}(\mc{M}, \mc{N}) \otimes \mc{M} \to \mc{N} \]
is a morphism of mixed Hodge modules. In particular, any morphism $f \in F_p^H\Hom_{\mc{D}_X}(\mc{M}, \mc{N})$ (resp., $W_w\Hom_{\mc{D}_X}(\mc{M}, \mc{N})$) satisfies
\[ f(F_{p'}^H\mc{M}) \subset F_{p + p'}^H\mc{N} \quad \text{(resp.,} \;\; f(W_{w'}\mc{M}) \subset W_{w + w'} \mc{N} \; \text{).}\]

Given a torus $H$ and an $H$-torsor $\pi\colon \tilde X \to X$, a twisted version of the theory can be defined as follows. Consider the sheaf of rings on $X$
\[ \tilde{\mc{D}} = \tilde{\mc{D}}_X := \pi_{\bigcdot}(\mc{D}_{\tilde X})^H,\]
where $\pi_{\bigcdot}$ denotes sheaf-theoretic direct image. The center of $\tilde{\mc{D}}$ is a copy of the symmetric algebra $S(\mf{h})$, where $\mf{h} = \mrm{Lie}(H)$, generated by the image of the derivative map $\mf{h} \to \mc{T}_{\tilde X}$ defined by the $H$-action on $\tilde X$. The sheaf $\tilde{\mc{D}}$ is locally free over this center and satisfies
\[ \tilde{\mc{D}} \otimes_{S(\mf{h}), 0} \mb{C} = \mc{D}_X.\]
Thus, we may regard $\tilde{\mc{D}}$ as a flat deformation of $\mc{D}_X$ over $\mf{h}^*$. For $\lambda \in \mf{h}^*$, a quasi-coherent $\tilde{\mc{D}}$-module $\mc{M}$ is said to be \emph{$\lambda$-twisted} (resp., \emph{$\lambda$-monodromic}) if $h - \lambda(h)$ acts on $\mc{M}$ by $0$ (resp., nilpotently) for all $h \in \mf{h}$. We write
\[ \Mod_\lambda(\mc{D}_{\tilde{X}}),\, \Mod_{\widetilde{\lambda}}(\mc{D}_{\tilde X}) \subset \Mod(\tilde{\mc{D}}) \]
for the full subcategories of $\lambda$-twisted and $\lambda$-monodromic modules respectively. The notation is motivated by the fact that, for fixed $\lambda$, the pullback functor
\[ \pi^{\bigcdot} \colon \Mod(\tilde{\mc{D}}) \to \Mod(\mc{D}_{\tilde X}) \]
is fully faithful when restricted to the subcategories $\Mod_\lambda(\mc{D}_{\tilde X})$ and $\Mod_{\widetilde{\lambda}}(\mc{D}_{\tilde X})$. We will often regard these as subcategories of $\Mod(\mc{D}_{\tilde X})$ in this way.

Similarly, we say that a mixed Hodge module $\mc{M}$ on $\tilde{X}$ is $\lambda$-twisted (resp., $\lambda$-monodromic) if its underlying $\mc{D}_{\tilde{X}}$-module lies in $\Mod_\lambda(\mc{D}_{\tilde X})$ (resp., $\Mod_{\widetilde{\lambda}}(\mc{D}_{\tilde X})$). We write $\mhm_\lambda(\tilde X)$ and $\mhm_{\widetilde{\lambda}}(\tilde X)$ for the corresponding full subcategories of $\mhm(\tilde X)$. Note that these categories are zero unless $\lambda \in \mf{h}^*_\mb{R} := \mb{X}^*(H) \otimes \mb{R}$. By \cite[Proposition 2.9]{DV}, for example, if $\mc{M} \in \mhm_{\widetilde{\lambda}}(\tilde X)$, then the Hodge and weight filtrations on $\tilde{X}$ descend to Hodge and weight filtrations on the corresponding $\tilde{\mc{D}}$-module on $X$, which we also denote by $F_\bullet^H$ and $W_\bullet$. In particular, we have a `forgetful' functor
\begin{align*}
 \mhm_{\widetilde{\lambda}}(\tilde X) &\to \coh(\tilde{\mc{D}}, F_\bullet) \\
 \mc{M} &\mapsto (\mc{M}, F_\bullet^H),
\end{align*}
where $\coh(\tilde{\mc{D}}, F_\bullet)$ is the category of coherent $\tilde{\mc{D}}$-modules equipped with a good filtration compatible with the order filtration $F_\bullet\tilde{\mc{D}}$. (We will sometimes call the latter \emph{coherent filtered $\tilde{\mc{D}}$-modules}.)

Finally, we remark that if an algebraic group $G$ acts compatibly on $X$ and $\tilde{X}$, then one can define $G$-equivariant monodromic mixed Hodge modules and their equivariant derived category $\mrm{D}^b_G\mhm_{\widetilde{\lambda}}(\tilde{X})$ in the usual way.

\subsection{Pro-mixed Hodge modules and filtered $\tilde{\mc{D}}$-modules} \label{subsec:pro mhm}

The above provides a theory of mixed Hodge structures on $\tilde{\mc{D}}$-modules on which the center acts with a single generalized eigenvalue. To construct our intertwining functors, we will also need to incorporate Hodge structures on $\tilde{\mc{D}}$-modules on which the center acts freely. We achieve this by adding certain pro-objects to the theory using a version of the machinery in \cite[Appendix]{bezrukavnikov-yun}. 

In the definition below, we recall that, for any category $\mc{C}$, the category $\pro \mc{C}$ of pro-objects is the full subcategory of $\mrm{Fun}(\mc{C}, \mrm{Set})^{\mathit{op}}$ generated under countable inverse limits by the image of the Yoneda embedding. More concretely, every pro-object $A$ is of the form $A = \varprojlim_n A_n$ for some $A_n \in \mc{C}$, and morphisms are given by
\[ \Hom_{\pro \mc{C}}(\varprojlim_n A_n, \varprojlim_m B_m) = \varprojlim_m \varinjlim_n \Hom_\mc{C}(A_n, B_m), \quad \text{for $A_n, B_m \in \mc{C}$}.\]
We also recall from \cite[Proposition 2.9]{DV}, for example, that for $\mc{M} \in \mhm_{\widetilde{\lambda}}(\tilde X)$ the nilpotent action of $\mf{h}$ on $\mc{M}$ defined by
\begin{align*}
\mf{h} &\to S(\mf{h}) \subset \tilde{\mc{D}} \\
h &\mapsto h - \lambda(h)
\end{align*}
determines a morphism of mixed Hodge modules $\mf{h}(1) \otimes \mc{M} \to \mc{M}$, where $\mf{h}(1) = \mf{h} \otimes \mb{C}(1)$.

\begin{defn} \label{defn:good pro-object}
Let $X$ be a smooth variety, $H$ a torus, $\pi \colon \tilde{X} \to X$ an $H$-torsor and $G$ an algebraic group acting compatibly on $X$ and $\tilde X$. For $\lambda \in \mf{h}^*_\mb{R}$, we say that a pro-object
\[ \mc{M} \in \pro \mrm{D}^b_G \mhm_{\widetilde{\lambda}}(\tilde X)\]
in the $G$-equivariant derived category is \emph{good} if the pro-object
\[ \mc{M} \overset{\mrm{L}}\otimes_{S(\mf{h}(1))} \mb{C} \in \pro \mrm{D}^b_G\mhm_{\widetilde{\lambda}}(\tilde X) \]
is constant (i.e., lies in the essential image of the Yoneda embedding). We write
\[ \widehat{\mrm{D}}^b_G\mhm_{\widetilde{\lambda}}(\tilde X)  \subset \pro \mrm{D}^b_G\mhm_{\widetilde{\lambda}}(\tilde X) \]
for the full subcategory of good pro-objects, and
\[ \widehat{\mhm}_{\widetilde{\lambda}}^G(\tilde X) = \widehat{\mrm{D}}^b_G\mhm_{\widetilde{\lambda}}(\tilde X) \cap \pro \mhm^G_{\widetilde{\lambda}}(\tilde X).\]
\end{defn}

\begin{prop}
The category $\widehat{\mrm{D}}^b_G\mhm_{\widetilde{\lambda}}(\tilde X)$ is triangulated, with the shift functor induced from the shift on $\mrm{D}^b_G\mhm_{\widetilde{\lambda}}(\tilde X)$ and exact triangles given by inverse limits of exact triangles in $\mrm{D}^b_G\mhm_{\widetilde{\lambda}}(\tilde X)$. Moreover, the subcategory $\widehat{\mhm}_{\widetilde{\lambda}}^G(\tilde X) \subset \widehat{\mrm{D}}^b_G\mhm_{\widetilde{\lambda}}(\tilde X)$ is the heart of a bounded t-structure.
\end{prop}
\begin{proof}
The techniques of \cite[Appendix A]{bezrukavnikov-yun} show that $\widehat{\mrm{D}}^b_G\mhm_{\widetilde{\lambda}}(\tilde X)$ is triangulated. To prove the statement about the t-structure, consider the case where $G$ acts on $X$ with finitely many orbits and on $\tilde{X}$ with unipotent stabilizers: in this setting, the argument of \cite[Lemma A.6.1]{bezrukavnikov-yun} applies. We will only consider this situation, so we omit the proof of the general case.
\end{proof}

As shown in \cite{DV}, the good pro-objects have a theory of underlying filtered $\tilde{\mc{D}}$-module: if $\mc{M} \in \widehat{\mhm}_{\widetilde{\lambda}}^G(\mc{D})$, then the pro-objects
\[ F_p^H \mc{M} \in \pro \coh^G(X) \]
are constant for all $p$, i.e., are actually coherent sheaves on $X$. The filtered $\tilde{\mc{D}}$-module underlying $\mc{M}$ is
\[ (\mc{M}, F_\bullet^H) := \left(\bigcup_p F_p^H \mc{M}, F_\bullet^H \mc{M}\right).\]

\begin{prop} \label{prop:forgetful coherence}
The filtered $\tilde{\mc{D}}$-module $(\mc{M}, F_\bullet^H)$ is coherent.
\end{prop}
\begin{proof}
This was proved in a special case in \cite[Proposition 6.15]{DV}; the exact same proof applies in general.
\end{proof}

We can extend this to the level of derived categories as follows. Let us write
\[ \tilde{\mu} \colon T^*\tilde{X}/H \to \mf{g}^*\]
for the moment map of the $G$-action on $\tilde{X}$ and let
\[ \mrm{D}^b_{\tilde{\mu}^{-1}(0)} \coh^G(\tilde{\mc{D}}, F_\bullet) \subset \mrm{D}^b\coh^G(\tilde{\mc{D}}, F_\bullet) \]
denote the full subcategory of complexes $(\mc{M}, F_\bullet)$ of weakly $G$-equivariant filtered $\tilde{\mc{D}}$-modules such that the cohomology sheaves of $\Gr^F\mc{M}$ are supported set-theoretically on $\tilde{\mu}^{-1}(0)$. We also write
\[ \mrm{D}^b_{\tilde{\mu}^{-1}(0)} \coh^G_{\widetilde{\lambda}} (\tilde{\mc{D}}, F_\bullet)\subset \mrm{D}^b_{\tilde{\mu}^{-1}(0)} \coh^G(\tilde{\mc{D}}, F_\bullet)\]
for the full subcategory on which $h - \lambda(h)$ acts nilpotently for all $h \in \mf{h}$. Then the filtered complex underlying any object in $\mrm{D}^b_G\mhm_{\widetilde{\lambda}}(\tilde X)$ lies in $\mrm{D}^b_{\widetilde{\mu}^{-1}(0)} \coh_{\widetilde{\lambda}}^G(\tilde{\mc{D}}, F_\bullet)$. Consider the \emph{completion functor}
\begin{equation} \label{eq:pro forgetful functor 1}
\begin{aligned}
\mrm{D}^b_{\tilde{\mu}^{-1}(0)}\coh^G(\tilde{\mc{D}}, F_\bullet) &\to \pro \mrm{D}^b_{\tilde{\mu}^{-1}(0)} \coh^G_{\widetilde{\lambda}}(\tilde{\mc{D}}, F_\bullet) \\
\mc{M} &\mapsto \widehat{\mc{M}} := \varprojlim_{n} S(\mf{h})_{\lambda, n} \overset{\mrm{L}}\otimes_{S(\mf{h})} \mc{M},
\end{aligned}
\end{equation}
where $S(\mf{h})_{\lambda, n}$ is the filtered quotient $S(\mf{h})/I_\lambda^n$, where
\[ I_\lambda = \ker(S(\mf{h}) \xrightarrow{\lambda} \mb{C}).\]

\begin{prop} \label{prop:pro forgetful functor}
The completion functor \eqref{eq:pro forgetful functor 1} is fully faithful, and the forgetful functor
\begin{equation} \label{eq:pro forgetful functor 4}
\widehat{\mrm{D}}^b_G \mhm_{\widetilde{\lambda}}(\tilde X) \to \pro \mrm{D}^b_{\tilde{\mu}^{-1}(0)} \coh^G(\tilde{\mc{D}}, F_\bullet)_{\widetilde{\lambda}}
\end{equation}
factors through its essential image. Hence, we have a well-defined forgetful functor
\[ \widehat{\mrm{D}}^b_G \mhm_{\widetilde{\lambda}}(\tilde X) \to \mrm{D}^b_{\tilde{\mu}^{-1}(0)}\coh^G(\tilde{\mc{D}}, F_\bullet).\]
\end{prop}
\begin{proof}
Let us first show that the completion functor is fully faithful. To see this, suppose that $\mc{M}_1, \mc{M}_2 \in \mrm{D}^b_{\tilde{\mu}^{-1}(0)}\coh^G(\tilde{\mc{D}}, F_\bullet)$; we need to show that
\begin{equation} \label{eq:pro forgetful functor 3}
\Hom(\mc{M}_1, \mc{M}_2) = \varprojlim_m \varinjlim_n \Hom(S(\mf{h})_{\lambda, n} \overset{\mrm{L}}\otimes_{S(\mf{h})} \mc{M}_1, S(\mf{h})_{\lambda, m} \overset{\mrm{L}}\otimes_{S(\mf{h})} \mc{M}_2).
\end{equation}
For fixed $m$, we have
\[ \Hom(\mc{M}_1, S(\mf{h})_{\lambda, m} \overset{\mrm{L}}\otimes_{S(\mf{h})} \mc{M}_2) =\varinjlim_{n}\Hom(S(\mf{h})_{\lambda, n} \overset{\mrm{L}}\otimes_{S(\mf{h})} \mc{M}_1 , S(\mf{h})_{\lambda, m} \overset{\mrm{L}}\otimes_{S(\mf{h})} \mc{M}_2 );\]
this follows formally from the fact that $S(\mf{h})$ acts with generalized eigenvalue $\lambda$ on the target. Now, we can resolve $(\mc{M}_1, F_\bullet)$ by a bounded complex whose terms are finite direct sums of the form $\tilde{\mc{D}}\otimes_{\mc{O}_X}\mc{F}\{p\}[-q]$ for $p, q \in \mb{Z}$ and $\mc{F}$ a $G$-equivariant coherent sheaf on $X$. We have, for any $\mc{P}$,
\begin{equation} \label{eq:pro forgetful functor 2}
\Hom_{(\tilde{\mc{D}}, F_\bullet)}(\tilde{\mc{D}} \otimes_{\mc{O}_X}\mc{F}\{p\}[-q], \mc{P}) = \mrm{Ext}^q_{\mc{O}_X}(\mc{F}, F_p\mc{P}).
\end{equation}
For each $p$, the pro-object
\[ \varprojlim_m F_p(S(\mf{h})_{\lambda, m}\overset{\mrm{L}}\otimes_{S(\mf{h})} \mc{M}_2)\]
is constant, with value $F_p\mc{M}_2$. Applying \eqref{eq:pro forgetful functor 2}, we deduce
\[ \Hom_{(\tilde{\mc{D}}, F_\bullet)}(\tilde{\mc{D}} \otimes_{\mc{O}_X}\mc{F}\{p\}[-q], \mc{M}_2) = \varprojlim_m \Hom_{(\tilde{\mc{D}}, F_\bullet)}(\tilde{\mc{D}} \otimes_{\mc{O}_X}\mc{F}\{p\}[-q], S(\mf{h})_{\lambda, m}\overset{\mrm{L}}\otimes_{S(\mf{h})} \mc{M}_2))\]
and hence \eqref{eq:pro forgetful functor 3}. This proves full faithfulness of the completion functor.

To prove that \eqref{eq:pro forgetful functor 4} factors through the image of \eqref{eq:pro forgetful functor 1}, let us first suppose that $X$ is proper. Then, from \eqref{eq:pro forgetful functor 2}, the triangulated category $\mrm{D}^b_{\tilde{\mu}^{-1}(0)}\coh^G(\tilde{\mc{D}}, F_\bullet)_{\widetilde{\lambda}}$ has finite dimensional hom spaces. Hence, we may apply \cite[Theorem A.2.2]{bezrukavnikov-yun} and the argument of \cite[Example A.2.4]{bezrukavnikov-yun} to conclude that the essential image of \eqref{eq:pro forgetful functor 1} is a triangulated category (with exact triangles defined as inverse limits of exact triangles) and that \eqref{eq:pro forgetful functor 1} is a triangulated functor. Since \eqref{eq:pro forgetful functor 4} sends exact triangles to limits of exact triangles by definition, it therefore suffices to show that each $\mc{M} \in \widehat{\mhm}^G_{\widetilde{\lambda}}(\tilde{X})$ is mapped into the image of \eqref{eq:pro forgetful functor 1} i.e., that there exists a coherent filtered $\tilde{\mc{D}}$-module whose completion is the image of $\mc{M}$. But $(\mc{M}, F_\bullet^H)$ constructed above is such a module by Proposition \ref{prop:forgetful coherence}, so we are done in this case.

Finally, if $X$ is not proper, choose a smooth $G$-equivariant compactification $j \colon X \hookrightarrow Y$ and an $H$-torsor $\tilde{Y}$ over $Y$ whose pullback to $X$ is $\tilde{X}$. Then we have a commutative diagram
\[
\begin{tikzcd}
\widehat{\mrm{D}}^b_G \mhm_{\widetilde{\lambda}}(\tilde{Y}) \ar[d] & \widehat{\mrm{D}}^b_G \mhm_{\widetilde{\lambda}}(\tilde{X}) \ar[l, "j_*"'] \ar[d, "\eqref{eq:pro forgetful functor 4}"] \\
\pro \mrm{D}^b_{\tilde{\mu}^{-1}(0)}\coh^G(\tilde{\mc{D}}_Y, F_\bullet)_{\widetilde{\lambda}} \ar[r, "j^*"] & \pro \mrm{D}^b_{\tilde{\mu}^{-1}(0)} \coh^G(\tilde{\mc{D}}_X, F_\bullet)_{\widetilde{\lambda}} \\
\mrm{D}^b_{\tilde{\mu}^{-1}(0)}\coh^G(\tilde{\mc{D}}_Y, F_\bullet) \ar[r] \ar[u] & \mrm{D}^b_{\tilde{\mu}^{-1}(0)}\coh^G(\tilde{\mc{D}}_X, F_\bullet) \ar[u, "\eqref{eq:pro forgetful functor 1}"'].
\end{tikzcd}
\]
Since $\widehat{\mrm{D}}^b_G \mhm_{\widetilde{\lambda}}(\tilde{Y}) \to \pro \mrm{D}^b_{\tilde{\mu}^{-1}(0)}\coh^G(\tilde{\mc{D}}_Y, F_\bullet)_{\widetilde{\lambda}}$ factors through the essential image of $\mrm{D}^b_{\tilde{\mu}^{-1}(0)}\coh^G(\tilde{\mc{D}}_Y, F_\bullet)$, we deduce that \eqref{eq:pro forgetful functor 4} factors through the essential image of \eqref{eq:pro forgetful functor 1} as claimed.
\end{proof}

\subsection{Hodge structures on free-monodromic local systems} \label{subsec:free-monodromic}

In this subsection, we recall the construction of free-monodromic local systems from \cite[\S A.4]{bezrukavnikov-yun} and describe their lifts to objects in $\widehat{\mhm}_{\widetilde{\lambda}}^G(\tilde X)$.

Continuing with the setup of the previous subsection, suppose now that $G$ acts transitively on $X$ and choose a base point $x \in X$ with stabilizer $G_x$. The group $G_x$ necessarily acts on the $H$-torsor $\pi^{-1}(x)$ via a homomorphism $G_x \to H$. We write $\mf{h}_x \subset \mf{h}$ for the image of the derivative of this homomorphism, and $R_x  = S(\mf{h}/\mf{h}_x)$. In the statement below, we write $\Mod_{\mathit{nilp}}(R_x)$ for the category of $R_x$-modules on which each $h \in \mf{h}/\mf{h}_x$ acts nilpotently.

\begin{prop} \label{prop:loc sys classification}
Assume that the stabilizer $G_x$ is connected. Then we have an equivalence of categories
\[ \Mod_{\widetilde{\lambda}}^G(\mc{D}_{\tilde{X}}) \cong \begin{cases} \Mod_{\mathit{nilp}}(R_x), & \text{if $\lambda$ integrates to a character of $G_x$,} \\ 0, & \text{otherwise.}\end{cases}\]
The equivalence sends a $G$-equivariant $\tilde{\mc{D}}$-module $\mc{M}$ (necessarily a vector bundle) to its fiber at $x$, with $h \in \mf{h}/\mf{h}_x \subset R_x$ acting via $h - \lambda(h) \in S(\mf{h}) \subset \tilde{\mc{D}}$. The subalgebra $\mf{h}_x \subset \mf{h}$ and the equivalence above are independent of the choice of base point $x \in X$.
\end{prop}
\begin{proof}
Restricting to $\pi^{-1}(x)$ defines an equivalence of categories
\[ \Mod_{\widetilde{\lambda}}^G(\mc{D}_{\tilde{X}}) \cong \Mod_{\widetilde{\lambda}}^{G_x}(\mc{D}_{\pi^{-1}(x)}).\]
Since $x$ is a point, $\tilde{\mc{D}}_x = S(\mf{h})$, so by definition
\[ \Mod_{\widetilde{\lambda}}(\mc{D}_{\pi^{-1}(x)}) = \Mod(S(\mf{h}))_{\widetilde{\lambda}} \]
is the category of $S(\mf{h})$-modules with generalized eigenvalue $\lambda$. Adding an action of $G_x$ on the corresponding $\mc{D}_{\pi^{-1}(x)}$-module, this lifts to an equivalence
\[ \Mod_{\widetilde{\lambda}}^{G_x}(\mc{D}_{\pi^{-1}(x)}) = \Mod(\mf{h}, G_x)_{\widetilde{\lambda}},\]
with the category of $(\mf{h}, G_x)$-modules on which $\mf{h}$ acts with generalized eigenvalue $\lambda$. Note that the definition of equivariance for $\mc{D}$-modules enforces that the derivative of the $G_x$-action must agree with the restriction of the $\mf{h}$-action. In particular, $\Mod(\mf{h}, G_x)_{\widetilde{\lambda}} = 0$ unless $\lambda$ integrates to a character of $G_x$. In the latter case, tensoring with the inverse of this character gives
\[ \Mod(\mf{h}, G_x)_{\widetilde{\lambda}} \cong \Mod(\mf{h}, G_x)_{\widetilde{0}} = \Mod_{\mathit{nilp}}(R_x),\]
since $G_x$ is connected and acts semi-simply on any $(\mf{h}, G_x)$-module. This gives the equivalence of the proposition. Independence of $x$ follows since conjugation by $g \in G$ gives an isomorphism $G_x \cong G_{gx}$ commuting with the maps to $H$.
\end{proof}

In view of Proposition \ref{prop:loc sys classification}, we assume from now on that $G_x$ is connected and that $\lambda$ integrates to a character. In this setting, there is a unique irreducible equivariant $\lambda$-twisted local system on $\tilde{X}$ (corresponding to $\mb{C} \in \Mod_{\mathit{nilp}}(R_x)$ under the equivalence of Proposition \ref{prop:loc sys classification}). We will denote this by $\mc{O}_X(\lambda)$. We also make the following definition.

\begin{defn}
The \emph{free-monodromic local system on $\tilde{X}$} of twist $\lambda$ is the pro-object
\[ \widetilde{\mc{O}}_X(\lambda) \in \pro \Mod_{\widetilde{\lambda}}^G(\mc{D}_{\tilde X})\]
corresponding to the completion
\[ \widehat{R}_x = \varprojlim_n R_x/I_x^n \in \pro \Mod_{\mathit{nilp}}(R_x) \]
under the equivalence of Proposition \ref{prop:loc sys classification}. Here $I_x = \ker(R_x \to \mb{C})$ is the augmentation ideal.
\end{defn}

We now consider the Hodge-theoretic version of this story. First, we have the following analog of Proposition \ref{prop:loc sys classification}. In the statement, we write $\Mod_{\mrm{MHS}}(R_x)$ for the category of (finite dimensional) mixed Hodge structures equipped with an $R_x$-module structure compatible with the mixed Hodge structure on $R_x$ defined by $R_x = S(\mf{h}/\mf{h}_x(1))$. Note that such an $R_x$-module is automatically nilpotent.

\begin{prop} \label{prop:hodge local systems}
Under the assumptions of Proposition \ref{prop:loc sys classification}, assume further that $\lambda$ integrates to a character of $G_x$ and fix a pre-image $\tilde{x} \in \pi^{-1}(x)$ of $x$. Then there is an equivalence of categories
\[ \mhm^G_{\widetilde{\lambda}}(\tilde X) \cong \Mod_{\mrm{MHS}}(R_x),\]
given by taking the fiber of the mixed Hodge module on $\tilde X$ at $\tilde{x}$ with $R_x$-module structure defined by the monodromy action of $\mf{h}(1)$.
\end{prop}
\begin{proof}
As in the proof of Proposition \ref{prop:loc sys classification}, we have an equivalence
\[ \mhm^G_{\widetilde{\lambda}}(\tilde X) \cong \mhm^{G_x}_{\widetilde{\lambda}}(\pi^{-1}(x)).\]
By \cite[Lemma 2.8]{DV}, restriction to $\tilde{x}$ defines an equivalence
\[ \mhm_{\widetilde{\lambda}}(\pi^{-1}(x)) \cong \Mod_{\mrm{MHS}}(S(\mf{h}(1))).\]
Since $G_x$ is connected, $\mhm^{G_x}_{\widetilde{\lambda}}(\pi^{-1}(x))$ is the full subcategory of $\mhm_{\widetilde{\lambda}}(\tilde X)$ consisting of mixed Hodge modules whose underlying $\mc{D}$-module is $G_x$-equivariant, so the proposition now follows from Proposition \ref{prop:loc sys classification}.
\end{proof}

Of course, the equivalences of Propositions \ref{prop:hodge local systems} and \ref{prop:loc sys classification} are compatible with the forgetful functors $\mhm^G_{\widetilde{\lambda}}(\mc{D}_{\tilde X}) \to \Mod^G_{\widetilde{\lambda}}(\mc{D}_{\tilde X})$ and $\Mod_{\mrm{MHS}}(R_x) \to \Mod_{\mathit{nilp}}(R_x)$. The mixed Hodge structure on $R_x$ therefore defines a lift of $\widehat{R}_x$ to an object
\[ \widehat{R}_x = \varprojlim_n R_x/I_x^n \in \pro \Mod_{\mrm{MHS}}(R_x) \]
and thus a lift of $\widetilde{\mc{O}}_X(\lambda)$ to an object in $\pro \mhm^G_{\widetilde{\lambda}}(\tilde X)$. It is easy to see that this pro-object is good. We will endow $\widetilde{\mc{O}}_X(\lambda)$ with this Hodge structure unless otherwise specified.

\begin{rmk} \label{rmk:hodge change of base point}
While the definition of the free-monodromic local system $\widetilde{\mc{O}}_X(\lambda)$ is independent of the choice of base point, its upgrade to a mixed Hodge module is not. Indeed, if $\mc{M} \in \mhm_{\widetilde{\lambda}}^G(\tilde X)$ corresponds to 
\[ \mc{M}_{\tilde{x}} = (V, W_\bullet V, F_\bullet^H V, \bar{F}_\bullet^H V) \in \Mod_{\mrm{MHS}}(R_x) \]
under the equivalence of Proposition \ref{prop:hodge local systems}, then for $h \in \mf{h}$, the fiber of $\mc{M}$ at $\exp(2\pi i h)\tilde{x} \in \pi^{-1}(x)$ can be identified with
\[ \mc{M}_{\exp(2 \pi i h)\tilde{x}} = (V, W_\bullet V, F_\bullet^H V, \exp(4\pi \operatorname{Im} h) \bar{F}_\bullet^HV) \in \Mod_{\mrm{MHS}}(R_x).\]
This follows, for example, from the proof of \cite[Lemma 2.8]{DV}. The functor of Proposition \ref{prop:hodge local systems}, and hence the Hodge structure on $\widetilde{\mc{O}}_X(\lambda)$, therefore depends on the $G \times H^c$-orbit of $\tilde{x}$, where $H^c = \mb{X}_*(H) \otimes_{\ZZ} U(1)$ is the unique compact form of $H$.
\end{rmk}

Let us now relax the assumption that $G$ acts transitively on $X$ and let $Q \subset X$ be a $G$-orbit with connected stabilizers. Write $\tilde{Q}$ for the pre-image of $Q$ in $\tilde{X}$. Then for any choice of base point $\tilde{x} \in \tilde{Q}$, with image $x \in Q$, and $\lambda \in \mf{h}^*_\mb{R}$ integrating to a character of $G_x$, we have a free-monodromic local system $\widetilde{\mc{O}}_Q(\lambda) \in \widehat{\mhm}_{\widetilde{\lambda}}^G(\tilde Q)$ as above. Writing $j \colon \tilde Q \to \tilde X$ for the inclusion, we obtain objects
\[ j_!\widetilde{\mc{O}}_Q(\lambda), j_*\widetilde{\mc{O}}_Q(\lambda) \in \widehat{\mhm}_{\widetilde{\lambda}}^G(\tilde X).\]
The following proposition shows that while these objects depend on the choice of $\tilde{x} \in \tilde Q$, the underlying filtered $\tilde{\mc{D}}$-modules do not.

\begin{prop} \label{prop:filtration uniqueness}
Let $\widetilde{\mc{O}}_Q(\lambda)'$ be any object in $\widehat{\mhm}_{\widetilde{\lambda}}^G(\tilde Q)$ whose underlying pro-$\mc{D}$-module is $\widetilde{\mc{O}}_Q(\lambda)$, such that the unique morphism $\widetilde{\mc{O}}_Q(\lambda)'\to \mc{O}_Q(\lambda)$ is a morphism of pro-mixed Hodge modules. Then we have filtered isomorphisms
\[ (j_{!}\widetilde{\mc{O}}_Q(\lambda), F^H_\bullet) \cong (j_{!} \widetilde{\mc{O}}_Q(\lambda)', F^H_\bullet)\]
and
\[ (j_{*}\widetilde{\mc{O}}_Q(\lambda), F^H_\bullet) \cong (j_{*} \widetilde{\mc{O}}_Q(\lambda)', F^H_\bullet).\]
\end{prop}
\begin{proof}
Consider the pro-mixed Hodge structure
\[ \Hom(\widetilde{\mc{O}}_Q(\lambda), \widetilde{\mc{O}}_Q(\lambda)');\]
note that this is precisely the image of $\widetilde{\mc{O}}_Q(\lambda)'$ under the functor of Proposition \ref{prop:hodge local systems}. In particular, the underlying pro-nilpotent $R_x$-module is $\widehat{R}_x$, and we have a surjective morphism
\[ \Hom(\widetilde{\mc{O}}_Q(\lambda), \widetilde{\mc{O}}_Q(\lambda)') \to \Hom(\widetilde{\mc{O}}_Q(\lambda), \mc{O}_Q(\lambda)) = \mb{C}\]
in $\pro \Mod_{R_x}(\mrm{MHS})$. Hence, we must have
\[ \Gr^W\Hom(\widetilde{\mc{O}}_Q(\lambda), \widetilde{\mc{O}}_Q(\lambda)') = R_x\]
as mixed Hodge structures. In particular,
\[ F_0^H \Hom(\widetilde{\mc{O}}_Q(\lambda), \widetilde{\mc{O}}_Q(\lambda)') = \mb{C}\cdot f\]
for some $f$. By the same argument, we also have
\[ F_0^H \Hom(\widetilde{\mc{O}}_Q(\lambda)', \widetilde{\mc{O}}_Q(\lambda)) = \mb{C}\cdot g, \quad F_0^H\Hom(\widetilde{\mc{O}}_Q(\lambda), \widetilde{\mc{O}}_Q(\lambda)) = \mb{C} \cdot\mrm{id},\]
and
\[ F_0^H\Hom(\widetilde{\mc{O}}_Q(\lambda)',\widetilde{\mc{O}}_Q(\lambda)') = \mb{C} \cdot \mrm{id}.\]
Since $f\circ g \neq 0$ and $g \circ f \neq 0$, we can therefore arrange so that $g = f^{-1}$. Applying the functor $j_!$ (resp., $j_*$) we obtain inverse isomorphisms of pro-$\mc{D}$-modules between $j_{!}\widetilde{\mc{O}}_Q(\lambda)$ and $j_{!} \widetilde{\mc{O}}_Q(\lambda)'$ (resp., the $j_*$ versions) lying in $F_0^H$ of the relevant hom spaces. Since any morphism in $F_0^H$ respects Hodge filtrations, this implies the proposition.
\end{proof}

\subsection{Deformation and wall-crossing for pro-mixed Hodge modules} \label{subsec:semi-continuity}

In this subsection, we establish a refinement of the deformation and wall-crossing theory of \cite[\S 3]{DV} for certain pro-mixed Hodge modules. As a special case, we get a semi-continuity result for the Hodge filtrations on extensions of free-monodromic local systems.

We begin with a general definition.

\begin{definition}
Let $X$ be a smooth variety. We say that a pro-object $\mc{M} \in \pro \mhm(X)$ is \emph{fair} if the pro-objects $F^H_p \mc{M}$ and $\mc{M}/W_w\mc{M}$ are constant for all $p, w \in \mb{Z}$.
\end{definition}

In the setting of \S\ref{subsec:pro mhm}, any good pro-object is also fair. Note that any fair pro-mixed Hodge module $\mc{M} = \varprojlim_n \mc{M}_n$ has an underlying filtered $\mc{D}$-module
\[ (\mc{M}, F^H_\bullet) = \left(\bigcup_p \varprojlim_n F^H_p\mc{M}_n, \varprojlim_n F^H_\bullet \mc{M}_n\right).\]
The underlying filtered $\mc{D}$-module need not be coherent in general.

Now suppose that $\pi \colon \tilde{X} \to X$ is an $H$-torsor (for some torus $H$) and $j \colon Q \hookrightarrow X$ is the inclusion of an affinely embedded locally closed subvariety. Attached to this setup, we have the group of semi-invariant functions
\[ \Gamma(\tilde{Q}, \mc{O}^\times)^{\mathit{mon}} = \coprod_{\mu \in \mb{X}^*(H)} \Gamma(\tilde{Q}, \mc{O}^\times)_\mu,\]
where $\tilde{Q} = \pi^{-1}(Q)$ and $\Gamma(\tilde{Q}, \mc{O}^\times)_\mu$ is the set of functions $f \colon \tilde{Q} \to \mb{C}^\times$ such that $h \cdot f = \mu(h) f$ for $h \in H$. Regarding $\Gamma(\tilde{Q}, \mc{O}^\times)^{\mathit{mon}}$ as an abelian group under multiplication (and hence a $\mb{Z}$-module), we set
\[ \Gamma_\mb{R}(\tilde{Q})^{\mathit{mon}} = \frac{\Gamma(\tilde{Q}, \mc{O}^\times)^{\mathit{mon}}}{\mb{C}^\times} \otimes_\mb{Z} \mb{R}.\]
We will write elements in the vector space $\Gamma_\mb{R}(\tilde{Q})^{\mathit{mon}}$ multiplicatively; e.g., for $f \in \Gamma(\tilde{Q}, \mc{O}^\times)^{\mathit{mon}}$ and $a \in \mb{R}$, we have the element $f^a \in \Gamma_\mb{R}(\tilde{Q})^{\mathit{mon}}$.

For $f \in \Gamma_\mb{R}(\tilde{Q})^{\mathit{mon}}$ and $\mc{M} \in \mhm(\tilde{Q})$, we have the associated deformation $f\mc{M} \in \mhm(\tilde{Q})$. If $\mc{M}$ is $\lambda$-monodromic for some $\lambda \in \mf{h}^*_\mb{R}$, then $f\mc{M}$ is naturally $(\lambda + \varphi(f))$-monodromic, where
\[ \varphi \colon \Gamma_\mb{R}(\tilde{Q})^{\mathit{mon}} \to \mf{h}^*_\mb{R} \]
is the linear map sending $f \in \Gamma(\tilde{Q}, \mc{O}^\times)_\mu$ to $\mu$. This construction extends directly to arbitrary pro-objects in $\mhm(\tilde{Q})$. If we choose a pre-image of $f$ in $\Gamma(\tilde{Q}, \mc{O}^\times)^{\mathit{mon}} \otimes \mb{R}$, we also have the associated pro-object
\[ f^s\mc{M}[[s]] = \varprojlim_n \frac{f^s\mc{M}[s]}{(s^n)} \in \pro \mhm(\tilde{Q}).\]
Note that the mixed Hodge structure on $f^s\mc{M}[s]/(s^n)$ and hence on $f^s\mc{M}[[s]]$ changes if we multiply $f$ by a scalar of norm $\neq 1$; this is essentially the same phenomenon as in Remark \ref{rmk:hodge change of base point}. The mixed Hodge structure is defined so that multiplication by $s$ corresponds to a Tate twist $(1)$: in particular, if $\mc{M}$ is a fair pro-mixed Hodge module, then the pro-mixed Hodge module $f^s\mc{M}[[s]]$ is also fair.

Finally, recall from \cite[Definition 3.1]{DV} the cone of positive elements $\Gamma_\mb{R}(\tilde{Q})^{\mathit{mon}}_+ \subset \Gamma_\mb{R}(\tilde{Q})^{\mathit{mon}}$ spanned by the boundary equations for $Q$.

\begin{theorem} \label{thm:pro semi-continuity}
Let $f \in \Gamma_\mb{R}(\tilde{Q})^{\mathit{mon}}_+$ be positive. Assume that $\mc{M} \in \pro \mhm(\tilde{Q})$ is a fair pro-(monodromic)-mixed Hodge module such that the natural surjection
\[ f^s\mc{M}[[s]] \to \mc{M} \]
splits as a morphism of pro-$\mc{D}$-modules. Then:
\begin{enumerate}
\item \label{itm:pro semi-continuity 1} We have isomorphisms $(f^\alpha \mc{M}, F_\bullet^H) \cong (\mc{M}, F_\bullet^H)$ of filtered $\mc{D}_{\tilde{Q}}$-modules for every $\alpha \in \mb{R}$. If $\mc{M}$ is monodromic, then these descend to isomorphisms of $\tilde{\mc{D}}_Q$-modules on $Q$.
\item \label{itm:pro semi-continuity 2} The canonical filtered morphisms
\begin{equation} \label{eq:pro semi-continuity 2}
j_!f^\alpha\mc{M} \to j_*f^\alpha \mc{M} \to j_+f^\alpha \mc{M} = j_+\mc{M}
\end{equation}
are injective and strict with respect to the Hodge filtrations. Here $j_+$ is Laumon's `naive' pushforward for filtered $\mc{D}$-modules (see, e.g., \cite[\S 3.2]{DV}).
\item \label{itm:pro semi-continuity 3} If $\alpha < \beta$, then
\[ j_*f^\alpha \mc{M} \subset j_!f^\beta \mc{M} \]
as submodules of $j_+\mc{M}$, with equality if $0 < \beta - \alpha \ll 1$.
\end{enumerate}
\end{theorem}
\begin{proof}
Since the morphism splits, we have that the identity $\mrm{id}_{\mc{M}}$ lies in the image of the $\mc{D}$-module hom
\[ \Hom(\mc{M}, f^s\mc{M}[[s]]) \to \Hom(\mc{M}, \mc{M}).\]
Since $\mrm{id}_{\mc{M}} \in F_0^H W_0 \Hom(\mc{M}, \mc{M})$, we conclude from Lemma \ref{lem:fair pro hom 2} below that there exists a morphism
\[ g \in F^H_0 W_0 \Hom(\mc{M}, f^s\mc{M}[[s]]) \]
splitting $f^s\mc{M}[[s]] \to \mc{M}$. Extending linearly in $s$, we obtain an isomorphism
\[ \tilde{g} \in F^H_0 W_0 \Hom(\mc{M}[[s]], f^s\mc{M}[[s]])\]
whose inverse is also in $F_0^HW_0$. Indeed, the inverse of $\tilde{g}$ is
\[ \tilde{g}^{-1} = \sum_{j \geq 0} c_j \prod_{k = 0}^{j - 1}(1 - \tilde{g}c_k) \in F^H_0W_0 \Hom(f^s\mc{M}[[s]], \mc{M}[[s]])\]
where we write $c_j$ for the `leading order term' morphism
\[ c_j \colon s^j f^s\mc{M}[[s]] \to s^j\mc{M} \to s^j \mc{M}[[s]].\]
In particular, $\tilde{g}$ induces an isomorphism between the underlying filtered $\mc{D}_{\tilde Q}$-modules
\[ (\mc{M}[s], F^H_\bullet) \cong (f^s\mc{M}[s], F^H_\bullet).\]
Taking the cokernel of $s - \alpha$, we deduce the desired filtered $\mc{D}_{\tilde Q}$-module isomorphism
\[ (\mc{M}, F^H_\bullet) \cong (f^\alpha\mc{M}, F^H_\bullet),\]
proving \eqref{itm:pro semi-continuity 1}.

To prove \eqref{itm:pro semi-continuity 2}, recall that by \cite[Lemma 3.5]{DV}, the morphisms
\begin{equation} \label{eq:pro semi-continuity 3}
j_!^{(\alpha)} f^s\mc{M}[s] \to j_*^{(\alpha)} f^s\mc{M}[s] \to j_+f^s\mc{M}[s]
\end{equation}
are injective and strict. Here $j_!^{(\alpha)}f^s\mc{M}[s]$ and $j_*^{(\alpha)}f^s\mc{M}[s]$ denote the filtered $\mc{D}$-modules underlying
\[ j_!f^s\mc{M}[[s - \alpha]] \cong j_!f^{s + \alpha} \mc{M}[[s]] \quad \text{and} \quad j_*f^s\mc{M}[[s - \alpha]] \cong j_*f^{s + \alpha}\mc{M}[[s]]\]
respectively. Now, the morphism $f^\alpha \tilde{g}$ defines filtered isomorphisms
\begin{equation} \label{eq:pro semi-continuity 1}
j_!^{(\alpha)}f^s\mc{M}[s] \cong j_!f^\alpha \mc{M} \otimes \mb{C}[s] \quad \text{and} \quad j_*^{(\alpha)} f^s\mc{M}[s]\cong j_*f^\alpha \mc{M} \otimes \mb{C}[s]
\end{equation}
So \eqref{eq:pro semi-continuity 2} is a direct summand of \eqref{eq:pro semi-continuity 3} and we deduce \eqref{itm:pro semi-continuity 2}.

Finally, to prove \eqref{itm:pro semi-continuity 3}, note that by \cite[Lemma 3.5]{DV} we have the asserted relation between $j_*^{(\alpha)} f^s\mc{M}[s]$ and $j_!^{(\beta)}f^s\mc{M}[s]$. By \eqref{eq:pro semi-continuity 1}, this implies the relation between $j_*f^\alpha\mc{M}$ and $j_!f^\beta \mc{M}$, so we are done.
\end{proof}

\begin{lemma} \label{lem:fair pro hom 1}
Let $\mc{M} = \varprojlim_n \mc{M}_n \in \pro \mhm(\tilde{Q})$ be a fair pro-object and $\mc{N} \in \mhm(\tilde{Q})$ be a mixed Hodge module. Then, for every weight $w \in \mb{Z}$, the ind-mixed Hodge structure
\[ W_w\Hom(\mc{M}, \mc{N}) := \varinjlim_n W_w\Hom(\mc{M}_n, \mc{N}) \in \operatorname{Ind}\mrm{MHS}\]
is constant, i.e., $W_w\Hom(\mc{M}, \mc{N})$ is a mixed Hodge structure.
\end{lemma}
\begin{proof}
Pick $w_0 \in \mb{Z}$ such that $W_{w_0}\mc{N} = 0$. Then
\[ W_w \Hom(\mc{M}_n, \mc{N}) = W_w \Hom(\mc{M}_n/W_{w_0 - w}\mc{M}_n, \mc{N}).\]
Since $\mc{M}$ is fair, the pro-object $\varprojlim_n \mc{M}_n/W_{w_0 - w}\mc{M}_n$ is constant. So
\[ \Hom(\mc{M}, \mc{N}) = \varinjlim_n W_w\Hom(\mc{M}_n/W_{w_0 - w}\mc{M}_n, \mc{N}) \]
is also constant.
\end{proof}

\begin{lemma} \label{lem:fair pro hom 2}
Let $\mc{M}, \mc{N}, \mc{P} \in \pro \mhm(\tilde{Q})$ and let $\mc{N} \to \mc{P}$ be a morphism. Then, the morphism
\[ \phi \colon \Hom(\mc{M}, \mc{N}) \to \Hom(\mc{M}, \mc{P}) \]
of pro-ind-mixed Hodge structures is strict with respect to the Hodge and weight filtrations. That is, writing $\mc{M}= \varprojlim_m \mc{M}_m$, $\mc{N} = \varprojlim_n \mc{N}_n$ and $\mc{P} = \varprojlim_n \mc{P}_n$, for each $p$ and each $w$, the morphism
\[ \varprojlim_n F^H_p W_w\Hom(\mc{M}, \mc{N}_n) \to \varprojlim_n F^H_p W_w\Hom(\mc{M}, \mc{P}_n) \cap \phi(\varprojlim_n \Hom(\mc{M}, \mc{N}_n)) \]
is surjective.
\end{lemma}
\begin{proof}
First note that, by passing to subsequences if necessary, we can assume without loss of generality that $\mc{N} \to \mc{P}$ is given by a morphism of inverse systems $\mc{N}_n \to \mc{P}_n$. Now, suppose $(f_n)$ is sequence defining an element
\[ (f_n) \in \varprojlim_n F^H_p W_w\Hom(\mc{M}, \mc{P}_n) \cap \phi(\varprojlim_n \Hom(\mc{M}, \mc{N}_n)).\]
Then, for each $n$, there exists $m$ such that
\[ f_n \in F^H_p W_w \Hom(\mc{M}_m, \mc{P}_n) \cap \phi(\Hom(\mc{M}_m, \mc{N}_n)) = \phi(F^H_p W_w\Hom(\mc{M}_m, \mc{N}_n))\]
by strictness. By Lemma \ref{lem:fair pro hom 1}, we have $W_w \Hom(\mc{M}_m, \mc{N}_n) = W_w\Hom(\mc{M}, \mc{N}_n)$ for $m$ large enough, so we can write
\[ f_n \in \phi(F^H_p W_w\Hom(\mc{M}, \mc{N}_n))\]
for all $n$. Now, let us write $\phi_n$ for the morphism
\[ \phi_n = \phi \colon F^H_pW_w\Hom(\mc{M}, \mc{N}_n) \to F^H_p W_w\Hom(\mc{M}, \mc{P}_n).\]
For all $m > n$, the image of $\phi_m^{-1}(f_m)$ in $F^H_pW_w\Hom(\mc{M}, \mc{N}_n)$ is a torsor under a subspace of the finite dimensional vector space $F^H_pW_w\Hom(\mc{M}, \mc{N}_n)$; hence, the image stabilizes for $m \gg n$. In other words, the inverse system $\phi_n^{-1}(f_n)$ has the Mittag-Leffler property, so the inverse limit $\varprojlim_n \phi_n^{-1}(f_n)$ is non-empty. So
\[ (f_n) \in \phi(\varprojlim_n F^H_p W_w\Hom(\mc{M}, \mc{N}_n) )\]
as claimed.
\end{proof}

We note the following corollary for free-monodromic local systems. Assume that $\pi \colon \tilde X \to X$ is an $H$-torsor, that $G$ acts compatibly on $\tilde X$ and $X$, that $j \colon Q \hookrightarrow X$ is the inclusion of a $G$-orbit with connected stabilizers, and that $\tilde{x} \in \tilde{Q}$ is a base point with image $x = \pi(\tilde x)$. For $\lambda \in \mf{h}^*_\mb{R}$ integrating to a character of $G_x$, let $\widetilde{\mc{O}}_Q(\lambda)$ be the corresponding free-monodromic local system on $\tilde{Q}$ equipped with its usual Hodge structure defined by the base point $\tilde x$. Note that the map $\varphi$ identifies $\Gamma_\mb{R}^G(\tilde Q)^{\mathit{mon}}$ with the subspace $(\mf{h}_\mb{R}/\mf{h}_{\mb{R}, x})^* \subset \mf{h}^*_\mb{R}$.

\begin{cor} \label{cor:free-monodromic semi-continuity}
For all $\lambda \in \mf{h}^*_{\mb{R}}$ integrating to a character of $G_x$ and all $\mu \in \varphi(\Gamma_\mb{R}^G(\tilde Q)^{\mathit{mon}}_+)$, we have strict injections
\[ (j_!\widetilde{\mc{O}}_Q(\lambda), F_\bullet^H) \hookrightarrow (j_*\widetilde{\mc{O}}_Q(\lambda), F_\bullet^H) \hookrightarrow (j_!\widetilde{\mc{O}}_Q(\lambda + \mu), F_\bullet^H)\]
such that $(j_*\widetilde{\mc{O}}_Q(\lambda), F_\bullet^H) = (j_!\widetilde{\mc{O}}_Q(\lambda + \epsilon \mu), F_\bullet^H)$ for $0 < \epsilon \ll 1$.
\end{cor}
\begin{proof}
Let $f \in \Gamma_\mb{R}^G(\tilde Q)^{\mathit{mon}}_+$ be the unique element with $\varphi(f) = \mu$. Then
\[ \widetilde{\mc{O}}_{Q}(\lambda + \mu) = f\widetilde{\mc{O}}_{Q}(\lambda).\]
Since $\widetilde{\mc{O}}_{Q}(\lambda)$ is a projective object in $\pro \Mod_{\widetilde{\lambda}}^G(\tilde Q)$, the surjective morphism
\[ f^s\widetilde{\mc{O}}_{Q}(\lambda)[[s]] \to \widetilde{\mc{O}}_{Q}(\lambda) \]
must split as a morphism of pro-$\mc{D}$-modules. Applying Theorem \ref{thm:pro semi-continuity} now gives the result.
\end{proof}

\section{Intertwining functors and the Hodge filtration} \label{sec:intertwining}

In this section, we write down the theory of intertwining functors for Hodge modules. In \S\ref{subsec:localization}, we recall the basic statement of Beilinson-Bernstein localization. In \S\ref{subsec:intertwining construction}, we study the completed monodromic Hecke category and its convolution action on monodromic $\mc{D}$-modules and mixed Hodge modules on the flag variety, and use this to construct the Hodge and filtered intertwining functors. In \S\ref{subsec:affine hecke}, we use the results of \S\ref{subsec:semi-continuity} to deduce that the intertwining functors and line bundle twists satisfy the Bernstein relations for the affine Hecke algebra. Finally, we show that the filtered intertwining functors commute with derived global sections in \S\ref{subsec:intertwining}.

\subsection{Recollections on localization} \label{subsec:localization}

In this subsection, we record our notation and conventions for reductive groups and briefly recall the main constructions of Beilinson-Bernstein localization.

We fix from now on a connected complex reductive group $G$ with Lie algebra $\mf{g}$. We write $\mc{B}$ for the flag variety (the variety of Borel subgroups of $G$, or equivalently, Borel subalgebras of $\mf{g}$), $H$ for the abstract Cartan, and $\mf{h} = \mrm{Lie}(H)$. If we fix a maximal torus and Borel subgroup $T \subset B \subset G$, then we have natural identifications $\mc{B} \cong G/B$ and $H \cong B/N \cong T$, where $N$ is the unipotent radical of $B$. We thus obtain roots and coroots
\[ \Phi \subset \mb{X}^*(H) := \Hom(H, \mb{C}^\times) \quad \text{and} \quad \check\Phi \subset \mb{X}_*(H) := \Hom(\mb{C}^\times, H)\]
and a Weyl group $W = N_G(T)/T$ acting on $T = H$. We write $\Phi_+, \Phi_- \subset \Phi$ for the corresponding sets of positive and negative roots; here we adopt the convention that the negative roots $\Phi_-$ are the weights of $T$ acting on $\mrm{Lie}(N)$ and $\Phi_+ = -\Phi_-$. The based root datum constructed in this way is canonically independent of the choice of $T$ and $B$.

The flag variety $\mc{B}$ comes equipped with a canonical $G$-equivariant $H$-torsor $\tilde{\mc{B}}$ ($=G/N$ for any choice of Borel as above), called the \emph{base affine space}, and thus a sheaf of rings $\tilde{\mc{D}}$ as in \S\ref{subsec:twisted mhm}. For $\lambda \in \mf{h}^*$, we set
\[ \mc{D}_\lambda = \tilde{\mc{D}} \otimes_{S(\mf{h}), \lambda - \rho} \mb{C},\]
where $\rho = \frac{1}{2}\sum_{\alpha \in \Phi_+} \alpha$. With these conventions, the line bundle $\mc{O}(\mu) = \tilde{\mc{B}} \times^H \mb{C}_\mu$ is a $\mc{D}_{\mu + \rho}$-module for $\mu \in \mb{X}^*(H)$. Our convention for positive roots is such that $\mc{O}(\mu)$ is ample if and only if $\mu$ is regular dominant. Note that
\[ \Mod(\mc{D}_\lambda) = \Mod_{\lambda - \rho}(\mc{D}_{\tilde{\mc{B}}}).\]
By analogy, we write
\[ \Mod(\mc{D}_{\widetilde{\lambda}}) = \Mod_{\widetilde{\lambda - \rho}}(\mc{D}_{\widetilde{\mc{B}}}), \quad \mhm(\mc{D}_{\widetilde{\lambda}}) = \mhm_{\widetilde{\lambda - \rho}}(\tilde{\mc{B}}), \quad \text{etc.}\]

The starting point for Beilinson-Bernstein localization \cite{BB1} is that the $G$-action on $\tilde{\mc{B}}$ defines an algebra homomorphism $U(\mf{g}) \to \tilde{\mc{D}}$. It is shown in \cite{BB1} that for all $\lambda \in \mf{h}^*$, the composition $\mf{Z}(\mf{g}) \to U(\mf{g}) \to \tilde{\mc{D}} \to \mc{D}_\lambda$ factors through the infinitesimal character $\chi_\lambda$ associated to $\lambda$ via the Harish-Chandra isomorphism. Thus, we have a derived functor
\begin{equation} \label{eq:big bb functor}
\mrm{R}\Gamma \colon \mrm{D}^b\coh(\tilde{\mc{D}}) \to \mrm{D}^b \Mod_{fg}(U(\mf{g}))
\end{equation}
restricting to
\begin{equation} \label{eq:small bb functor}
\mrm{R}\Gamma \colon \mrm{D}^b\coh(\mc{D}_{\widetilde{\lambda}}) \to \mrm{D}^b\Mod_{fg}(U(\mf{g}))_{\widetilde{\chi_\lambda}}.
\end{equation}
The main theorem of the subject is that the latter functor is t-exact (i.e., it restricts to an exact functor between the corresponding abelian categories) when $\lambda$ is integrally dominant and an equivalence when $\lambda$ is regular.

Dually, one also has the \emph{localization functor}
\begin{align*}
\mrm{L}\Delta_{\widetilde{\lambda}} \colon \mrm{D}^b \Mod_{fg}(U(\mf{g}))_{\widetilde{\chi_\lambda}} &\to \mrm{D}^b \coh(\mc{D}_{\widetilde{\lambda}}), \\
M &\mapsto (\tilde{\mc{D}} \overset{\mrm{L}}\otimes_{U(\mf{g})} M)_{\widetilde{\lambda - \rho}},
\end{align*}
where $(-)_{\widetilde{\lambda - \rho}}$ denotes the projection to the $(\lambda - \rho)$-generalized eigenspace of $S(\mf{h}) \subset \tilde{\mc{D}}$. The functor $\mrm{L}\Delta_{\widetilde{\lambda}}$ is left adjoint to \eqref{eq:small bb functor} by construction, and right adjoint to \eqref{eq:small bb functor} by Serre duality. In particular, just as for $\mrm{R}\Gamma$, $\mrm{L}\Delta_{\widetilde{\lambda}}$ is t-exact whenever $\lambda$ is integrally dominant: note that this is \emph{not} the case in the twisted (rather than monodromic) version of the theory, see, e.g., \cite[Theorem 2.1]{HMSW2}. 

\subsection{The Hodge and filtered intertwining functors} \label{subsec:intertwining construction}

We now consider a filtered version of the above story. By construction, the homomorphism $U(\mf{g}) \to \tilde{\mc{D}}$ is compatible with the PBW filtration $F_\bullet U(\mf{g})$ and the order filtration $F_\bullet \tilde{\mc{D}}$. We may therefore lift \eqref{eq:big bb functor} to a functor
\[ \mrm{R}\Gamma \colon \mrm{D}^b\coh(\tilde{\mc{D}}, F_\bullet) \to \mrm{D}^b\Mod_{fg}(U(\mf{g}), F_\bullet)\]
from the filtered derived category of coherent $\tilde{\mc{D}}$-modules with good filtration to the filtered derived category of finitely generated $U(\mf{g})$-modules with good filtration. This has a left adjoint $\mrm{L}\Delta$ defined by
\[ \mrm{L}\Delta(M, F_\bullet) = (\widetilde{\mc{D}}, F_\bullet) \overset{\mrm{L}}\otimes_{(U(\mf{g}), F_\bullet)} (M, F_\bullet).\]

In \cite[\S 5-7]{DV}, Vilonen and the first named author studied the composition
\begin{equation} \label{eq:hodge bb functor}
 \mrm{D}^b\mhm(\mc{D}_{\widetilde{\lambda}}) \to \mrm{D}^b\coh(\tilde{\mc{D}}, F_\bullet) \to \mrm{D}^b\Mod_{fg}(U(\mf{g}), F_\bullet).
\end{equation}
and proved some analogs of the localization theorems in this setting. For example, it was shown that \eqref{eq:hodge bb functor} is filtered t-exact (i.e., the composition with $\Gr^F$ is t-exact) whenever $\lambda \in \mf{h}^*_\mb{R}$ is dominant \cite[Theorem 5.2]{DV}. Note that the original proof of Beilinson and Bernstein's theorem, which works by tensoring with finite dimensional representations and splitting via the action of the center, does not generalize to the filtered or Hodge settings. Instead, the localization theorems in \cite{DV} are proved by studying the functor
\[ \mrm{L}\Delta \circ \mrm{R}\Gamma \colon \mrm{D}^b\coh(\tilde{\mc{D}}, F_\bullet) \to \mrm{D}^b\coh(\tilde{\mc{D}}, F_\bullet)\]
and relating it to well-understood convolution functors on the categories of mixed Hodge modules.

Our approach to intertwining functors works in a similar way. We begin in this subsection by writing down appropriate categories of convolution functors on Hodge and filtered modules, including the Hodge and filtered intertwining functors in particular. We then study how these interact with $\mrm{L}\Delta \circ \mrm{R}\Gamma$ and with each other in subsequent subsections, and use this to deduce the results stated in the introduction.

Consider the product $\tilde{\mc{B}} \times \tilde{\mc{B}}$, regarded as an $H \times H$-torsor over $\mc{B} \times \mc{B}$ equivariant under the diagonal action of $G$. Applying the constructions of \S\ref{subsec:pro mhm}, we obtain for all $\lambda, \mu \in \mf{h}^*_\mb{R}$ a category
\[ \widehat{\mrm{D}}^b_G\mhm(\mc{D}_{\widetilde{\mu}} \boxtimes \mc{D}_{\widetilde{-\lambda}}) := \widehat{\mrm{D}}^b_G\mhm_{\widetilde{\mu - \rho}, \widetilde{-\lambda - \rho}}(\tilde{\mc{B}} \times \tilde{\mc{B}}) \]
and a functor
\begin{align*}
\widehat{\mrm{D}}^b_G\mhm(\mc{D}_{\widetilde{\mu}} \boxtimes \mc{D}_{\widetilde{-\lambda}}) &\to \mrm{D}^b_{\mrm{St}}\coh^G(\tilde{\mc{D}} \boxtimes \tilde{\mc{D}}, F_\bullet) \\
\mc{K} &\mapsto (\mc{K}, F_\bullet^H),
\end{align*}
where $\mrm{St} \cong \tilde{\mf{g}}^* \times_{\mf{g}^*} \tilde{\mf{g}}^*$ is the fiber over $0$ of the moment map
\[ \tilde{\mf{g}}^* \times \tilde{\mf{g}}^* = T^*\tilde{\mc{B}}/H \times T^*\tilde{\mc{B}}/H \to \mf{g}^*.\]

Let us describe some objects in these categories. Recall that $G$ acts on $\mc{B} \times \mc{B}$ with finitely many orbits, indexed naturally by the (abstract) Weyl group $W$. To fix notation, for $w \in W$, we choose a base point $x \in \mc{B}$ and a maximal torus $T$ fixing $x$ (and hence identified with $H$) and set
\[ X_w = G \cdot (x, \tilde{w}x) \]
where $\tilde{w} \in N_G(T)$ is a lift of $w \in W = N_G(T)/T$. The orbit $X_w$ is independent of the choices of $x$, $T$ and $\tilde{w}$. The stabilizer of $(x, \tilde{w}x)$ is an extension of $T$ by a unipotent group, hence connected, so we are in the situation of \S\ref{subsec:free-monodromic}. The associated homomorphism $T \to H \times H$ is $t \mapsto (t, w^{-1}(t))$. So by Proposition \ref{prop:loc sys classification}, $X_w$ either has no equivariant twisted local systems (if $\mu - w\lambda \not\in\mb{X}^*(H)$) or else the category of equivariant twisted local systems on $X_w$ is equivalent to that of nilpotent modules over the ring
\[ R_w := R_{(x, \tilde{w}x)} = \frac{S(\mf{h}) \otimes S(\mf{h})}{(h \otimes 1 + 1 \otimes w^{-1}(h) \mid h \in \mf{h})} \cong S(\mf{h}).\]
Assuming $\mu - w\lambda \in \mb{X}^*(H)$, we thus have an irreducible (resp., free-monodromic) twisted local system $\mc{O}_{X_w}(\mu - \rho, -\lambda - \rho)$ and (resp., $\widetilde{\mc{O}}_{X_w}(\mu - \rho, -\lambda - \rho)$) on $X_w$.

As in \S\ref{subsec:free-monodromic}, to fix a Hodge module structure on $\widetilde{\mc{O}}_{X_w}(\mu - \rho, -\lambda - \rho)$, we need to choose a base point in the pre-image $\tilde{X}_w$ of $X_w$ in $\tilde{\mc{B}} \times \tilde{\mc{B}}$. It will be helpful to make our choices compatibly for different $w$ as follows. Fix a compact form $U_\mb{R}$ of $G$, a base point $x \in \mc{B}$ and a point $\tilde{x} \in \pi^{-1}(x) \subset \tilde{\mc{B}}$. Let $T^c = \mrm{Stab}_{U_\mb{R}}(x)$ (a maximal torus in $U_\mb{R}$) and $\tilde{W} = N_{U_\mb{R}}(T^c)$; the group $\tilde{W}$ is an extension of the Weyl group $W$ by the compact torus $T^c$. Choose $\tilde{w} \in \tilde{W}$ and define the Hodge structure on $\widetilde{\mc{O}}_{X_w}(\mu - \rho, -\lambda - \rho)$ using the base point $(\tilde{x}, \tilde{w}\tilde{x})$. The $G \times H^c \times H^c$-orbit of this point, and hence the associated Hodge structure (Remark \ref{rmk:hodge change of base point}), depends only on $\tilde{x}$ and $U_\mb{R}$.

The following objects will play an important role.

\begin{notation}
For $\lambda \in \mf{h}^*_\mb{R}$, $\nu \in \mb{X}^*(H)$ and $w \in W$, we set
\[ \tilde{\Delta}_w^{(\lambda)} = j_{w!}\widetilde{\mc{O}}_{X_w}(w\lambda - \rho, -\lambda - \rho) (-\dim \mc{B}) \in \widehat{\mhm}(\mc{D}_{\widetilde{w\lambda}} \boxtimes \mc{D}_{\widetilde{-\lambda}}, G),\]
\[ \tilde{\nabla}_w^{(\lambda)} = j_{w*}\widetilde{\mc{O}}_{X_w}(w\lambda - \rho, -\lambda - \rho)(\ell(w) -\dim \mc{B}) \in \widehat{\mhm}(\mc{D}_{\widetilde{w\lambda}} \boxtimes \mc{D}_{\widetilde{-\lambda}}, G),\]
and
\begin{align*}
\tilde{\delta}_\nu^{(\lambda)} &= j_{1!}\widetilde{\mc{O}}_{X_1}(\lambda + \nu - \rho, -\lambda - \rho) (-\dim\mc{B}) \\
&= j_{1*}\widetilde{\mc{O}}_{X_1}(\lambda + \nu - \rho, -\lambda - \rho) (-\dim\mc{B}) \in \widehat{\mhm}(\mc{D}_{\widetilde{\lambda + \nu}} \boxtimes \mc{D}_{\widetilde{-\lambda}}, G),
\end{align*}
where $j_w \colon X_w \to \mc{B} \times \mc{B}$ is the inclusion. We will occasionally use the same notation without the tilde for the objects defined with the irreducible local system $\mc{O}_{X_w}$ in place of the free-monodromic local system $\widetilde{\mc{O}}_{X_w}$.
\end{notation}

Note that while the Hodge structures on the above objects depend on our choices of $U_\mb{R}$ and $\tilde{x}$, by Proposition \ref{prop:filtration uniqueness}, the underlying filtered $\tilde{\mc{D}} \boxtimes \tilde{\mc{D}}$-modules do not. Moreover, by the results of \S\ref{subsec:semi-continuity}, we have the following relations between these filtered modules as $\lambda$ varies. In the statement below, let $\mf{h}^*_{\mb{R}, +} \subset \mf{h}^*_\mb{R}$ denote the set of regular dominant elements.

\begin{cor} \label{cor:intertwining semi-continuity}
For all $\lambda \in \mf{h}^*_\mb{R}$ and $\mu \in w^{-1}\mf{h}^*_{\mb{R}, +} - \mf{h}^*_{\mb{R}, +}$, we have strict injections
\[ (\tilde{\Delta}_w^{(\lambda)}, F_\bullet^H) \hookrightarrow (\tilde{\nabla}_w^{(\lambda)}, F_\bullet^H)\{-\ell(w)\} \hookrightarrow (\tilde{\Delta}_w^{(\lambda + \mu)}, F_\bullet^H)\]
such that $(\tilde{\nabla}_w^{(\lambda)}, F_\bullet^H)\{-\ell(w)\} = (\tilde{\Delta}_w^{(\lambda + \epsilon \mu)}, F_\bullet^H)$ for $0 < \epsilon \ll 1$.
\end{cor}
\begin{proof}
By \cite[Proposition 10.3]{DV} applied to the symmetric pair $(G \times G, G)$, the map
\[ \varphi \colon \Gamma_\mb{R}^G(\tilde{X}_w)_+^{\mathit{mon}} \to \mf{h}^*_\mb{R} \oplus \mf{h}^*_\mb{R} \]
is injective, with image
\[ \{(\nu_1 - w\nu_2, \nu_2 - w^{-1}\nu_1) \mid \nu_1, \nu_2 \in \mf{h}^*_{\mb{R}, +}\}.\]
The result now follows from the definitions and Corollary \ref{cor:free-monodromic semi-continuity}.
\end{proof}

\begin{rmk} \label{rmk:cone}
It is easy to see that, for $w = 1$ (resp., a simple reflection $s_\alpha$), the cone $w^{-1}\mf{h}^*_{\mb{R}, +} - \mf{h}^*_{\mb{R}, +}$ appearing in Corollary \ref{cor:intertwining semi-continuity} is equal to $\mf{h}^*_\mb{R}$ (resp., $\{\mu \in \mf{h}^*_\mb{R} \mid \langle \mu, \check\alpha\rangle < 0\}$). More generally, we have
\[ w^{-1}\mf{h}^*_{\mb{R}, +} - \mf{h}^*_{\mb{R}, +} = \{\mu \in \mf{h}^*_\mb{R} \mid \text{if $\alpha \in \Phi_+$ and $w\alpha \in \Phi_-$ then $\langle \mu, \check\alpha \rangle < 0$}\},\]
although we will not use this fact.
\end{rmk}

Similarly, we have:

\begin{cor} \label{cor:translation independence}
We have
\[ (\tilde{\Delta}_1^{(\lambda)}, F^H_\bullet) \cong (\tilde{\Delta}_1^{(0)}, F^H_\bullet) \quad \text{and} \quad (\tilde{\delta}_\nu^{(\lambda)}, F^H_\bullet) \cong (\tilde{\delta}_\nu^{(0)}, F_\bullet^H) \]
for all $\lambda \in \mf{h}^*_\mb{R}$ and all $\nu \in \mb{X}^*(H)$.
\end{cor}

In view of Corollary \ref{cor:translation independence}, we will sometimes write
\[ (\tilde{\delta}_\nu, F_\bullet^H) = (\tilde{\delta}^{(\lambda)}_\nu, F_\bullet^H).\]

Now, our main interest in the categories $\widehat{\mrm{D}}^b_G\mhm(\mc{D}_{\widetilde{\mu}} \boxtimes \mc{D}_{\widetilde{-\lambda}})$ is that they act on the derived categories of monodromic mixed Hodge modules on $\tilde{\mc{B}}$. More precisely, we have a \emph{convolution functor}
\[ * \colon \widehat{\mrm{D}}^b_G \mhm(\mc{D}_{\widetilde{\mu}} \boxtimes \mc{D}_{\widetilde{-\lambda}}) \times \mrm{D}^b \mhm(\mc{D}_{\widetilde{\lambda}}) \to \mrm{D}^b \mhm(\mc{D}_{\widetilde{\mu}}) \]
given as follows. For
\[ \mc{K} \in \widehat{\mrm{D}}^b_G\mhm(\mc{D}_{\widetilde{\mu}} \boxtimes \mc{D}_{\widetilde{-\lambda}}) \quad \text{and} \quad \mc{M} \in \mrm{D}^b\mhm(\mc{D}_{\widetilde{\lambda}}),\]
we regard $\mc{K}$ and $\mc{M}$ as (pro-)complexes of mixed Hodge modules on $\tilde{\mc{B}} \times \tilde{\mc{B}}$ and $\tilde{\mc{B}}$ respectively and set
\begin{equation} \label{eq:mhm convolution}
\mc{K} * \mc{M} = \widetilde{\mrm{pr}}_{1*}\tilde{D}_{23}^\circ(\mc{K} \boxtimes \mc{M})(\dim \tilde{\mc{B}}) \in \pro \mrm{D}^b\mhm(\tilde{\mc{B}})
\end{equation}
where $\widetilde{\mrm{pr}}_1$ and $\tilde{D}_{23}$ are the morphisms
\begin{align*}
\tilde{\mc{B}} \times \tilde{\mc{B}} \times \tilde{\mc{B}} \xleftarrow{\tilde{D}_{23}} & \tilde{\mc{B}} \times \tilde{\mc{B}} \xrightarrow{\widetilde{\mrm{pr}}_1} \tilde{\mc{B}} \\
(x, y, y) \mapsfrom & (x, y) \mapsto x
\end{align*}
and $\tilde{D}_{23}^\circ = \tilde{D}_{23}^*[-\dim \tilde{\mc{B}}]$ is the intermediate pullback.

\begin{prop} \label{prop:good convolution}
We have
\[ \mc{K} * \mc{M} \in \mrm{D}^b\mhm(\mc{D}_{\widetilde{\mu}}) \subset \pro \mrm{D}^b\mhm(\tilde{\mc{B}}).\]
\end{prop}
\begin{proof}
Since $\mc{K}$ is good, we have by definition that the object
\begin{equation} \label{eq:good convolution 1}
\mc{K} \overset{\mrm{L}}\otimes_{S(\mf{h}(1)) \otimes S(\mf{h}(1))}  \mb{C} \in \pro \mrm{D}^b_G\mhm(\mc{D}_{\widetilde{\mu}} \boxtimes \mc{D}_{\widetilde{-\lambda}})
\end{equation}
is constant. We claim that in fact
\begin{equation} \label{eq:good convolution 2}
 \mc{K} \overset{\mrm{L}}\otimes_{S(\mf{h}(1))} \mb{C}
\end{equation}
is constant as well, where $S(\mf{h}(1))$ acts via the monodromy on (say) the second copy of $\tilde{\mc{B}}$. To see this, it suffices to consider the case where $\mc{K} = j_{w*}\mc{K}'$ for some $w \in W$ and some $\mc{K}' \in \widehat{\mhm}_{\widetilde{\mu - \rho}, \widetilde{-\lambda - \rho}}(\tilde X_w)$. Since the quotient map $S(\mf{h}(1)) \to R_w$ is an isomorphism, it follows from the equivalence of Proposition \ref{prop:hodge local systems} that
\[ \mc{K}' \overset{\mrm{L}} \otimes_{S(\mf{h}(1)) \otimes S(\mf{h}(1))} \mb{C} \cong \bigoplus_n \wedge^n\mf{h}(1)[n] \otimes \mc{K}' \overset{\mrm{L}}\otimes_{S(\mf{h}(1))} \mb{C}.\]
So \eqref{eq:good convolution 2} is a direct summand of \eqref{eq:good convolution 1}, hence constant as claimed.

Now, by \cite[Proposition 7.4]{DV}, the constancy of \eqref{eq:good convolution 2} implies that the cohomology sheaves $\mc{H}^i(\mc{K} *\mc{M})$ are themselves constant pro-objects for all $i$. To conclude that $\mc{K}*\mc{M}$ is itself constant, we argue as follows. First, we may write
\[ \mc{K} * \mc{M} = \varprojlim_n (\mc{K} \overset{\mrm{L}} \otimes_{S(\mf{h}(1))} S(\mf{h}(1))_n) * \mc{M},\]
where $S(\mf{h}(1))_n = S(\mf{h}(1))/I^n$ for $I$ the augmentation ideal. Since \eqref{eq:good convolution 2} is constant, each term on the right hand side is an object in
\[ \mrm{D}^{[a, b]}_G\mhm(\mc{D}_{\widetilde{\mu}} \boxtimes \mc{D}_{\widetilde{-\lambda}}) \]
for some $a, b \in \mb{Z}$ independent of $n$. We now conclude by inductively applying the following lemma to the (inverse limits of) exact triangles
\[ \tau^{< i}(\mc{K} * \mc{M}) \to \tau^{\leq i}(\mc{K} * \mc{M}) \to \mc{H}^i(\mc{K} * \mc{M})[-i] \to \tau^{< i}(\mc{K} * \mc{M})[1].\]
\end{proof}

\begin{lem} \label{lem:constant exact}
Let $\mc{C}$ be a triangulated category linear over a field with finite dimensional morphism spaces and let
\[ A_n \to B_n \to C_n \to A_n[1], \quad n \in \mb{Z}_{> 0}\]
be an inverse system of exact triangles in $\mc{C}$. If the pro-objects $\varprojlim_n A_n$ and $\varprojlim_n C_n$ are constant, then so is $\varprojlim_n B_n$.
\end{lem}
\begin{proof}
Write $A = \varprojlim_n A_n$, $C = \varprojlim_n C_n$ and $B = \mrm{Cone}(C[-1] \to A)$. We claim that $\varprojlim_n  B_n \cong B$. To see this, we argue as in the proof of (TR3) in \cite[Theorem A.2.2]{bezrukavnikov-yun}. Let $E_n$ denote the set of morphisms $B \to B_n$ fitting into a commutative diagram
\[
\begin{tikzcd}
A \ar[r] \ar[d] & B \ar[r] \ar[dashed, d] & C \ar[r] \ar[d] & A[1] \ar[d] \\
A_n \ar[r] & B_n \ar[r] & C_n \ar[r] & A_n[1].
\end{tikzcd}
\]
Since both triangles are exact, $E_n$ is a torsor under a subspace of the finite dimensional vector space $\Hom(B, B_n)$. Hence, the image of $E_m \to E_n$ stabilizes for $m \gg 0$. So the inverse system $\{E_n\}$ is Mittag-Leffler and hence there exists $e \in \varprojlim_n E_n$ defining a morphism $B \to \varprojlim_n B_n$ compatible with the exact triangles. For all $D \in \mc{C}$, the associated long exact sequences for $\Hom(-, D)$ imply that
\[ \Hom(\varprojlim_n B_n, D) = \varinjlim_n \Hom(B_n, D) \to \Hom(B, D) \]
is an isomorphism (the equality above being by definition of the category of pro-objects), and hence $B \to \varprojlim_n B_n$ is an isomorphism as well.
\end{proof}

Similarly, for
\[ \mc{K}_1 \in \widehat{\mrm{D}}^b_G\mhm(\mc{D}_{\widetilde{\nu}} \boxtimes \mc{D}_{\widetilde{-\mu}}) \quad \text{and} \quad \mc{K}_2 \in \widehat{\mrm{D}}^b_G \mhm(\mc{D}_{\widetilde{\mu}} \boxtimes \mc{D}_{\widetilde{-\lambda}}) \]
we have a convolution
\[ \mc{K}_1 * \mc{K}_2 = \widetilde{\mrm{pr}}_{13 *} \widetilde{D}_{23}^\circ (\mc{K}_1 \boxtimes \mc{K}_2)(\dim \tilde{\mc{B}}) \in \pro \mrm{D}^b_G \mhm(\tilde{\mc{B}} \times \tilde{\mc{B}}),\]
where $\widetilde{\mrm{pr}}_{14}$ and $\widetilde{D}_{23}$ are now the morphisms
\begin{align*}
\tilde{\mc{B}} \times \tilde{\mc{B}} \times \tilde{\mc{B}} \times \tilde{\mc{B}} \xleftarrow{\tilde{D}_{23}} \tilde{\mc{B}} \times \tilde{\mc{B}} & \times \tilde{\mc{B}} \xrightarrow{\widetilde{\mrm{pr}}_{14}} \tilde{\mc{B}} \times \tilde{\mc{B}} \\
(x, y, y, z) \mapsfrom  (x, y&, z) \mapsto (x, z).
\end{align*}
A similar argument to Proposition \ref{prop:good convolution} shows that
\[ \mc{K}_1 * \mc{K}_2 \in \widehat{\mrm{D}}^b_G\mhm(\mc{D}_{\widetilde{\nu}} \boxtimes \mc{D}_{\widetilde{-\lambda}}) \subset \pro\mrm{D}^b_G\mhm(\tilde{\mc{B}} \times \tilde{\mc{B}}).\]

\begin{prop}
With the notation as above, we have filtered isomorphisms
\[ (\mc{K} * \mc{M}, F^H_\bullet) \cong (\mc{K}, F^H_\bullet) * (\mc{M}, F^H_\bullet) \in \mrm{D}^b\coh(\tilde{\mc{D}}, F_\bullet) \]
and
\[ (\mc{K}_1 * \mc{K}_2, F^H_\bullet) \cong (\mc{K}_1, F^H_\bullet) * (\mc{K}_2, F^H_\bullet) \in \mrm{D}^b_{\mrm{St}}\coh(\tilde{\mc{D}} \boxtimes \tilde{\mc{D}}, F_\bullet).\]
\end{prop}

Here we write
\[ (\mc{K}, F_\bullet) *(\mc{M}, F_\bullet) = \mrm{R}\mrm{pr}_{1\bigcdot}((\mc{K}, F_\bullet) \overset{\mrm{L}}\otimes_{\mrm{pr}_2^{-1}(\tilde{\mc{D}}, F_\bullet)} \mrm{pr}_2^{-1}(\mc{M}, F_\bullet)) \in \mrm{D}^b\coh(\tilde{\mc{D}}, F_\bullet), \]
and
\[ (\mc{K}_1, F_\bullet) * (\mc{K}_2, F_\bullet) = \mrm{R}\mrm{pr}_{13 \bigcdot}((\mrm{pr}_{12}^{-1}\mc{K}_1, F_\bullet) \overset{\mrm{L}}\otimes_{(\mrm{pr}_2^{-1}\tilde{\mc{D}}, F_\bullet)} (\mrm{pr}_{23}^{-1}\mc{K}_2, F_\bullet)) \in \mrm{D}^b_{\mrm{St}}\coh^G(\tilde{\mc{D}}\boxtimes \tilde{\mc{D}}, F_\bullet),\]
where $\tilde{\mc{D}}$ acts on a $\tilde{\mc{D}} \boxtimes \tilde{\mc{D}}$-module on the right via the side-changing isomorphism $\tilde{\mc{D}} \cong \tilde{\mc{D}}^{\mathit{op}}$ coming from the fact that $\tilde{\mc{B}}$ has trivial canonical bundle.

\begin{proof}
Applying \cite[(6.16)]{DV}, we see that, after taking completions, we have a natural isomorphism
\[  (\mc{K} * \mc{M}, F^H_\bullet) \cong (\mc{K}, F^H_\bullet) * (\mc{M}, F^H_\bullet) \in \pro \mrm{D}^b\coh(\tilde{\mc{D}}, F_\bullet)_{\widetilde{\mu - \rho}},\]
so we deduce the claim for $\mc{K}*\mc{M}$ by Proposition \ref{prop:pro forgetful functor}. The proof for $\mc{K}_1 *\mc{K}_2$ is similar.
\end{proof}

We are now ready to define the intertwining functors. For $\lambda \in \mf{h}^*_\mb{R}$ and $w \in W$, we write
\[ \mc{I}_w^{!H} = \tilde{\Delta}_w^{(\lambda)} * -,\, \mc{I}_w^{*H} = \tilde{\nabla}_w^{*(\lambda)} * - \colon \mrm{D}^b\mhm(\mc{D}_{\widetilde{\lambda}}) \to \mrm{D}^b\mhm(\mc{D}_{\widetilde{w\lambda}}), \]
\[ \mc{I}_w^{!(\lambda)} = (\tilde{\Delta}_w^{(\lambda)}, F_\bullet^H) * - ,\, \mc{I}_w^{*(\lambda)} = (\tilde{\nabla}_w^{(\lambda)}, F_\bullet^H) * - \colon  \mrm{D}^b\coh(\tilde{\mc{D}}, F_\bullet) \to \mrm{D}^b\coh(\tilde{\mc{D}}, F_\bullet).\]

\begin{defn}
The functors $\mc{I}_w^{!H}$ and $\mc{I}_w^{*H}$ are called the \emph{Hodge intertwining functors}. The functors $\mc{I}_w^! := \mc{I}_w^{!(0)}$ and $\mc{I}_w^* := \mc{I}_w^{*(0)}$ are called the \emph{filtered intertwining functors}.
\end{defn}

\begin{rmk}
In the setting of twisted (rather than monodromic) mixed Hodge modules, one can give a much simpler geometric definition (see, e.g., \cite[\S 5.1]{DV1}) of Hodge intertwining functors
\[ \mc{I}_w^{!H},\, \mc{I}_w^{*H} \colon \mrm{D}^b\mhm(\mc{D}_\lambda) \to \mrm{D}^b\mhm(\mc{D}_{w\lambda})\]
by imitating the Beilinson-Bernstein definition for $\mc{D}_\lambda$-modules. One can check that this is compatible with our more sophisticated definition for monodromic mixed Hodge modules. The main advantage of our approach, beyond its generalization to monodromic objects, is that it constructs both the Hodge and the filtered intertwining functors and is well-adapted to comparing the two.
\end{rmk}

We conclude this subsection with a brief discussion of the (much simpler) \emph{translation functors}. For $\lambda \in \mf{h}^*_\mb{R}$ and $\nu \in \mb{X}^*(H)$, we set
\[ t_\nu^H := \tilde{\delta}_\nu^{(\lambda)} * - \colon \mrm{D}^b\mhm(\mc{D}_{\widetilde{\lambda}}) \to \mrm{D}^b\mhm(\mc{D}_{\widetilde{\lambda + \nu}})\]
and
\[ t_\nu := (\tilde{\delta}_\nu^{(0)}, F_\bullet^H) * - \colon \mrm{D}^b\coh(\tilde{\mc{D}}, F_\bullet) \to \mrm{D}^b\coh(\tilde{\mc{D}}, F_\bullet).\]
These functors can be described explicitly as follows.

\begin{prop} \label{prop:translations}
Let $\mc{M} \in \mrm{D}^b\mhm(\mc{D}_{\widetilde{\lambda}})$ and $\mc{N} \in \mrm{D}^b\coh(\tilde{\mc{D}}, F_\bullet)$. Then
\[ t_\nu^H(\mc{M}) \cong \mc{M} \otimes \mc{O}(\nu)\]
and
\[ t_\nu(\mc{N}, F_\bullet) \cong (\mc{N} \otimes \mc{O}(\nu), F_\bullet).\]
\end{prop}
\begin{proof}
Observe that we have an isomorphism
\begin{equation} \label{eq:translations 1}
\begin{aligned}
H \times \tilde{\mc{B}} &\xrightarrow{\sim} \tilde{X}_1 \subset \tilde{\mc{B}} \times \tilde{\mc{B}} \\
(h, \tilde{y}) &\mapsto (h\tilde{y}, \tilde{y}).
\end{aligned}
\end{equation}
Recall that by Proposition \ref{prop:hodge local systems}, we have an equivalence
\[ \pro \mhm_{\widetilde{0}}(H) \cong \pro \Mod_{\mrm{MHS}}(S(\mf{h}(1))) \]
given by taking the fiber at $1 \in H$; for $V \in \pro \Mod_{\mrm{MHS}}(S(\mf{h}(1)))$, we write $V$ also for the corresponding object in $\pro \mhm_{\widetilde{0}}(H)$. Then the pullback of $\widetilde{\mc{O}}_{X_1}(\lambda + \nu  - \rho, -\lambda - \rho)$ under the isomorphism \eqref{eq:translations 1} is $S(\mf{h}(1))^\wedge \boxtimes \mc{O}_{\tilde{\mc{B}}}$. Thus, we get
\[ t_\nu^H(\mc{M}) = \tilde{\delta}_\nu^{(\lambda)} * \mc{M} = a_*(S(\mf{h}(1))^\wedge \boxtimes \mc{M})(\dim H),\]
as pro-mixed Hodge modules on $\tilde{\mc{B}}$, where $a \colon H \times \tilde{\mc{B}} \to \tilde{\mc{B}}$ is the action map. Now, by Lemma \ref{lem:monodromic pushforward} below, we can rewrite the pushforward as
\[ a_*(S(\mf{h}(1))^\wedge \boxtimes \mc{M})(\dim H) = (S(\mf{h}(1))^\wedge \boxtimes \mc{M}) \overset{\mrm{L}}\otimes_{S(\mf{h}(1))} \mb{C} = \mc{M},\]
where the tensor product is over the anti-diagonal action of $S(\mf{h}(1))$ on the two factors. So $t_\nu(\mc{M}) = \mc{M}$, regarded as an object of $\mrm{D}^b\mhm(\mc{D}_{\widetilde{\lambda + \nu}})$ instead of $\mrm{D}^b\mhm(\mc{D}_{\widetilde{\lambda}})$. This is the definition of $\mc{M} \otimes \mc{O}(\mu)$ as a monodromic mixed Hodge module, so this proves the first assertion. For the second, note that, by construction,
\[ \tilde{\delta}^{(0)}_\nu = \mc{O}(\nu, 0) \otimes \tilde{\delta}^{(0)}_0,\]
so it suffices to treat the case $\nu = 0$. In this case, by the argument of \cite[Lemma 7.1]{DV}, we have $(\tilde{\delta}_0^{(0)}, F^H_\bullet) \cong (\tilde{\mc{D}}, F_\bullet)$, so the claim follows.
\end{proof}

\begin{lem} \label{lem:monodromic pushforward}
Let $X$ be a smooth variety, $H$ a torus and $\mc{M} \in \mrm{D}^b\mhm_{\widetilde{0}}(H \times X)$. Then
\begin{equation} \label{eq:monodromic pushforward}
 \pi_*\mc{M} = \mc{M}|_{\{1\} \times X} \overset{\mrm{L}}\otimes_{S(\mf{h}(1))} \mb{C}(-\dim H),
\end{equation}
where $\pi \colon X \times H \to X$ is the natural projection and $\mc{M}|_{\{1\} \times X}$ is the restriction to $X = \{1\} \times X$.
\end{lem}
\begin{proof}
By \cite[Lemma 2.8]{DV}, the restriction to $\{1\} \times X$ defines an equivalence
\[ \mhm_{\widetilde{0}}(H \times X) \cong \mhm(X, S(\mf{h}(1)),\]
with the category of mixed Hodge modules on $X$ with an action of $S(\mf{h}(1))$. For $\mc{N} \in \mrm{D}^b\mhm(X)$, the pullback $\pi^*\mc{N}$ is sent under this equivalence to $\mc{N}[-\dim H]$ with trivial $S(\mf{h}(1))$-action. The right adjoint $\pi_*$ is hence given by the formula \eqref{eq:monodromic pushforward} as claimed.
\end{proof}

In view of Proposition \ref{prop:translations}, $t_\nu^H$ and $t_\nu$ are compatible under the functor from Hodge modules to filtered $\tilde{\mc{D}}$-modules. We will therefore drop the superscript $H$ from now on and write $t_\nu^H = t_\nu$ since confusion is unlikely to arise.

\subsection{The affine Hecke algebra} \label{subsec:affine hecke}

Let us now describe the full convolution relations satisfied by the objects $\tilde{\Delta}_w^{(\lambda)}$, $\tilde{\nabla}_w^{(\lambda)}$ and $\tilde{\delta}_\mu^{(\lambda)}$.

\begin{thm} \label{thm:hecke relations}
We have the following.
\begin{enumerate}
\item \label{itm:hecke relations 1} For $\lambda \in \mf{h}^*_\mb{R}$ and $w_1, w_2 \in W$, if $\ell(w_1w_2) = \ell(w_1) + \ell(w_2)$ then
\[ \tilde{\Delta}_{w_1}^{(w_2\lambda)} *\tilde{\Delta}_{w_2}^{(\lambda)} = \tilde{\Delta}_{w_1w_2}^{(\lambda)} \quad \text{and} \quad \tilde{\nabla}_{w_1}^{(w_2\lambda)} * \tilde{\nabla}_{w_2}^{(\lambda)} = \tilde{\nabla}_{w_1w_2}^{(\lambda)}.\]
\item \label{itm:hecke relations 2} For $\lambda \in \mf{h}^*_\mb{R}$ and $w \in W$,
\[ \tilde{\Delta}_w^{(w^{-1}\lambda)} *\tilde{\nabla}_{w^{-1}}^{(\lambda)} = \tilde{\delta}_0^{(\lambda)} = \tilde{\nabla}_{w^{-1}}^{(w\lambda)} * \tilde{\Delta}_w^{(\lambda)}.\]
\item \label{itm:hecke relations 3} For $\lambda \in \mf{h}^*_\mb{R}$ and $\mu, \nu \in \mb{X}^*(H)$,
\[ \tilde{\delta}_\mu^{(\lambda + \nu)} * \tilde{\delta}_\nu^{(\lambda)} = \tilde{\delta}_{\mu + \nu}^{(\lambda)}.\]
\item \label{itm:hecke relations 4} For $\lambda \in \mf{h}^*_\mb{R}$, $\mu \in \mb{X}^*(H)$ and $w \in W$,
\[ \tilde{\delta}_{w\mu}^{(w\lambda)} * \tilde{\Delta}_w^{(\lambda)} * \tilde{\delta}_{-\mu}^{(\lambda + \mu)} = \tilde{\Delta}_w^{(\lambda + \mu)} \quad \text{and} \quad \tilde{\delta}_{w\mu}^{(w\lambda)} * \tilde{\nabla}_w^{(\lambda)} * \tilde{\delta}_{-\mu}^{(\lambda + \mu)} = \tilde{\nabla}_w^{(\lambda + \mu)}.\]
\item \label{itm:hecke relations 5} For any simple root $\alpha$ and any $\lambda \in \mf{h}^*_\mb{R}$, we have an exact sequence
\[ 0 \to \tilde{\Delta}_{s_\alpha}^{(\lambda)} \to \tilde{\nabla}_{s_\alpha}^{(\lambda)}(-1) \to \mc{K} \to 0,\]
where
\[ \mc{K} = \begin{cases} \coker\left(\check\alpha + n + 1 \colon \tilde{\delta}_{-n\alpha}^{(\lambda)} \to \tilde{\delta}_{-n\alpha}^{(\lambda)}(-1)\right), & \text{if $n := \langle \lambda, \check\alpha\rangle \in \mb{Z}$}, \\ 0, & \text{otherwise}.\end{cases}\]
Here $\check \alpha + n + 1$ denotes the action of $\check\alpha + n + 1 \in S(\mf{h})$ in the first factor of $\tilde{\mc{D}}$ or, equivalently, the action of $\check\alpha \in R_1 = S(\mf{h}(1))$.
\end{enumerate}
\end{thm}

The proof of Theorem \ref{thm:hecke relations}, which we give at the end of the subsection, consists of a standard set of calculations in the less standard setting of pro-monodromic mixed Hodge modules---this requires a little extra care due to the non-uniqueness of Hodge structures on free-monodromic local systems. Using Corollary \ref{cor:intertwining semi-continuity}, we also have the following.

\begin{thm} \label{thm:deformation relations}
Let $\alpha$ be a simple root and $\lambda \in \mf{h}^*_\mb{R}$. Then:
\begin{enumerate}
\item \label{itm:deformation relations 1} The filtered module $(\tilde{\Delta}_{s_\alpha}^{(\lambda)}, F_\bullet^H)$ (resp., $(\tilde{\nabla}_{s_\alpha}^{(\lambda)}, F_\bullet^H)$) depends only on $\lfloor\langle \lambda, \check\alpha\rangle\rfloor$ (resp., $\lceil \langle \lambda, \check\alpha \rangle \rceil$) and
\[ (\tilde{\nabla}_{s_\alpha}^{(\lambda)}, F_\bullet^H) = (\tilde{\Delta}_{s_\alpha}^{\left(\lambda - \frac{1}{2}\alpha\right)}, F_\bullet^H)\{1\} \quad \text{if $\langle \lambda, \check\alpha \rangle \in \mb{Z}$}.\]
\item \label{itm:deformation relations 2} If $\langle \lambda, \check\alpha \rangle < 0$ then we have a strict exact sequence
\[  0 \to (\tilde{\nabla}_{s_\alpha}^{(0)}, F_\bullet^H)\{-1\} \to (\tilde{\Delta}_{s_\alpha}^{(\lambda)}, F_\bullet^H) \to (\mc{K}, F_\bullet) \to 0\]
where $\mc{K}$ is an iterated extension of the filtered sheaves
\[ \coker(\check\alpha - n + 1 \colon \tilde{\delta}_{n\alpha} \to \tilde{\delta}_{n\alpha}\{-1\}) \quad \text{for $n \in \mb{Z}$ with $0 < n < - \langle \lambda, \check\alpha \rangle$}.\]
\item \label{itm:deformation relations 3} If $\langle \lambda, \check\alpha \rangle > 0$, then we have a strict exact sequence
\[ 0 \to (\tilde{\Delta}_{s_\alpha}^{(\lambda)}, F^H_\bullet) \to (\tilde{\Delta}_{s_\alpha}^{(0)}, F^H_\bullet) \to (\mc{K}, F_\bullet) \to 0\]
where $\mc{K}$ is an iterated extension of the filtered sheaves
\[ \coker(\check\alpha +  n + 1\colon \tilde{\delta}_{-n\alpha} \to \tilde{\delta}_{-n\alpha}\{-1\}) \quad \text{for $n \in \mb{Z}$ with $0 < n \leq \langle \lambda, \check\alpha \rangle$}.\]
\end{enumerate}
\end{thm}
\begin{proof}
We first prove that $(\tilde{\Delta}^{(\lambda)}_{s_\alpha}, F_\bullet^H)$ depends only on $\lfloor\langle \lambda, \check\alpha\rangle\rfloor$. Consider first the case where $\langle \lambda, \check\alpha \rangle \not\in \mb{Z}$. In this case, by Remark \ref{rmk:cone}, we have
\[ \mu := \frac{1}{2}\lfloor \langle \lambda, \check\alpha\rangle \rfloor \alpha - \lambda \in s_\alpha\mf{h}^*_{\mb{R}, +} - \mf{h}^*_{\mb{R}, +}.\]
Now, by Theorem \ref{thm:hecke relations} \eqref{itm:hecke relations 5}, the natural map
\[ \tilde{\Delta}_{s_\alpha}^{(\lambda + s\mu)} \to \tilde{\nabla}_{s_\alpha}^{(\lambda + s\mu)}(-1) \]
is an isomorphism if $\langle \lambda + s\mu, \check\alpha \rangle \not \in \mb{Z}$, which holds, e.g., if $0 \leq s < 1$. Thus, by Corollary \ref{cor:intertwining semi-continuity}, we have an isomorphism
\[ (\tilde{\Delta}_{s_\alpha}^{(\lambda)}, F_\bullet^H) \xrightarrow{\sim} (\tilde{\Delta}_{s_\alpha}^{(\lambda + \mu)}, F_\bullet^H) = (\tilde{\Delta}^{\frac{1}{2}\lfloor\langle \lambda, \check\alpha\rangle \rfloor}, F_\bullet^H),\]
which establishes \eqref{itm:deformation relations 1} in this case. If $\langle \lambda, \check\alpha\rangle \in \mb{Z}$, then Corollary \ref{cor:intertwining semi-continuity} again implies that
\[ (\tilde{\Delta}_{s_\alpha}^{(\lambda)}, F_\bullet^H) = (\tilde{\Delta}_{s_\alpha}^{(\lambda + \epsilon \alpha)}, F_\bullet^H)\]
for $0 < \epsilon \ll 1$. Since $\lfloor \langle \lambda, \check\alpha \rangle \rfloor = \lfloor \langle \lambda + \epsilon \alpha, \check\alpha \rangle \rfloor$, this shows that $(\tilde{\Delta}_{s_\alpha}^{(\lambda)}, F_\bullet^H)$ depends only on $\lfloor \langle \lambda, \check\alpha\rangle\rfloor$, since this is so for $(\tilde{\Delta}_{s_\alpha}^{(\lambda + \epsilon \alpha)}, F_\bullet^H)$. By an identical argument, $(\tilde{\nabla}^{(\lambda)}_{s_\alpha}, F_\bullet^H)$ depends only on $\lceil\langle \lambda, \check\alpha\rangle\rceil$. To complete the proof of \eqref{itm:deformation relations 1}, we note that for $\langle \lambda, \check\alpha\rangle \in \mb{Z}$, the map
\[ (\tilde{\nabla}_{s_\alpha}^{(\lambda)}, F_\bullet^H) \to (\tilde{\Delta}_{s_\alpha}^{\left(\lambda - \frac{1}{2}\alpha\right)}, F_\bullet^H)\{1\} \]
is an isomorphism by Corollary \ref{cor:intertwining semi-continuity} and Theorem \ref{thm:hecke relations} \eqref{itm:hecke relations 5}.

We next prove \eqref{itm:deformation relations 2}. By Remark \ref{rmk:cone}, we have $\lambda \in s_\alpha\mf{h}^*_{\mb{R}, +} - \mf{h}^*_{\mb{R}, +}$. So Corollary \ref{cor:intertwining semi-continuity} provides the desired strict injection $\tilde{\nabla}_{s_\alpha}^{(0)}\{-1\} \to \tilde{\Delta}_{s_\alpha}^{(\lambda)}$. To compute the cokernel, we note that, by Corollary \ref{cor:intertwining semi-continuity} and Theorem \ref{thm:hecke relations} \eqref{itm:hecke relations 5}, we may factor this as
\[ \tilde{\nabla}_{s_\alpha}\{-1\} = \tilde{\Delta}_{s_\alpha}^{(\nu)} \to \tilde{\nabla}_{s_\alpha}^{(\nu)}\{-1\} = \tilde{\Delta}_{s_\alpha}^{(2\nu)}  \to \cdots \to \tilde{\nabla}_{s_\alpha}^{(N\nu)}\{-1\} = \tilde{\Delta}_{s_\alpha}^{(\lambda)},\]
where
\[ \nu = -\frac{\lambda}{\langle \lambda, \check\alpha \rangle} \quad \text{and} \quad N = \lceil -\langle \lambda, \check\alpha\rangle \rceil - 1.\]
By Theorem \ref{thm:hecke relations} \eqref{itm:hecke relations 5} again, we have
\[ \coker(\tilde{\Delta}_{s_\alpha}^{(n\nu)} \to \tilde{\nabla}_{s_\alpha}^{(n\nu)}\{-1\}) \cong \coker(\check\alpha - n + 1 \colon \tilde{\delta}_{n\alpha} \to \tilde{\delta}_{n\alpha}\{-1\}).\]
So $\mc{K}$ is an iterated extension of these filtered sheaves as claimed, proving \eqref{itm:deformation relations 2}. The proof of \eqref{itm:deformation relations 3} is similar.
\end{proof}

To illustrate their meaning, we note that at the level of Grothendieck groups, Theorems \ref{thm:hecke relations} and \ref{thm:deformation relations} imply that the intertwining and translation functors generate an action of the affine Hecke algebra on the Grothendieck group of filtered $\tilde{\mathcal{D}}$-modules.

\begin{cor} \label{cor:affine hecke}
Let $\mc{H}$ denote the affine Hecke algebra over $\mb{Z}[u, u^{-1}]$ with generators $T_w$ for $w \in W$ and $t_\mu$ for $\mu \in \mb{X}^*(H)$ and relations
\begin{align}
T_1 = t_0 = 1, \quad & \label{eq:bernstein 1} \\
t_\mu t_\nu = t_{\mu + \nu}, \quad & \text{for $\mu, \nu \in \mb{X}^*(H)$}, \label{eq:bernstein 2} \\
T_{w_1} T_{w_2} = T_{w_1w_2}, \quad  & \text{if $\ell(w_1w_2) = \ell(w_1) + \ell(w_2)$}, \label{eq:bernstein 3} \\
(T_{s_\alpha} - u)(T_{s_\alpha} + 1)= 0, \quad & \text{for $\alpha$ a simple root}, \label{eq:bernstein 4} \\
t_{s_\alpha\mu} T_{s_\alpha} t_{-\mu} = T_{s_\alpha} + (1 - u)\frac{1 - t_{-\langle \mu, \check\alpha \rangle\alpha}}{1 - t_\alpha}, \quad & \text{for $\alpha$ simple and $\mu \in \mb{X}^*(H)$}. \label{eq:bernstein 5}
\end{align}
Then there is an action of $\mc{H}$ on the Grothendieck group $\mrm{K}(\mrm{D}^b\coh(\tilde{\mc{D}}, F_\bullet))$ such that, for $\mc{M} \in \mrm{D}^b\coh(\tilde{\mc{D}}, F_\bullet)$ we have
\begin{equation} \label{eq:affine hecke K 1}
\begin{gathered}
 [\mc{M}\{-1\}] = u[\mc{M}], \quad  [t_\mu(\mc{M})] = t_\mu[\mc{M}], \\
 [\mc{I}_w^!(\mc{M})] = (-1)^{\ell(w)}T_w[\mc{M}], \quad [\mc{I}_w^*(\mc{M})] = (-1)^{\ell(w)}T_{w^{-1}}^{-1}[\mc{M}],
\end{gathered}
\end{equation}
and for $\mc{N} \in \mrm{D}^b\mhm(\mc{D}_\lambda)$ and $\alpha$ a simple root,
\begin{equation} \label{eq:affine hecke K 2}
[(t_\mu(\mc{N}), F_\bullet^H)] = t_\mu[(\mc{N}, F_\bullet^H)],
\end{equation}
\begin{equation} \label{eq:affine hecke K 3}
[(\mc{I}_{s_\alpha}^{!H}(\mc{N}), F_\bullet^H)] =  \left(-T_{s_\alpha} + (u - 1) \frac{1 - t_{-\lfloor\langle \lambda, \check\alpha\rangle\rfloor\alpha}}{1 - t_\alpha}\right)[(\mc{N}, F_\bullet^H)]
\end{equation}
and
\begin{equation} \label{eq:affine hecke K 4}
[(\mc{I}_{s_\alpha}^{*H}(\mc{N}), F_\bullet^H)] =  \left(-u^{-1}T_{s_\alpha} + (1 - u^{-1}) \frac{1 - t_{-\lceil\langle \lambda, \check\alpha\rangle - 1\rceil \alpha}}{1 - t_\alpha}\right)[(\mc{N}, F_\bullet^H)].
\end{equation}
\end{cor}
\begin{proof}
Define an algebra homomorphism
\begin{equation} \label{eq:affine hecke K 5}
\mc{H} \to \mrm{K}(\mrm{D}^b_{\mrm{St}}\coh^G(\tilde{\mc{D}}\boxtimes \tilde{\mc{D}}, F_\bullet))
\end{equation}
sending $t_\mu$ to $[(\tilde{\delta}_{\mu}, F_\bullet^H)]$, $T_w$ to $(-1)^{\ell(w)} [(\tilde{\Delta}_{s_\alpha}^{(0)}, F_\bullet^H)]$ and $u$ to the filtration shift. Relation \eqref{eq:bernstein 1} follows from Proposition \ref{prop:translations} while \eqref{eq:bernstein 2} and \eqref{eq:bernstein 3} are clear from Theorem \ref{thm:hecke relations}. For \eqref{eq:bernstein 4} combine Theorem \ref{thm:hecke relations} \eqref{itm:hecke relations 2} and \eqref{itm:hecke relations 5}, while for \eqref{eq:bernstein 5}, combine Theorem \ref{thm:hecke relations} \eqref{itm:hecke relations 4} with Theorem \ref{thm:deformation relations} \eqref{itm:deformation relations 3}. So the homomorphism \eqref{eq:affine hecke K 5} is well-defined. Since $\mrm{K}(\mrm{D}^b_{\mrm{St}}\coh^G(\tilde{\mc{D}}\boxtimes \tilde{\mc{D}}, F_\bullet))$ acts on $\mrm{K}(\mrm{D}^b\coh(\tilde{\mc{D}}, F_\bullet))$ by convolution, we obtain the desired action of $\mc{H}$.

Now \eqref{eq:affine hecke K 1} is clear by construction, while \eqref{eq:affine hecke K 2} follows from Corollary \ref{cor:translation independence}. The relations \eqref{eq:affine hecke K 3} and \eqref{eq:affine hecke K 4} follow from Theorem \ref{thm:deformation relations}.
\end{proof}

\begin{proof}[Proof of Theorem \ref{thm:hecke relations}]
Let us now pay our dues and check the relations in Theorem \ref{thm:hecke relations}. The relations \eqref{itm:hecke relations 1}, \eqref{itm:hecke relations 3} and \eqref{itm:hecke relations 4} are special cases of the following lemma.

\begin{lemma} \label{lem:basic convolution}
If $\ell(w_1w_2) = \ell(w_1) + \ell(w_2)$ then
\[ j_{w_1!} \widetilde{\mc{O}}_{X_{w_1}}(\nu - \rho, -\mu - \rho) * j_{w_2!} \widetilde{\mc{O}}_{X_{w_2}}(\mu - \rho, -\lambda - \rho) \cong j_{w_1w_2!} \widetilde{\mc{O}}_{X_{w_1w_2}}(\nu - \rho, -\lambda - \rho) (\dim \mc{B})\]
and
\[ j_{w_1*} \widetilde{\mc{O}}_{X_{w_1}}(\nu - \rho, -\mu - \rho) * j_{w_2*} \widetilde{\mc{O}}_{X_{w_2}}(\mu - \rho, -\lambda - \rho) \cong j_{w_1w_2*} \widetilde{\mc{O}}_{X_{w_1w_2}}(\nu - \rho, -\lambda - \rho) (\dim \mc{B})\]
\end{lemma}
\begin{proof}
We prove the claim for $j_!$; the proof for $j_*$ is identical.

First, the condition on $w_1$ and $w_2$ ensures that the morphism
\[ \mrm{pr}_{13} \colon X_{w_1} \times_{\mc{B}} X_{w_2} \subset \mc{B} \times \mc{B} \times \mc{B} \to \mc{B} \times \mc{B} \]
is an isomorphism onto $X_{w_1w_2}$. Thus,
\[ j_{w_1!}\widetilde{\mc{O}}_{X_{w_1}} * j_{w_2!}\widetilde{\mc{O}}_{X_{w_2}} = j_{w_1w_2!} \widetilde{\mrm{pr}}_{13*}(\widetilde{\mc{O}}_{X_{w_1}} \otimes \widetilde{\mc{O}}_{X_{w_2}}) (\dim \tilde{\mc{B}}),\]
where we write
\[ \widetilde{\mrm{pr}}_{13} \colon \tilde{X}_{w_1} \times_{\tilde{\mc{B}}} \tilde{X}_{w_2} \to \tilde{X}_{w_1w_2} \]
for the natural projection and we have omitted the twists on the free-monodromic local systems for the sake of brevity. To compute this, we take the stalk at $(\tilde{x}, \tilde{w}_1\tilde{w}_2\tilde{x}) \in \tilde{X}_{w_1w_2}$, where $\tilde{w}_i$ are lifts of $w_i$ to $N_{U_\mb{R}}(T^c)$. We have an isomorphism
\begin{align*}
H &\xrightarrow{\sim} \widetilde{\mrm{pr}}_{13}^{-1}(\tilde{x}, \tilde{w}_1\tilde{w}_2\tilde{x}) \\
h &\mapsto (\tilde{x}, \tilde{w}_1\tilde{x}h, \tilde{w}_1\tilde{w}_2\tilde{x}).
\end{align*}
Thus, by Lemma \ref{lem:monodromic pushforward}, we have
\begin{align*}
(j_{w_1!}\widetilde{\mc{O}}_{X_{w_1}} * j_{w_2!}\widetilde{\mc{O}}_{X_{w_2}})|_{(\tilde{x}, \tilde{w}_1\tilde{w}_2 \tilde{x})} &= (\widetilde{\mc{O}}_{X_{w_1}}|_{(\tilde{x}, \tilde{w}_1\tilde{x})} \otimes \widetilde{\mc{O}}_{X_{w_2}}|_{(\tilde{w}_1\tilde{x}, \tilde{w}_1\tilde{w}_2\tilde{x})}) \overset{\mrm{L}}\otimes_{S(\mf{h}(1))} \mb{C}(\dim \mc{B}) \\
&= \widehat{R}_{w_1} \otimes_{S(\mf{h}(1))} \widehat{R}_{w_2}(\dim \mc{B}) = \widehat{R}_{w_1w_2}(\dim \mc{B}).
\end{align*}
So $j_{w_1!} \widetilde{\mc{O}}_{X_{w_1}} * j_{w_2!} \widetilde{\mc{O}}_{X_{w_2}} \cong j_{w_1w_2!} \widetilde{\mc{O}}_{X_{w_1w_2}} (\dim \mc{B})$ as claimed.
\end{proof}

Continuing with the proof of Thereom \ref{thm:hecke relations}, consider now relation \eqref{itm:hecke relations 2}. By \eqref{itm:hecke relations 1}, it suffices to consider the case where $w = s_\alpha$ for some simple root $\alpha$. We verify the desired relation for $\tilde{\nabla}_{s_\alpha}^{(s_\alpha\lambda)} * \tilde{\Delta}_{s_\alpha}^{(\lambda)}$; the relation for $\tilde{\Delta}_{s_\alpha}^{(s_\alpha\lambda)}* \tilde{\nabla}_{s_\alpha}^{(\lambda)}$ follows by symmetry. By Proposition \ref{prop:hodge local systems}, it suffices to check that
\begin{equation} \label{eq:hecke relations 2}
\begin{gathered}
i_{(\tilde x, \tilde{s}_\alpha \tilde{x})}^!(\tilde{\nabla}_{s_\alpha}^{(s_\alpha\lambda)} * \tilde{\Delta}_{s_\alpha}^{(\lambda)}) = 0 \quad \text{and} \\
i_{(\tilde x, \tilde{x})}^!(\tilde{\nabla}_{s_\alpha}^{(s_\alpha\lambda)} * \tilde{\Delta}_{s_\alpha}^{(\lambda)}) = \widehat{R}_1(-2\dim \tilde{\mc{B}})[-2 \dim \mc{B} -\dim H]
\end{gathered}
\end{equation}
as pro-$S(\mf{h}(1) \oplus \mf{h}(1))$-modules, where $i_{\tilde y}$ denotes the inclusion of a point $\tilde y \in \tilde{\mc{B}} \times \tilde{\mc{B}}$.

Observe that $X_{s_\alpha}$ is the complement of the diagonal inside $\mc{B} \times_{\mc{P}_\alpha} \mc{B}$, where $\mc{B} \to \mc{P}_\alpha$ is the projection to the corresponding partial flag variety. Let us write $\mb{P}^1$ for the fiber of $\mc{B} \to \mc{P}_\alpha$ containing $\tilde{x}$ and, for any subset $U \subset \mb{P}^1,$ $\tilde{U}$ for the pre-image of $U$ in $\tilde{\mc{B}}$. Then we have a commutative diagram
\[
\begin{tikzcd}
\mb{P}^1 - \{x, s_\alpha x\} \ar[r, "a' "] & \mb{P}^1 - \{x\} \ar[r, "b' "] & \mb{P}^1 \ar[dr, "p' "] \\
(\mb{P}^1 - \{x, s_\alpha x\})^\sim \ar[r] \ar[d, "r' "] \ar[u, "q' "'] & (\mb{P}^1 - \{x\})^{\sim} \ar[r] \ar[d] \ar[u] & \tilde{\mb{P}}^1 \ar[r] \ar[u] \ar[d] & (\tilde{x}, \tilde{s}_\alpha\tilde{x}) \ar[d, "i_{(\tilde{x}, \tilde{s}_\alpha\tilde{x})}"] \\
\tilde{X}_{s_\alpha} \times_{\tilde{\mc{B}}} \tilde{X}_{s_\alpha} \ar[r, "a"] & \tilde{X}_{s_\alpha} \times_{\mc{P}_\alpha} \tilde{\mc{B}} \ar[r, " b"] & \tilde{\mc{B}} \times_{\mc{P}_\alpha} \tilde{\mc{B}} \times_{\mc{P}_\alpha} \tilde{\mc{B}} \ar[r, "p"] & \tilde{\mc{B}} \times \tilde{\mc{B}} \\
(\mb{P}^1 - \{x\})^\sim \ar[r, "\mrm{id} "] \ar[u, "r'' "'] \ar[d, "q'' "] & (\mb{P}^1 - \{x\})^{\sim} \ar[r] \ar[u] \ar[d] & \tilde{\mb{P}}^1 \ar[r] \ar[u] \ar[d] & (\tilde{x}, \tilde{x}) \ar[u, "i_{(\tilde{x}, \tilde{x})}"'] \\
\mb{P}^1 - \{x\} \ar[r, "\mrm{id}"] & \mb{P}^1 - \{x\} \ar[r, "b'' " ] & \mb{P}^1 \ar[ur, "p'' "]
\end{tikzcd}
\]
in which each square is Cartesian. By construction, we have
\[ \tilde{\nabla}_{s_\alpha}^{(s_\alpha\lambda)} * \tilde{\Delta}_{s_\alpha}^{(\lambda)} = p_*b_*a_!(\widetilde{\mc{O}}_{X_{s_\alpha}} \otimes \widetilde{\mc{O}}_{X_{s_\alpha}})(-\dim \mc{B} + \dim H + 1). \]
So, by base change, there is a complex of pro-VMHS
\[ \mc{V} = q'_*(r')^!(\widetilde{\mc{O}}_{X_{s_\alpha}} \otimes \widetilde{\mc{O}}_{X_{s_\alpha}})(-\dim \mc{B} + \dim H + 1)\]
on $\mb{P}^1 - \{x, s_\alpha x\} \cong \mb{C}^\times$ such that
\[ i_{(\tilde x, \tilde{s}_\alpha \tilde{x})}^!(\tilde{\nabla}_{s_\alpha}^{(s_\alpha\lambda)} * \tilde{\Delta}_{s_\alpha}^{(\lambda)}) = p'_*b'_*a'_!\mc{V} = 0.\]
Similarly, we have
\[ i_{(\tilde x, \tilde{x})}^!(\tilde{\nabla}_{s_\alpha}^{(s_\alpha\lambda)} * \tilde{\Delta}_{s_\alpha}^{(\lambda)}) = p''_*b''_*\mc{V}' = \mc{V}'_{s_\alpha x}[1],\]
where
\[ \mc{V}' = q''_*(r'')^!(\widetilde{\mc{O}}_{X_{s_\alpha}} \otimes \widetilde{\mc{O}}_{X_{s_\alpha}})(-\dim \mc{B} + \dim H + 1).\]
Using Lemma \ref{lem:monodromic pushforward} and the fact that $r''$ is a closed immersion of codimension $\dim \mc{B} + 2\dim H + 1$, we have
\begin{align*}
\mc{V}'_{s_\alpha x} &= (\widetilde{\mc{O}}_{X_{s_\alpha}}|_{(\tilde{x}, \tilde{s}_\alpha\tilde{x})} \otimes \widetilde{\mc{O}}_{X_{s_\alpha}}|_{(\tilde{s}_\alpha\tilde{x}, \tilde{x})})\overset{\mrm{L}}\otimes_{S(\mf{h}(1))}\mb{C}(-2\dim \tilde{\mc{B}})[-\dim \mc{B} - 2\dim H - 1] \\
&= \widehat{R}_1(-2\dim \tilde{\mc{B}})[-\dim \mc{B} - 2\dim H - 1],
\end{align*}
so this proves \eqref{eq:hecke relations 2} and hence \eqref{itm:hecke relations 2}.

Finally, let us prove \eqref{itm:hecke relations 5}. By equivariance, it is enough to check the relation after pulling back to $\{\tilde x\} \times \tilde{\mb{P}}^1 \subset \tilde{\mc{B}} \times \tilde{\mc{B}}$. Moreover, a choice of root subgroup $r_\alpha \colon \mrm{SU}_2 \to U_\mb{R}$ compatible with the torus $T^c$ fixing $x$ defines an isomorphism
\[ \tilde{\mb{P}}^1 \cong H \times^{\check\alpha, \mb{C}^\times} (\mb{C}^2 - \{0\}),\]
so we may even check the relation after pulling back along the map
\[ r \colon H \times (\mb{C}^2 - \{0\}) \to \tilde{\mc{B}} \times \tilde{\mc{B}}.\]
Now, the free-monodromic local system pulls back to the pro-VMHS
\[ r^\circ \widetilde{\mc{O}}_{X_{s_\alpha}}(s_\alpha \lambda - \rho, -\lambda - \rho) = \coker(\check\alpha - s \colon \widehat{R}_{s_\alpha} \otimes z^{s - \langle \lambda, \check\alpha\rangle - 1} \mc{O}_{U}[[s]](1) \to \widehat{R}_{s_\alpha} \otimes z^{s - \langle \lambda, \check\alpha\rangle - 1}\mc{O}_{U}[[s]]),\]
where $U = r^{-1}(\tilde{X}_{s_\alpha})$, $z$ is the linear coordinate on $\mb{C}^2 - \{0\}$ such that $r(1, z = 0) = \tilde{x}$ and $r(1, z = 1) = \tilde{s}_\alpha\tilde{x}$, and we write $r^\circ$ for $r^*$ with a cohomological shift so that it sends $\mc{D}$-modules to $\mc{D}$-modules. Writing $j \colon U \to H \times (\mb{C}^2 - \{0\})$ for the inclusion, we have an exact sequence
\[ 0 \to j_! z^{s - \langle \lambda, \check\alpha\rangle - 1} \mc{O}_{U}[[s]] \to j_* z^{s - \langle \lambda, \check\alpha\rangle - 1} \mc{O}_{U}[[s]] \to \mc{K}_{\mrm{SL}_2} \to 0,\]
where
\[ \mc{K}_{\mrm{SL}_2} = \begin{cases} i_* \mb{C}(-1), & \text{if $\langle \lambda,  \check \alpha \rangle \in \mb{Z}$}, \\ 0, & \text{otherwise},\end{cases}\]
for $i$ the inclusion of $\{z = 0\}$. Thus, we have an exact sequence
\begin{equation} \label{eq:hecke relations 1}
 0 \to r^\circ \tilde{\Delta}_{s_\alpha}^{(\lambda)} \to r^\circ \tilde{\nabla}_{s_\alpha}^{(\lambda)} (-1) \to \coker(\check\alpha \colon \widehat{R}_{s_\alpha} \to \widehat{R}_{s_\alpha}(1)) \otimes \mc{K}_{\mrm{SL}_2}(-\dim \mc{B}) \to 0.
 \end{equation}
 Since
\[ \coker(\check\alpha \colon \widehat{R}_{s_\alpha} \to \widehat{R}_{s_\alpha}(1)) \cong \coker(\check\alpha \colon \widehat{R}_{1} \to \widehat{R}_{1}(1)) \]
as $S(\mf{h}(1)\oplus \mf{h}(1))$-modules, the cokernel in \eqref{eq:hecke relations 1} is $r^\circ \mc{K}$, so this proves \eqref{itm:hecke relations 5}.
\end{proof}

\subsection{The filtered intertwining property} \label{subsec:intertwining}

In this subsection, we conclude our study of intertwining functors by establishing the following key property:

\begin{theorem} \label{thm:intertwining}
If $(\mc{M}, F_\bullet) \in \mrm{D}^b\coh(\tilde{\mc{D}}, F_\bullet)$ and $w \in W$ then we have isomorphisms
\[ \mrm{R}\Gamma(\mc{M}, F_\bullet) \cong \mrm{R}\Gamma(\mc{I}_w^!(\mc{M}, F_\bullet)) \cong \mrm{R}\Gamma(\mc{I}_w^*(\mc{M}, F_\bullet)) \]
in $\mrm{D}^b\Mod_{fg}(U(\mf{g}), F_\bullet)$.
\end{theorem}

Following the approach of \cite[\S 6]{DV}, Theorem \ref{thm:intertwining} is proved by analyzing the interaction between the intertwining functors and the big projective in $\widehat{\mhm}(\mc{D}_{\widetilde{0}} \boxtimes \mc{D}_{\widetilde{0}}, G)$.

\begin{definition}
The \emph{big projective} is the projective cover
\[ \tilde{\Xi} \in \pro\coh(\mc{D}_{\widetilde{0}} \boxtimes \mc{D}_{\widetilde{0}}, G) \]
of the simple object $\Delta^{(0)}_1$.
\end{definition}

It was shown in the proof of \cite[Theorem 6.16]{DV} that, for any choice of mixed Hodge module structure on $\tilde{\Xi}$ (i.e., an upgrade to an object in $\widehat{\mhm}(\mc{D}_{\widetilde{0}} \boxtimes \mc{D}_{\widetilde{0}}, G)$, which must exist on general grounds, cf., \cite[Lemma 4.5.3]{beilinson-ginzburg-soergel}) such that the morphism $\tilde{\Xi} \to \Delta^{(0)}_1$ is a morphism of pro-mixed Hodge modules, the underlying filtered $\tilde{\mc{D}}\boxtimes \tilde{\mc{D}}$-module is
\begin{equation} \label{eq:big projective hodge}
 (\tilde{\Xi}, F_\bullet^H) \cong (\tilde{\mc{D}} \otimes_{U(\mf{g})} \tilde{\mc{D}}, F_\bullet),
\end{equation}
where the filtration on the right hand side is induced by the order filtration on $\tilde{\mc{D}}$. We will fix such a mixed Hodge module structure from now on. The importance of $\tilde{\Xi}$ for localization theory comes from the fact that, by \eqref{eq:big projective hodge},
\[ (\tilde{\Xi}, F_\bullet^H) * - = (\tilde{\mc{D}}, F_\bullet) \overset{\mrm{L}}\otimes_{(U(\mf{g}), F_\bullet)} \mrm{R}\Gamma(-) = \mrm{L}\Delta \circ \mrm{R}\Gamma \]
is the composition of the derived global sections functor $\mrm{R}\Gamma$ with its left adjoint.

To prove Theorem \ref{thm:intertwining}, we begin with some preliminary lemmas on the objects $\tilde{\Xi}$, $\tilde{\Delta}_w^{(0)}$ and $\tilde{\nabla}_w^{(0)}$.

\begin{lemma} \label{lem:xitilde stalks}
For all $w \in W$, the stalks
\[ j_w^* \tilde{\Xi} \quad \text{and} \quad j_w^! \tilde{\Xi} \]
are indecomposable free-monodromic local systems equipped with surjective morphisms of mixed Hodge modules
\[ j_w^*\tilde{\Xi} \to \mc{O}_{X_w}(-\rho, -\rho)(\ell(w) - \dim \mc{B}) \quad \text{and} \quad j_w^! \tilde{\Xi} \to \mc{O}_{X_w}(-\rho, -\rho).\]
Hence, $\tilde{\Xi}$ can be written as an extension of objects $\tilde{\Delta}_w'(\ell(w))$ (resp., $\tilde{\nabla}_w'(\dim \mc{B})$), where 
\[ (\tilde{\Delta}_w', F_\bullet^H) \cong (\tilde{\Delta}_w^{(0)}, F_\bullet^H) \quad \text{and} \quad (\tilde{\nabla}_w', F_\bullet^H) \cong (\tilde{\nabla}_w^{(0)}, F_\bullet^H).\]
\end{lemma}
\begin{proof}
To prove the statement about $j_w^*$, note that by \cite[Lemma 6.7]{DV}, we have
\[ \mb{C} \overset{\mrm{L}}\otimes_{S(\mf{h}(1))} j_w^* \tilde{\Xi} = \mc{O}_{X_w}(-\rho, -\rho)(\ell(w) - \dim \mc{B}).\]
Hence, ignoring the Hodge structure, $j_w^* \tilde{\Xi} \cong \widetilde{\mc{O}}_{X_w}(-\rho, -\rho)$ is an indecomposable free-monodromic local system, with the desired morphism of Hodge modules. By \cite[Lemma 6.9]{DV}, we have
\[  \mb{C} \overset{\mrm{L}}\otimes_{S(\mf{h}(1))} j_w^! \tilde{\Xi} = \mb{D}(\mb{C} \overset{\mrm{L}}\otimes_{S(\mf{h}(1))} j_w^* \tilde{\Xi})(-2 \dim \mc{B}) = \mc{O}_{X_w}(-\rho, -\rho) \]
so $j_w^!\tilde{\Xi}$ is also an indecomposable free-monodromic local system with the desired property. To conclude the final claim about $\tilde{\Xi}$, we note that $\tilde{\Xi}$ is an iterated extension both of the objects $j_{w!}j_w^*\tilde{\Xi}$ and of the objects $j_{w*}j_w^!\tilde{\Xi}$. Setting $\tilde{\Delta}_w' = j_{w!}j_w^*\tilde{\Xi}(-\ell(w))$ and $\tilde{\nabla}_w' = j_{w*}j_w^!\tilde{\Xi}(-\dim \mc{B})$, we conclude by applying Proposition \ref{prop:filtration uniqueness}.
\end{proof}

\begin{lem} \label{lem:convolution exactness}
Assume $\mc{K} \in \pro \coh^G(\mc{D}_{\widetilde{\lambda}} \boxtimes \mc{D}_{\widetilde{\nu}})$ and $\mc{M} \in \coh(\mc{D}_{\widetilde{\lambda}})_{rh}$. Then
\[ \mc{H}^i(\tilde{\Delta}_w^{(\lambda)} * \mc{M}) = 0 \quad \text{and} \quad \mc{H}^i(\tilde{\Delta}_w^{(\lambda)} * \mc{K}) = 0 \quad  \text{for $i < 0$}, \]
and
\[ \mc{H}^i(\tilde{\nabla}_w^{(\lambda)} * \mc{M}) = 0 \quad \text{and} \quad \mc{H}^i(\tilde{\nabla}_w^{(\lambda)} * \mc{K}) = 0 \quad \text{for $i > 0$}.\]
\end{lem}
\begin{proof}
Let us prove the claims for $\mc{M}$: the proofs for $\mc{K}$ are similar. By construction, we may write
\[ \tilde{\nabla}_w^{(\lambda)} * \mc{M} = p_{w*}(\widetilde{\mc{O}}_{X_w}(w\lambda - \rho, -\lambda - \rho) \otimes q_w^\circ \mc{M})(\dim H + \ell(w)),\]
where $p_w$ and $q_w$ are the first and second projections
\[ \tilde{\mc{B}} \xleftarrow{p_w} \tilde{X}_w \xrightarrow{q_w} \tilde{\mc{B}}\]
and $q_w^\circ = q_w^*[\dim \tilde{X}_w - \dim \tilde{\mc{B}}]$. Since $p_w$ and $q_w$ are smooth affine morphisms, it follows that $\mc{H}^i(\tilde{\nabla}_w^{(\lambda)} * \mc{M}) = 0$ for $i > 0$ by Artin vanishing. Now, by Theorem \ref{thm:hecke relations} \eqref{itm:hecke relations 2}, the functor $\tilde{\Delta}_w^{(\lambda)} * -$ is inverse, hence right adjoint, to the functor $\tilde{\nabla}_{w^{-1}}^{(w\lambda)} * - $. Since the latter is right t-exact, the former is therefore left t-exact, i.e., $\mc{H}^i(\tilde{\Delta}_w^{(\lambda)} * \mc{M}) = 0$ for $i < 0$ as claimed.
\end{proof}

Next, we compute the convolution of $\tilde{\Xi}$ with $\tilde{\Delta}^{(0)}_w$ as a filtered module:

\begin{lemma} \label{lem:xitilde intertwining}
For all $w \in W$, we have a filtered isomorphism
\[ (\tilde{\Xi}, F_\bullet^H) * (\tilde{\Delta}^{(0)}_w, F_\bullet^H) \cong (\tilde{\Xi}, F_\bullet^H).\]
\end{lemma}
\begin{proof}
By Theorem \ref{thm:hecke relations} \eqref{itm:hecke relations 1}, it suffices to treat the case where $w = s_\alpha$ for some simple root $\alpha$.

First, observe that, by Lemma \ref{lem:convolution exactness} and Lemma \ref{lem:xitilde stalks}, convolution with $\tilde{\Xi}$ is t-exact. In particular, $\tilde{\Xi} * \tilde{\Delta}^{(0)}_{s_\alpha}$ is concentrated in degree zero. Moreover, since convolution with $\tilde{\Delta}_{s_\alpha}^{(0)}$ and $\tilde{\nabla}_{s_\alpha}^{(0)}$ are inverse functors by Theorem \ref{thm:hecke relations} \eqref{itm:hecke relations 2}, we have, at the level of pro-$\mc{D}$-modules,
\[ \mrm{Ext}^i(\tilde{\Xi} * \tilde{\Delta}_{s_\alpha}, \mc{K}) = \mrm{Ext}^i(\tilde{\Xi}, \mc{K} * \tilde{\nabla}_{s_\alpha}^{(0)}) \]
for any $\mc{K} \in \coh^G(\mc{D}_{\widetilde{0}} \boxtimes \mc{D}_{\widetilde{0}})$. Since convolution with $\tilde{\nabla}_{s_\alpha}^{(0)}$ is right t-exact and $\tilde{\Xi}$ is projective, we conclude that $\tilde{\Xi} * \tilde{\Delta}^{(0)}_{s_\alpha}$ is projective as a pro-$\mc{D}$-module. Hence, we may write, as a pro-$\mc{D}$-module,
\[ \tilde{\Xi} * \tilde{\Delta}^{(0)}_{s_\alpha} = \bigoplus_{w' \in W} \mc{P}_{w'}^{\oplus n_{w'}}\]
for some $n_{w'}\geq 0$, where $\mc{P}_{w'}$ denotes the pro-projective cover of the irreducible object $j_{w'!*}\widetilde{\mc{O}}_{X_{w'}}(-\rho, -\rho)$. We claim that $n_1 = 1$ and $n_{w'} = 0$ for $w' \neq 1$. To see this, observe that for $w' \in W$ such that $\ell(w's_\alpha) = \ell(w') + 1$, we have
\[ \nabla_{w'}^{(0)} * \tilde{\nabla}_{s_\alpha}^{(0)} = \mb{C} \overset{\mrm{L}}\otimes_{S(\mf{h})} \tilde{\nabla}_{w'}^{(0)} * \tilde{\nabla}_{s_\alpha}^{(0)} = \mb{C} \overset{\mrm{L}}\otimes_{S(\mf{h})} \tilde{\nabla}_{w's_\alpha}^{(0)} = \mc{\nabla}_{w's_\alpha}^{(0)}. \]
Since the irreducible object $\Delta_1^{(0)} = \nabla_1^{(0)}$ has multiplicity $1$ in each costandard object $\nabla_{w'}^{(0)}$, we conclude that
\[ n_1 = \dim \Hom(\tilde{\Xi} * \tilde{\Delta}_{s_\alpha}^{(0)}, \nabla_1^{(0)}) = \dim \Hom(\tilde{\Xi}, \nabla_{s_\alpha}^{(0)}) = 1\]
and
\[ n_{w'} \leq \dim \Hom(\tilde{\Xi} * \tilde{\Delta}_{s_\alpha}^{(0)}, \nabla_{w'}^{(0)}) - n_1 = \dim \Hom(\tilde{\Xi}, \nabla_{ w's_\alpha}^{(0)}) - 1 = 0\]
for all $w'$ such that $\ell(w's_\alpha) = \ell(w') + 1$. Similarly, for $w' \in W$ such that $\ell(w' s_\alpha) = \ell(w') - 1$, we have
\[ \Delta_{w'}^{(0)}* \tilde{\nabla}_{s_\alpha}^{(0)} = \Delta_{w'w}^{(0)}\]
and hence, since $\Delta_1^{(0)}$ has multiplicity $1$ in each standard object $\Delta_{w'}^{(0)}$,
\[ n_{w'} \leq \dim \Hom(\tilde{\Xi} * \tilde{\Delta}_{s_\alpha}^{(0)}, \Delta_{w'}^{(0)}) - n_1 = \dim \Hom(\tilde{\Xi}, \Delta_{w's_\alpha}^{(0)}) - 1 = 0.\]
We conclude that $\tilde{\Xi} * \tilde{\Delta}_{s_\alpha}^{(0)} = \mc{P}_1 = \tilde{\Xi}$ as pro-$\tilde{\mc{D}}$-modules.

To upgrade this to a filtered isomorphism, observe that, as Hodge structures,
\[ \Hom(\tilde{\Xi} * \tilde{\Delta}_{s_\alpha}^{(0)}, \Delta_1^{(0)}) = \Hom(\tilde{\Xi}, \nabla^{(0)}_{s_\alpha}) = \Hom(j_{s_\alpha}^*\tilde{\Xi}, \widetilde{\mc{O}}_{X_{s_\alpha}}(-\rho, -\rho)(-\dim \mc{B} + 1)) = \mb{C}\]
by Lemma \ref{lem:xitilde stalks}. In other words, the unique non-zero morphism $\tilde{\Xi} * \tilde{\Delta}_{s_\alpha}^{(0)} \to \Delta^{(0)}_1$ is a morphism of mixed Hodge modules. So
\[ (\tilde{\Xi} *\tilde{\Delta}_{s_\alpha}^{(0)}, F_\bullet^H) \cong (\tilde{\Xi}, F_\bullet^H) \]
by the proof of \cite[Theorem 6.16]{DV}.
\end{proof}

Next, observe that the filtered object $(\tilde{\Xi}, F^H_\bullet)$ comes equipped with a comultiplication map
\[ (\tilde{\Xi}, F_\bullet^H) = \tilde{\mc{D}} \otimes_{U(\mf{g})} \tilde{\mc{D}} = \tilde{\mc{D}} \otimes_{U(\mf{g})} U(\mf{g}) \otimes_{U(\mf{g})} \tilde{\mc{D}} \to \tilde{\mc{D}} \otimes_{U(\mf{g})} \tilde{\mc{D}} \otimes_{U(\mf{g})} \tilde{\mc{D}} = (\tilde{\Xi}, F_\bullet^H) * (\tilde{\Xi}, F_\bullet^H)\]
and a counit
\[ (\tilde{\Xi}, F_\bullet^H) = \tilde{\mc{D}} \otimes_{U(\mf{g})} \tilde{\mc{D}} \to \tilde{\mc{D}} = (\tilde{\Delta}^{(0)}_1, F^H_\bullet).\]
Acting by convolution on $\mrm{D}^b\Coh(\tilde{\mc{D}}, F_\bullet)$, these are identified with the natural comonad structure on the functor
\[ (\tilde{\Xi}, F_\bullet^H) * - = \mrm{L}\Delta \circ \mrm{R}\Gamma(-) \colon \mrm{D}^b\coh(\tilde{\mc{D}}, F_\bullet) \to \mrm{D}^b\coh(\tilde{\mc{D}}, F_\bullet) \]
arising from the adjunction between $\mrm{R}\Gamma$ and $\mrm{L}\Delta$.

\begin{lemma} \label{lem:xitilde intertwining comodule}
The isomorphism of Lemma \ref{lem:xitilde intertwining} is an isomorphism of $(\tilde{\Xi}, F_\bullet)$-comodules.
\end{lemma}
\begin{proof}
We need to show that the left square in the diagram
\begin{equation} \label{eq:xitilde intertwining convolution 1}
\begin{tikzcd}
\tilde{\Xi} \ar[r] \ar[d, "\rotatebox{90}{$\sim$}"] & \tilde{\Xi} * \tilde{\Xi}\ar[d, "\rotatebox{90}{$\sim$}"]  \ar[r] & \tilde{\Xi} \ar[d, "\rotatebox{90}{$\sim$}"] \\
\tilde{\Xi} * \tilde{\Delta}^{(0)}_w \ar[r] & \tilde{\Xi} * \tilde{\Xi} * \tilde{\Delta}^{(0)}_w \ar[r] & \tilde{\Xi}* \tilde{\Delta}^{(0)}_w
\end{tikzcd}
\end{equation}
commutes. Here the top row is given by the comultiplication and first factor counit for $\tilde{\Xi}$ and the top row is given by convolving the bottom row with $\tilde{\mc{I}}_w$. Observe that, since $\tilde{\Xi}$ is generated as a $\tilde{\mc{D}} \boxtimes \tilde{\mc{D}}$-module by $\Gamma(F_0\tilde{\Xi}) = \mb{C}$, it is enough to check commutativity after applying the functor $\Gamma(F_0(-))$. Now, under this functor, all arrows in \eqref{eq:xitilde intertwining convolution 1} become isomorphisms, so it suffices to show that the outer square and the right hand square both commute. But this is obvious, since the right hand square is given by convolving the isomorphism $\tilde{\Xi} \cong \tilde{\Xi} * \tilde{\Delta}^{(0)}_w$ with the counit of $\tilde{\Xi}$ and the horizontal morphisms in the outer square are identities, so the lemma is proved.
\end{proof}

\begin{lemma} \label{lem:comonad}
Let $\mc{C}$ and $\mc{D}$ be categories and $F^L \colon \mc{D} \to \mc{C}$ a conservative functor with right adjoint $F^R \colon \mc{C} \to \mc{D}$. If $x_1, x_2 \in \mc{C}$ are objects and $\psi \colon F^LF^R(x_1) \cong F^LF^R(x_2)$ is an isomorphism of comodules under the comonad $F^LF^R$, then there exists an isomorphism $\psi \colon F^R(x_1) \cong F^R(x_2)$ such that $\phi = F^L(\psi)$.
\end{lemma}
\begin{proof}
Let us write $u \colon \mrm{id}_{\mc{D}} \to F^R F^L$ and $c \colon F^L F^R \to \mrm{id}_{\mc{C}}$ for the unit and counit of the adjunction. We set $\psi$ to be the composition
\[ \psi \colon F^R(x_1) \xrightarrow{uF^R} F^R F^L F^R(x_1) \xrightarrow{F^R(\phi)} F^RF^LF^R(x_2) \xrightarrow{F^Rc} F^R(x_2).\]
To see that $\phi = F^L(\psi)$, consider the commutative diagram
\[
\begin{tikzcd}
F^L F^R(x_1) \ar[r, "\phi"] \ar[d, "F^L u F^R"] & F^L F^R(x_2) \ar[rd, "\mrm{id}"] \ar[d, "F^LuF^R"] \\
F^LF^RF^LF^R(x_1) \ar[r, "F^LF^R(\phi)"] & F^LF^RF^LF^R(x_2) \ar[r, "F^LF^R c"] & F^LF^R(x_2).
\end{tikzcd}
\]
The left hand square commutes since $\phi$ is assumed to be a map of comodules, while the triangle on the right commutes because of unit-counit equations. The composition along the top is $\phi$, while the composition along the bottom is $F^L(\psi)$, so $\phi = F^L(\psi)$ as claimed. Finally, since $F^L(\psi)$ is an isomorphism and $F^L$ is conservative, $\psi$ itself is an isomorphism.
\end{proof}

\begin{proof}[Proof of Theorem \ref{thm:intertwining}]
By Lemmas \ref{lem:xitilde intertwining} and \ref{lem:xitilde intertwining comodule}, we have an isomorphism
\[ (\tilde{\Xi}, F_\bullet^H) * (\tilde{\Delta}^{(0)}_w, F_\bullet^H) \cong (\tilde{\Xi}, F_\bullet^H) \]
of comodules over $(\tilde{\Xi}, F_\bullet^H)$. Convolving with $(\mc{M}, F_\bullet) \in \mrm{D}^b\coh(\tilde{\mc{D}}, F_\bullet)$, we obtain filtered isomorphisms
\[ \mrm{L}\Delta \circ \mrm{R}\Gamma((\tilde{\Delta}_w^{(0)}, F_\bullet^H) * (\mc{M}, F_\bullet)) \cong \mrm{L}\Delta \circ \mrm{R}\Gamma(\mc{M}, F_\bullet) \]
of comodules over the comonad
\[ (\tilde{\Xi}, F_\bullet^H) * - = \mrm{L}\Delta \circ \mrm{R}\Gamma(-).\]
Applying Lemma \ref{lem:comonad} with $F^R = \mrm{R}\Gamma$ and $F^L = \mrm{L}\Delta$, we deduce that
\[ \mrm{R}\Gamma(\mc{I}_w^!(\mc{M}, F_\bullet)) = \mrm{R}\Gamma((\tilde{\Delta}^{(0)}_w, F_\bullet^H) * (\mc{M}, F_\bullet)) \cong \mrm{R}\Gamma(\mc{M}, F_\bullet) \]
as claimed. Since $\mc{I}_w^* = (\mc{I}_{w^{-1}}^!)^{-1}$, the assertion for $\mc{I}_w^*$ follows.
\end{proof}

We note that Theorem \ref{thm:intertwining} implies a monodromic version of Beilinson and Bernstein's result \cite[Theorem 12]{BeilinsonBernstein1983} on intertwining functors for unfiltered $\mc{D}$-modules. For any $\lambda \in \mf{h}^*_\mb{R}$, we have $\mc{D}$-module intertwining functors
\[
\mc{I}_w^{!\mc{D}}, \mc{I}_w^{*\mc{D}} \colon \mrm{D}^b\coh(\mc{D}_{\widetilde{\lambda}}) \to \mrm{D}^b\coh(\mc{D}_{\widetilde{w\lambda}}) 
\]
defined by convolving with the pro-objects $\tilde{\Delta}_w^{(\lambda)}$ and $\tilde{\nabla}_w^{(\lambda)}$ respectively. By construction, $\mc{I}_w^{!H}(\mc{M}) = \mc{I}_w^{!\mc{D}}(\mc{M})$ and $\mc{I}_w^{*H}(\mc{M}) = \mc{I}_w^{*\mc{D}}(\mc{M})$ as $\tilde{\mc{D}}$-modules.

\begin{cor} \label{cor:intertwining}
Let $\lambda \in \mf{h}^*_\mb{R}$ and $\mc{M} \in \mrm{D}^b\coh(\mc{D}_{\widetilde{\lambda}})$. Then for any simple root $\alpha$, we have
\[ \mrm{R}\Gamma(\mc{M}) \cong \begin{cases} \mrm{R}\Gamma(\mc{I}_{s_\alpha}^{*\mc{D}}(\mc{M})),  & \text{if $\langle \lambda, \check\alpha\rangle \geq 0$}, \\ \mrm{R}\Gamma(\mc{I}_{s_\alpha}^{!\mc{D}}(\mc{M})), & \text{if $\langle \lambda, \check\alpha\rangle \leq 0$},\end{cases}\]
as objects in $\mrm{D}^b\Mod_{fg}(U(\mf{g}))$.
\end{cor}
\begin{proof}
If $\langle \lambda, \check\alpha \rangle = 0$ then the claim is true even at the filtered level by Theorem \ref{thm:deformation relations} \eqref{itm:deformation relations 1} and Theorem \ref{thm:intertwining}. We prove the claim for $\langle \lambda, \check\alpha \rangle < 0$; the case $\langle \lambda, \check\alpha \rangle > 0$ follows since $\mc{I}_{s_\alpha}^{*\mc{D}}$ is inverse to $\mc{I}_{s_\alpha}^{!\mc{D}}$. Observe that for any $n \neq - \langle \lambda, \check\alpha \rangle$, the morphism
\[ \check\alpha - n + 1 \colon \mc{M}\otimes \mc{O}(n\alpha) \to \mc{M} \otimes \mc{O}(n\alpha) \]
is an isomorphism. So, in the exact triangle of Theorem \ref{thm:deformation relations} \eqref{itm:deformation relations 2}, we have $\mc{K} * \mc{M} = 0$ at the level of unfiltered $\tilde{\mc{D}}$-modules. Thus
\[ \mc{I}_{s_\alpha}^{!\mc{D}}(\mc{M}) = \tilde{\Delta}_{s_\alpha}^{(\lambda)} *\mc{M} \cong \widetilde{\nabla}_{s_\alpha}^{(0)} * \mc{M} \]
as complexes of $\tilde{\mc{D}}$-modules. We conclude by applying Theorem \ref{thm:intertwining} to any good filtration on $\mc{M}$.
\end{proof}

\section{Weak unipotence, vanishing theorems and the Cohen-Macaulay property} \label{sec:CM}

In this section, we apply the theory developed in \S\ref{sec:intertwining} to prove vanishing theorems for Hodge filtrations at non-dominant twists, under the additional hypothesis of \emph{very weak unipotence}. We define this property for $\mc{D}$-modules in \S\ref{subsec:strictness} and show that it implies unusually good behavior of the intertwining functors (Lemma \ref{lem:weakly unipotent intertwining 2}) and hence the strictness of the Hodge filtration on global sections (Theorem \ref{thm:strictness}). In the following subsections \S\ref{subsec:ideals}--\ref{subsec:CM property}, we use our strictness theorem to deduce the Cohen-Macaulay property for unipotent representations: we recall some fundamental facts about ideals in $U(\mf{g})$ in \S\ref{subsec:ideals}, discuss very weak unipotence for $\mc{D}$-modules versus $U(\mf{g})$-modules in \S\ref{subsec:D vs Ug}, and give the proof of Theorem \ref{thm:CM} in \S\ref{subsec:CM property}. Finally, in \S\ref{subsec:partial flags}, we use Theorem \ref{thm:strictness} to prove our vanishing theorem for Hodge modules on partial flag varieties (Theorem \ref{thm:partial filtered exactness}).

\subsection{The strictness theorem} \label{subsec:strictness}

Recall the following theorem from \cite{DV}.

\begin{thm}[{\cite[Theorem 5.2]{DV}}] \label{thm:filtered exactness}
Let $\lambda \in \mf{h}^*_\mb{R}$ be dominant and $\mc{M} \in \mhm(\mc{D}_{\widetilde{\lambda}})$. Then
\[ \mrm{H}^i(\mc{B}, \Gr^{F^H}_p\mc{M}) = 0 \quad \text{for $i > 0$}.\]
\end{thm}

In this subsection, we generalize Theorem \ref{thm:filtered exactness} to non-dominant situations under the following key hypothesis.

\begin{defn}\label{def:CandD}
For $\lambda \in \mf{h}^*_\mb{R}$ we write
\[ C(\lambda) = \mrm{Conv}(W\cdot \lambda) \subset \mf{h}^*_\mb{R}, \quad C^\circ(\lambda) = C(\lambda) - W\cdot \lambda \subset C(\lambda),\]
\[ D(\lambda) = \{ \mu \in \mb{Z}\Phi \mid \lambda + \mu \in C(\lambda)\} \quad \text{and} \quad D^\circ(\lambda) = \{ \mu \in \mb{Z}\Phi \mid \lambda + \mu \in C^\circ(\lambda)\},\]
where $\mrm{Conv}$ denotes convex hull. We say that $\mc{M} \in \mrm{D}^b\coh(\mc{D}_{\widetilde{\lambda}})$ is \emph{very weakly unipotent} if
\[ \mrm{R}\Gamma(\mc{B}, \mc{M} \otimes \mc{O}(\mu)) = 0 \quad \text{for all $\mu \in D^\circ(\lambda)$}.\]
\end{defn}

Assuming that $\lambda$ is dominant, it is a theorem of Beilinson-Bernstein that
\[ \mrm{H}^i(\mc{B}, \mc{M}) = 0 \quad \text{for $i > 0$}.\]
Given this fact, Theorem \ref{thm:filtered exactness} is equivalent to the assertion that the filtered complex $\mrm{R}\Gamma(\mc{M}, F_\bullet^H\mc{M})$ is \emph{strict}, i.e., that the associated spectral sequence
\[ \mrm{E}_1^{p ,q} = \mrm{H}^{p + q}(\mc{B}, \Gr^{F^H}_{-p} \mc{M}) \Rightarrow \mrm{H}^{p + q}(\mc{B}, \mc{M}) \]
degenerates at $\mrm{E}_1$. It is this formulation that we generalize to non-dominant twists.

\begin{thm} \label{thm:strictness}
Let $\lambda \in \mf{h}^*_\mb{R}$ be arbitrary and let $\mc{M} \in \mrm{D}^b\mhm(\mc{D}_{\widetilde{\lambda}})$ be very weakly unipotent. Then the filtered complex
\[ \mrm{R}\Gamma(\mc{M}, F_\bullet^H) \in \mrm{D}^b\Mod_{\mathit{fg}}(U(\mf{g}), F_\bullet) \]
is strict. In particular, $\mrm{H}^i(\mc{B}, \Gr^{F^H}\mc{M}) = \Gr^{F^H}\mrm{H}^i(\mc{B}, \mc{M})$ is the associated graded of a good filtration on a $U(\mf{g})$-module.
\end{thm}

We prove Theorem \ref{thm:strictness} below as a consequence of unusually good behavior of the intertwining functors for very weakly unipotent Hodge modules. We begin with the following lemma; recall for the statement the $\mc{D}$-module intertwining functors $\mc{I}_w^{!\mc{D}}$ and $\mc{I}_w^{*\mc{D}}$ introduced before Corollary \ref{cor:intertwining}.

\begin{lem} \label{lem:weakly unipotent intertwining 1}
Let $\mc{M} \in \mrm{D}^b\coh(\mc{D}_{\widetilde{\lambda}})$ be very weakly unipotent. Then $\mc{I}_w^{!\mc{D}}(\mc{M})$ is very weakly unipotent for all $w \in W$.
\end{lem}
\begin{proof}
It suffices to treat the case where $w = s_\alpha$ is a simple reflection. For $\mu \in D^\circ(s_\alpha\lambda) = s_\alpha D^\circ(\lambda)$, we have
\[ \mc{I}_{s_\alpha}^{!\mc{D}}(\mc{M}) \otimes \mc{O}(\mu) = \mc{I}_{s_\alpha}^{!\mc{D}}(\mc{M} \otimes \mc{O}(s_\alpha\mu)).\]
If $\langle \lambda + s_\alpha \mu, \check\alpha \rangle \leq 0$ then
\[ \mrm{R}\Gamma(\mc{I}_{s_\alpha}^{!\mc{D}}(\mc{M} \otimes \mc{O}(s_\alpha\mu))) = \mrm{R}\Gamma(\mc{M} \otimes \mc{O}(s_\alpha \mu)) = 0\]
by Corollary \ref{cor:intertwining}. If $\langle \lambda + s_\alpha \mu, \check\alpha\rangle > 0$, on the other hand, then Corollary \ref{cor:intertwining} gives
\[ \mrm{R}\Gamma(\mc{I}_{s_\alpha}^{*\mc{D}}(\mc{M}\otimes \mc{O}(s_\alpha\mu))) = \mrm{R}\Gamma(\mc{M} \otimes \mc{O}(s_\alpha \mu)) = 0.\]
So by Theorem \ref{thm:hecke relations} \eqref{itm:hecke relations 5}, we have
\[  \mrm{R}\Gamma(\mc{I}_{s_\alpha}^{!\mc{D}}(\mc{M} \otimes \mc{O}(s_\alpha\mu))) = 0 \quad \text{if $\langle \lambda + s_\alpha\mu,\check\alpha \rangle \not\in\mb{Z}$}\]
and
\[ \mrm{R}\Gamma(\mc{I}_{s_\alpha}^{!\mc{D}}(\mc{M} \otimes \mc{O}(s_\alpha\mu))) =  \mrm{Cone}(\mrm{R}\Gamma(\mc{M} \otimes \mc{O}(\mu - \langle \lambda, \check\alpha \rangle \alpha)) \to \mrm{R}\Gamma(\mc{M} \otimes\mc{O}(\mu - \langle \lambda, \check\alpha \rangle \alpha)))[-1],\]
if $\langle \lambda + s_\alpha\mu,\check\alpha \rangle \in\mb{Z}$. Since $\lambda + \mu - \langle \lambda, \check\alpha \rangle \alpha = s_\alpha(\lambda + s_\alpha\mu) \in C^\circ(\lambda)$, we have $\mu - \langle \lambda, \check\alpha \rangle \alpha \in D^\circ(\lambda)$, so
\[\mrm{R}\Gamma(\mc{M} \otimes\mc{O}(\mu - \langle \lambda, \check\alpha \rangle \alpha)) = 0.\]
So $\mrm{R}\Gamma(\mc{I}_{s_\alpha}^{!\mc{D}}(\mc{M}) \otimes \mc{O}(\mu)) = 0$ and hence $\mc{I}_{s_\alpha}^{!\mc{D}}(\mc{M})$ is very weakly unipotent as claimed.
\end{proof}

The next lemma, which we prove together with Theorem \ref{thm:strictness}, shows that, for very weakly unipotent modules, the Hodge intertwining functors commute with global sections after all.

\begin{lem} \label{lem:weakly unipotent intertwining 2}
Let $\lambda \in \mf{h}^*_\mb{R}$ be arbitrary and let $\mc{M} \in \mrm{D}^b\mhm(\mc{D}_{\widetilde{\lambda}})$ be very weakly unipotent. Then for any simple root $\alpha$, we have
\begin{equation} \label{eq:weakly unipotent intertwining 1}
\mrm{R}\Gamma(\mc{M}, F_\bullet^H) \cong \begin{cases} \mrm{R}\Gamma(\mc{I}_{s_\alpha}^{*H}(\mc{M}), F_\bullet^H)\{-1\},  & \text{if $\langle \lambda, \check\alpha\rangle > 0$}, \\ \mrm{R}\Gamma(\mc{I}_{s_\alpha}^{!H}(\mc{M}), F_\bullet^H)\{1\}, & \text{if $\langle \lambda, \check\alpha\rangle < 0$},\\ \mrm{R}\Gamma(\mc{I}_{s_\alpha}^{!H}(\mc{M}), F_\bullet^H) = \mrm{R}\Gamma(\mc{I}_{s_\alpha}^{*H}(\mc{M}), F_\bullet^H), & \text{if $\langle \lambda, \check\alpha\rangle = 0$}.\end{cases}
\end{equation}
Moreover, if $w \in W$ is the shortest element such that $w\lambda$ is dominant, then
\begin{equation} \label{eq:weakly unipotent intertwining 2}
 \mrm{R}\Gamma(\mc{M}, F_\bullet^H) \cong \mrm{R}\Gamma(\mc{I}_w^{!H}(\mc{M}), F_\bullet^H)\{\ell(w)\}.
\end{equation}
\end{lem}

\begin{proof}[Proof of Lemma \ref{lem:weakly unipotent intertwining 2} and Theorem \ref{thm:strictness}]
We prove both statements simultaneously by induction on the size of the finite set $D^\circ(\lambda)$.

We first prove the assertion \eqref{eq:weakly unipotent intertwining 1}. The case $\langle \lambda, \check\alpha\rangle = 0$ is clear from Theorems \ref{thm:intertwining} and \ref{thm:deformation relations} \eqref{itm:deformation relations 1}. We give the proof for $\langle \lambda, \check\alpha \rangle < 0$; the case $\langle \lambda, \check\alpha \rangle > 0$ is similar. Convolving the exact sequence of Theorem \ref{thm:deformation relations} \eqref{itm:deformation relations 2} with $(\mc{M}, F_\bullet^H)$, we get a distinguished triangle
\[ \mc{I}_{s_\alpha}^*(\mc{M}, F_\bullet^H)\{-1\} \to (\mc{I}_{s_\alpha}^{!H}(\mc{M}), F_\bullet^H) \to (\mc{M}', F_\bullet) \xrightarrow{+1},\]
where $(\mc{M}', F_\bullet)$ lies in the full triangulated subcategory generated by filtration shifts of $(\mc{M} \otimes \mc{O}(n\alpha), F_\bullet^H)$ for $0 < n < -\langle \lambda, \check\alpha\rangle$. Now, for each such $n$, clearly $n\alpha \in D^\circ(\lambda)$ so $\mrm{R}\Gamma(\mc{M} \otimes \mc{O}(n\alpha)) = 0$ since $\mc{M}$ is very weakly unipotent. Moreover,
\[ D^\circ(\lambda + n\alpha) \subsetneq \{\mu - n\alpha \mid \mu \in D^\circ(\lambda)\},\]
so $\mc{M} \otimes \mc{O}(n\alpha)$ is very weakly unipotent and, by induction, Theorem \ref{thm:strictness} holds for $\mc{M} \otimes \mc{O}(n\alpha)$. In particular, the complex $\mrm{R}\Gamma(\mc{M} \otimes \mc{O}(n\alpha), F_\bullet^H)$ is strict, and hence
\[ \mrm{R}\Gamma(\mc{M} \otimes \mc{O}(n\alpha), F_\bullet^H) = 0.\]
So $\mrm{R}\Gamma(\mc{M}', F_\bullet) = 0$ and hence
\[ \mrm{R}\Gamma(\mc{M}, F_\bullet^H) \cong \mrm{R}\Gamma(\mc{I}_{s_\alpha}^*(\mc{M}, F_\bullet^H))\cong \mrm{R}\Gamma(\mc{I}_{s_\alpha}^{!H}(\mc{M}), F_\bullet^H)\{1\} \]
by Theorem \ref{thm:intertwining}.

We next prove \eqref{eq:weakly unipotent intertwining 2}. Letting $w$ be the shortest element such that $w\lambda$ is dominant, we can write $w = w's_\alpha$ where $\alpha$ is a simple root such that $\langle \lambda, \check\alpha \rangle < 0$ and $\ell(w) = \ell(w') + 1$. Then, by \eqref{eq:weakly unipotent intertwining 1},
\[ \mrm{R}\Gamma(\mc{M}, F_\bullet^H) \cong \mrm{R}\Gamma(\mc{I}_{s_\alpha}^{!H}(\mc{M}), F_\bullet^H) \{1\}.\]
By Lemma \ref{lem:weakly unipotent intertwining 1}, $\mc{I}_{s_\alpha}^{!H}(\mc{M}) = \mc{I}_{s_\alpha}^{!\mc{D}}(\mc{M})$ is also very weakly unipotent. Since $|D^\circ(s_\alpha\lambda)| = |D^\circ(\lambda)|$, we have by induction on $\ell(w)$ that
\[ \mrm{R}\Gamma(\mc{I}_{s_\alpha}^{!H}(\mc{M}), F_\bullet^H) \{1\} = \mrm{R}\Gamma(\mc{I}_{w'}^{!H}\mc{I}_{s_\alpha}^{!H}(\mc{M}), F_\bullet^H)\{\ell(w)\} = \mrm{R}\Gamma(\mc{I}_w^{!H}(\mc{M}), F_\bullet^H)\{\ell(w)\} \]
as claimed.

Finally, we prove Theorem \ref{thm:strictness}. Let $w \in W$ be the shortest element such that $w \lambda$ is dominant. Since $\mc{I}_w^{!H}(\mc{M}) \in \mrm{D}^b\mhm(\mc{D}_{\widetilde{w\lambda}})$, by Theorem \ref{thm:filtered exactness}, the associated spectral sequence for $\mrm{R}\Gamma(\mc{I}_w^{!H}(\mc{M}), F_\bullet^H)$ coincides with the global sections of the spectral sequence for $(\mc{I}_w^{!H}(\mc{M}), F_\bullet^H)$. But since $(\mc{I}_w^{!H}(\mc{M}), F_\bullet^H)$ is a complex of mixed Hodge modules, it is automatically strict, and hence so is $\mrm{R}\Gamma(\mc{I}_w^{!H}(\mc{M}), F_\bullet^H)$. Since $\mc{M}$ is very weakly unipotent, by Lemma \ref{lem:weakly unipotent intertwining 2}, we deduce that
\[ \mrm{R}\Gamma(\mc{M}, F_\bullet^H) = \mrm{R}\Gamma(\mc{I}_w^{!H}(\mc{M}), F_\bullet^H)\{\ell(w)\} \]
is strict as claimed.
\end{proof}

\subsection{Ideals in the universal enveloping algebra} \label{subsec:ideals}

In this subsection, we take a brief digression to recall some basic facts about two-sided ideals in $U(\mf{g})$.

Recall that $\mf{Z}(\mf{g}) \cong S(\mf{h})^W$ denotes the center of $U(\mf{g})$ and that for $\lambda \in \mf{h}^*$ we write $\chi_\lambda$ for the corresponding character of $\mf{Z}(\mf{g})$.

\begin{defn}
We say that an ideal $I \subset U(\mf{g})$ has \emph{infinitesimal character} (resp., \emph{generalized infinitesimal character}) $\chi_\lambda$ if $I \cap \mf{Z}(\mf{g})$ contains the maximal ideal $\ker \chi_\lambda$ (resp., some power of $\ker \chi_\lambda$).
\end{defn}

Clearly if $M$ is a finitely generated $U(\mf{g})$-module, then $M$ has (generalized) infinitesimal character $\chi_\lambda$ if and only if $\mrm{Ann}_{U(\mf{g})}(M)$ does.

Next, recall that, by the Poincar\'{e}-Birkhoff-Witt theorem, there is a natural isomorphism of $G$-equivariant graded algebras $\gr^F U(\fg) \xrightarrow{\sim} S(\fg)$, where $F_\bullet U(\mf{g})$ is the degree filtration. If $I \subset U(\fg)$ is a two-sided ideal, then $\gr(I) := \gr^F(I)$ can be identified, by means of this isomorphism, with a $G$-invariant homogeneous ideal in $S(\fg)$.

\begin{defn}
The \emph{associated variety of $I$} is the closed $G \times \mb{C}^\times$-invariant subset $V(I) = V(\Gr^F I) \subset \fg^* = \Spec(S(\fg))$. For an irreducible component $Z$ of $V(I)$, we write $m_Z(I)$ for the multiplicity of the $S(\fg)$-module $S(\fg)/\gr(I)$ along $Z$.
\end{defn}

If $I$ has a generalized infinitesimal character, then the associated variety $V(I)$ is always contained in the nilpotent cone $\mc{N} \subset \mf{g}^*$.

Among the two-sided ideals in $U(\mf{g})$, the most important in representation theory are the primitive ideals.

\begin{defn}
An ideal $I \subset U(\mf{g})$ is \emph{primitive} if it is of the form $I = \mrm{Ann}(M)$ for some irreducible $U(\mf{g})$-module $M$.
\end{defn}

By Schur's lemma, every primitive ideal $I$ has an infinitesimal character. In particular, its associated variety is a closed $G$-invariant subset of the nilpotent cone. By \cite{Joseph1985}, this variety is irreducible, i.e., $V(I) = \overline{\mb{O}}$ for some nilpotent $G$-orbit $\mb{O} \subset \mc{N} \subset \mf{g}^*$. 

Under inclusions of ideals, associated varieties behave as follows.

\begin{prop} \label{prop:primitive assoc var}
If $I \subset J \subset U(\mf{g})$ are two-sided ideals then $V(J) \subset V(I)$. If, moreover, $I$ and $J$ are primitive and $I \subsetneq J$, then $V(J) \subsetneq V(I)$.
\end{prop}

\begin{proof}
The first statement is obvious, while the second is \cite[Corollary 3.8]{BorhoKraft}.
\end{proof}

We will be especially interested in \emph{maximal} ideals $I \subset U(\mf{g})$. These are automatically primitive, and have the following classification.

\begin{prop}[\cite{Duflo1977}]
There is a bijection
\[ \left\{\begin{matrix} \text{Maximal ideals} \\ \text{in $U(\mf{g})$} \end{matrix}\right\} \xrightarrow{\sim} \mf{h}^*/W\]
sending a maximal ideal $I \subset U(\mf{g})$ to its infinitesimal character.
\end{prop}

Recall that if $M$ is a finitely-generated $U(\fg)$-module equipped with a good filtration $F_{\bullet}M$, the \emph{associated variety of $M$} is the closed $\CC^{\times}$-invariant subset
\[ \mrm{AV}(M) = \operatorname{Supp} \Gr^F M \subset \mf{g}^*,\]
We note that $\AV(M)$ is \emph{independent} of the choice of good filtration. Of course, there is always an inclusion %
$$\AV(M) \subset V(\Ann(M)).$$
In certain special cases (e.g., when $M$ is a Harish-Chandra module) much more can be said about the relationship between $\AV(M)$ and $V(\Ann(M))$, see Lemma \ref{lem:support} below.

\subsection{Very weakly unipotent ideals and $U(\mf{g})$-modules} \label{subsec:D vs Ug}

Theorem \ref{thm:strictness} highlights the importance of the very weakly unipotent property for $\tilde{\mc{D}}$-modules. In this subsection, we examine related notions for two-sided ideals in $U(\mf{g})$.

\begin{defn}\label{def:weaklyunipotent}
Let $M$ be a $U(\mf{g})$-module with generalized infinitesimal character $\chi_\lambda$ for $\lambda \in \mf{h}^*_\mb{R}$. We say that $M$ is \emph{very weakly unipotent} if, for every finite dimensional (algebraic) representation $F$ of the complex adjoint group $G^{\mathit{ad}}$, we have
\[ M \otimes F = \bigoplus_{\gamma \in \mf{h}^*/W} (M \otimes F)_\gamma,\]
where $(M \otimes F)_\gamma$ has generalized infinitesimal character $\gamma$ and
\[ \text{$\gamma \in C^\circ(\lambda)$ implies $(M \otimes F)_\gamma = 0$}.\]
\end{defn}

\begin{rmk}
Definition \ref{def:weaklyunipotent} generalizes the notion of a \emph{weakly unipotent} module, see \cite[Definition 8.16]{Vogan1984}. That is, all weakly unipotent modules are very weakly unipotent. 
\end{rmk}

\begin{prop} \label{prop:weakly unipotent ideal}
Fix $\lambda \in \mf{h}^*_\mb{R}$ and let $I \subset U(\mf{g})$ be a two-sided ideal with generalized infinitesimal character $\chi_\lambda$. The following are equivalent.
\begin{enumerate}
\item \label{itm:weakly unipotent ideal 1} Every $U(\mf{g})$-module annihilated by $I$ is very weakly unipotent.
\item \label{itm:weakly unipotent ideal 2} There exists a very weakly unipotent $U(\mf{g})$-module $M$ with $\mrm{Ann}_{U(\mf{g})}(M) = I$.
\item \label{itm:weakly unipotent ideal 3} If $\gamma \in C^\circ(\lambda)$  and $F$ is a finite dimensional algebraic representation of $G^{\mathit{ad}}$ then $\mrm{Ann}_{\mf{Z}(\mf{g})}(U(\mf{g})/I \otimes F) \not\subset \ker \chi_\gamma$.
\end{enumerate}
\end{prop}
\begin{proof}
Clearly \eqref{itm:weakly unipotent ideal 1} implies \eqref{itm:weakly unipotent ideal 2}. To see that \eqref{itm:weakly unipotent ideal 2} implies \eqref{itm:weakly unipotent ideal 3}, note that if $\mrm{Ann}(M) = I$ then
\[ \mrm{Ann}_{\mf{Z}(\mf{g})}(U(\mf{g})/I \otimes F) = \mrm{Ann}_{\mf{Z}(\mf{g})}(M \otimes F)\]
is a finite codimension ideal in $\mf{Z}(\mf{g})$. So
\[ (M \otimes F)_\gamma \neq 0 \Leftrightarrow \mrm{Ann}_{\mf{Z}(\mf{g})}(U(\mf{g})/I \otimes F) \subset \ker \chi_\gamma.\]
Taking $M$ to be very weakly unipotent as in \eqref{itm:weakly unipotent ideal 2}, we deduce \eqref{itm:weakly unipotent ideal 3}. Finally, to prove \eqref{itm:weakly unipotent ideal 3} implies \eqref{itm:weakly unipotent ideal 1}, clearly if $I \cdot M = 0$ then any vector in $(M \otimes F)_\gamma$ is annihilated by
\[ \mrm{Ann}_{\mf{Z}(\mf{g})}(U(\mf{g})/I \otimes F) + (\ker \chi_\gamma)^n \]
for some $n > 0$. By \eqref{itm:weakly unipotent ideal 3}, $\gamma \in C^\circ(\lambda)$ implies
\[ \mrm{Ann}_{\mf{Z}(\mf{g})}(U(\mf{g})/I \otimes F) \not\subset \ker \chi_\gamma,\]
which implies $(M \otimes F)_\gamma = 0$ since $\ker\chi_\gamma$ is a maximal ideal in $\mf{Z}(\mf{g})$. So $M$ is very weakly unipotent.
\end{proof}

\begin{defn}\label{def:weaklyunipotentideal}
We say that an ideal $I \subset U(\mf{g})$ is \emph{very weakly unipotent} if it satisfies the equivalent conditions of Proposition \ref{prop:weakly unipotent ideal}.
\end{defn}

The following proposition relates very weak unipotence for $U(\mf{g})$-modules to the corresponding condition for $\tilde{\mc{D}}$-modules. Recall that for $\lambda \in \mf{h}^*$ integrally dominant, Beilinson-Bernstein localization gives two adjoint functors
\[ \Gamma \colon \coh(\mc{D}_{\widetilde{\lambda}}) \adjo \Mod_{fg}(U(\mf{g}))_{\widetilde{\chi_\lambda}} \cocolon \Delta_{\widetilde{\lambda}}.\]
The left adjoint $\Delta_{\widetilde{\lambda}}$ is given by
\[ \Delta_{\widetilde{\lambda}}(M) = (\tilde{\mc{D}}\otimes_{U(\mf{g})} M)_{\widetilde{\lambda}},\]
where $(-)_{\widetilde{\lambda}}$ denotes the $(\lambda - \rho)$-generalized eigenspace of $\mf{h} \subset \tilde{\mc{D}}$.

\begin{prop} \label{prop:weakly unipotent D to g}
Assume $\lambda \in \mf{h}^*_\mb{R}$ is integrally dominant, let $M \in \Mod_{fg}(U(\mf{g}))_{\widetilde{\chi_\lambda}}$ and $\mc{M} \in \coh(\mc{D}_{\widetilde{\lambda}})$.
\begin{enumerate}
\item \label{itm:weakly unipotent D to g 1} If $\mc{M}$ is very weakly unipotent then so is $\Gamma(\mc{M})$.
\item \label{itm:weakly unipotent D to g 2} If $M$ is very weakly unipotent then so is $\Delta_{\widetilde{\lambda}}(M)$.
\end{enumerate}
\end{prop}
\begin{proof}
To prove \eqref{itm:weakly unipotent D to g 1}, observe that $\Gamma(\mc{M}) \otimes F = \mrm{R}\Gamma(\mc{M} \otimes F)$ has a filtration in the derived category with subquotients $\mrm{R}\Gamma(\mc{M} \otimes \mc{O}(\mu))$ where $\mu \in \mb{Z}\Phi$ is a weight of the representation $F$ of $G^{\mathit{ad}}$ and $\mrm{R}\Gamma(\mc{M} \otimes \mc{O}(\mu))$ is a complex of $U(\mf{g})$-modules of generalized infinitesimal character $\chi_{\lambda + \mu}$. Since $\mrm{R}\Gamma(\mc{M} \otimes \mc{O}(\mu)) = 0$ for $\mu \in D^\circ(\lambda)$, we deduce that $(\Gamma(\mc{M}) \otimes F)_\gamma = 0$ for $\gamma \in C^\circ(\lambda)$. So $\Gamma(\mc{M})$ is very weakly unipotent as claimed.

To prove \eqref{itm:weakly unipotent D to g 2}, note that for $\mu \in D^\circ(\lambda)$, $\mrm{R}\Gamma(\mc{O}(\mu) \otimes \Delta_{\widetilde{\lambda}}(M))$ is a direct summand of the complex
\begin{equation} \label{eq:weakly unipotent D to g 1}
\mrm{R}\Gamma(\mc{O}(\mu) \otimes \tilde{\mc{D}} \overset{\mrm{L}}\otimes_{U(\mf{g})} M) = \mrm{R}\Gamma(\mc{O}(\mu) \otimes \tilde{\mc{D}}) \overset{\mrm{L}}\otimes_{U(\mf{g})} M
\end{equation}
with generalized infinitesimal character $\chi_{\lambda + \mu}$. We will show that there are no non-zero such summands, and thus $\Delta_{\widetilde{\lambda}}(M)$ is very weakly unipotent as claimed. To see this, observe that $\mrm{R}\Gamma(\mc{O}(\mu) \otimes \tilde{\mc{D}})$ is a complex of finitely generated (but not admissible) bimodules over $U(\mf{g})$ for which the diagonal action integrates to a representation of $G^{\mathit{ad}}$. Thus, it has a finite resolution by terms of the form $F \otimes U(\mf{g)}$ for $F$ a finite dimensional representation of $G^{\mathit{ad}}$. So \eqref{eq:weakly unipotent D to g 1} has a finite resolution by terms $F \otimes M$. Since $M$ is very weakly unipotent, the infinitesimal character $\chi_{\lambda + \mu}$ does not appear in $F \otimes M$, and hence it cannot appear in \eqref{eq:weakly unipotent D to g 1}.
\end{proof}

\subsection{The Cohen-Macaulay property}\label{subsec:CM property}

In this subsection, we prove the key result that the Hodge filtration on a very weakly unipotent representation with maximal annihilator is a Cohen-Macaulay $S(\fg)$-module.

We first recall the definition. If $(M, F_\bullet) \in \mrm{D}^b\Mod_{fg}(U(\mf{g}), F_\bullet)$, we define the \emph{filtered dual of $(M, F_\bullet)$} to be
\[ \mb{D}(M, F_\bullet) = \mrm{R}\Hom_{(U(\mf{g}), F_\bullet)}((M, F_\bullet), (U(\mf{g}), F_\bullet))\{-\dim \tilde{\mc{B}}\}[\dim \tilde{\mc{B}}] \in \mrm{D}^b\Mod_{fg}(U(\mf{g}), F_\bullet).\]
The shifts in filtration and cohomological degree are chosen to be compatible with the duality for mixed Hodge modules. Recall that the associated graded of a good filtration on a $U(\mf{g})$-module is a finitely generated $S(\mf{g})$-module or, equivalently, a coherent sheaf on $\mf{g}^*$. By construction,
\[ \Gr \mb{D}(M, F_\bullet) \cong \mb{D}\Gr^F M,\]
where, for $N \in \mrm{D}^b\Mod(S(\mf{g}))_{fg}$ we write
\[ \mb{D}N = \mrm{R}\Hom_{S(\mf{g})}(N, S(\mf{g}))\{-\dim \tilde{\mc{B}}\}[\dim \tilde{\mc{B}}].\]

\begin{defn}
We say that $N \in \Mod(S(\mf{g}))_{fg}$ is \emph{Cohen-Macaulay} if $\mb{D}N$ is concentrated in a single cohomological degree. We say that a filtered $U(\mf{g})$-module $(M, F_\bullet) \in \Mod_{fg}(U(\mf{g}), F_\bullet)$ is Cohen-Macaulay if $\Gr^F M$ is so.
\end{defn}

\begin{thm} \label{thm:CM}
Assume $\lambda \in \mf{h}^*_\mb{R}$ is dominant and $\mc{M} \in \mhm(\mc{D}_{\widetilde{\lambda}})$. If $\Gamma(\mc{M})$ is annihilated by a very weakly unipotent maximal ideal $I \subset U(\mf{g})$, then the filtered $U(\mf{g})$-module $(\Gamma(\mc{M}), \Gamma(F_\bullet^H\mc{M}))$ is Cohen-Macaulay.
\end{thm}

The first step in the proof of Theorem \ref{thm:CM} is to relate filtered duality for $U(\mf{g})$-modules to duality for mixed Hodge modules. Recall that the duality functor for mixed Hodge modules defines an equivalence
\[ \mb{D} \colon \mhm(\mc{D}_{\widetilde{\lambda}})^{\mathit{op}} \to \mhm(\mc{D}_{\widetilde{-\lambda}}).\]
For $(\mc{M}, F_\bullet) \in \mrm{D}^b\coh(\tilde{\mc{D}}, F_\bullet)$, let us set
\[ \mb{D}(\mc{M}, F_\bullet) = \mrm{R}\shom_{(\tilde{\mc{D}}, F_\bullet)}((\mc{M}, F_\bullet), (\tilde{\mc{D}}, F_\bullet))\{-\dim \mf{g}\}[\dim \tilde{\mc{B}}].\]
Then we have
\[ (\mb{D}\mc{M}, F_\bullet^H) \cong \mb{D}(\mc{M}, F_\bullet^H)\]
for $\mc{M} \in \mhm(\mc{D}_{\widetilde{\lambda}})$. In particular, $\mb{D}(\mc{M}, F_\bullet^H)$ is concentrated in degree $0$ as a filtered complex, i.e., $\Gr^{F^H}\mc{M}$ is a Cohen-Macaulay sheaf on $\tilde{\mf{g}}^* = \operatorname{Spec}_{\mc{B}}(\Gr^F\tilde{\mc{D}})$.

\begin{lem} \label{lem:filtered serre}
For $(\mc{M}, F_\bullet) \in \mrm{D}^b\coh(\tilde{\mc{D}}, F_\bullet)$, we have
\[\mb{D}\mrm{R}\Gamma(\mc{M}, F_\bullet) \cong \mrm{R}\Gamma(\mb{D}(\mc{M}, F_\bullet)).\]
\end{lem}
\begin{proof}
Recall that we have a filtered isomorphism
\[ \mrm{R}\Gamma(\tilde{\mc{D}}, F_\bullet) \cong U(\mf{g}) \otimes_{\mf{Z}(\mf{g})} S(\mf{h}).\]
Now, we have a morphism
\[
S(\mf{h}) \to S(\mf{h})^W\{\dim \mc{B}\} = \mf{Z}(\mf{g})^W\{\dim \mc{B}\}
\]
of filtered $S(\mf{h})^W$-modules given by
\[ f \mapsto \frac{\sum_{w \in W} (-1)^{\ell(w)} w f}{\prod_{\alpha \in \Phi_+} \check\alpha}\]
and thus a morphism
\begin{equation} \label{eq:filtered serre 1}
\mrm{R}\Gamma(\tilde{\mc{D}}, F_\bullet) \to U(\mf{g})\{\dim \mc{B}\}
\end{equation}
of filtered $U(\mf{g})$-modules. This defines for us the desired morphism
\begin{equation} \label{eq:filtered serre 2}
\mrm{R}\Gamma \mb{D}(\mc{M}, F_\bullet) \to \mb{D}\mrm{R}\Gamma(\mc{M}, F_\bullet)
\end{equation}
in the filtered derived category. To show that \eqref{eq:filtered serre 2} is an isomorphism, it suffices to check this after taking associated gradeds. This yields the morphism
\begin{equation} \label{eq:filtered serre 3}
\mrm{R}\tilde{\mu}_{\bigcdot}(\mb{D}\Gr^F \mc{M}) \to \mb{D}\mrm{R}\tilde{\mu}_{\bigcdot}\Gr^F\mc{M}
\end{equation}
where $\tilde{\mu} \colon \tilde{\mf{g}}^* \to \mf{g}^*$ is the Grothendieck-Springer map and $\mb{D}$ now denotes the analogous duality for coherent sheaves. Now, since the canonical bundles of $\tilde{\mf{g}}^*$ and $\mf{g}^*$ are trivial, we may identify these dualities with the Serre dual up to a cohomological shift. Moreover, the associated graded of \eqref{eq:filtered serre 1} is precisely the canonical trace morphism
\[ \mrm{R}\tilde{\mu}_{\bigcdot}(\omega_{\tilde{\mf{g}}^*}) \to \omega_{\mf{g}^*},\]
so by Grothendieck-Serre duality, \eqref{eq:filtered serre 3}, and hence \eqref{eq:filtered serre 2}, is an isomorphism as claimed.
\end{proof}

\begin{lem} \label{lem:strict dual}
If $\lambda \in \mf{h}^*_\mb{R}$ is dominant, $\mc{M} \in \mhm(\mc{D}_{\widetilde{\lambda}})$ and $\Gamma(\mc{M})$ is very weakly unipotent, then the filtered complex $\mb{D}\Gamma(\mc{M}, F_\bullet^H)$ is strict.
\end{lem}
\begin{proof}
Let us first prove the lemma in the case where $\mc{M}$ itself is very weakly unipotent. By Lemma \ref{lem:filtered serre} we have
\[ \mb{D}\Gamma(\mc{M}, F_\bullet^H) = \mb{D}\mrm{R}\Gamma(\mc{M}, F_\bullet^H) = \mrm{R}\Gamma(\mb{D}(\mc{M}, F_\bullet^H)) =\mrm{R}\Gamma(\mb{D}\mc{M}, F_\bullet^H).\]
Since $\mc{M}$ is assumed very weakly unipotent, clearly $\mb{D}\mc{M}$ is also (e.g., by Lemma \ref{lem:filtered serre} again). So we may apply Theorem \ref{thm:strictness} to conclude that $\mrm{R}\Gamma(\mb{D}\mc{M}, F_\bullet^H)$ is a strict complex as claimed.

Now consider the general case where only $\Gamma(\mc{M})$ is very weakly unipotent. Consider the localization $\Delta_{\widetilde{\lambda}}\Gamma(\mc{M}) \in \Mod(\mc{D}_{\widetilde{\lambda}})$; by Proposition \ref{prop:weakly unipotent D to g}, this is a very weakly unipotent $\tilde{\mc{D}}$-module. By \cite[Theorem 5.6]{DV}, $\Delta_{\widetilde{\lambda}}\Gamma(\mc{M})$ underlies an object in $\mhm(\mc{D}_{\widetilde{\lambda}})$ constructed as follows.

Let $S$ be the set of singular simple roots for $\lambda$, $\mc{B} \to \mc{P}_S$ the corresponding projection to a partial flag variety and
\[ j_S \colon \tilde{X}_S := \tilde{\mc{B}} \times_{\mc{P}_S} \tilde{\mc{B}} \to \tilde{\mc{B}} \times \tilde{\mc{B}}\]
the inclusion. Let $f \in \Gamma_{\mb{R}}^G(\tilde{X}_S)^{\mathit{mon}}$ be the unique element such that $\varphi(f) = (\lambda, -\lambda)$. Then it was shown in \cite[Lemmas 7.2 and 7.3]{DV} that
\[ \Delta_{\widetilde{\lambda}}\Gamma(\mc{M}) \cong \mc{H}^0(j_{S*}fj_S^*\tilde{\Xi} * \mc{M}) \in \mhm(\mc{D}_{\widetilde{\lambda}}),\]
where $\tilde{\Xi}$ is the big projective; we will regard $\Delta_{\widetilde{\lambda}}\Gamma(\mc{M})$ as endowed with the mixed Hodge module structure given by this formula from now on.

Now, Lemmas \ref{lem:xitilde stalks} and \ref{lem:convolution exactness} imply that
\[ \mc{H}^i(j_{S*}fj_S^*\tilde{\Xi} * \mc{M}) = 0 \quad \text{for $i \neq 0$},\]
so, in fact,
\[ \Delta_{\widetilde{\lambda}}\Gamma(\mc{M}) = j_{S*}fj_S^*\tilde{\Xi} * \mc{M}.\]
By \cite[Lemma 7.3]{DV}, we have
\[ (j_{S*} fj_S^*\tilde{\Xi}, F_\bullet^H) \cong (j_{S*}j_S^*\tilde{\Xi}, F_\bullet^H)\]
so
\[ (\Delta_{\widetilde{\lambda}}\Gamma(\mc{M}), F_\bullet^H) \cong (j_{S*}j_S^*\tilde{\Xi}, F_\bullet^H) * (\mc{M}, F_\bullet^H).\]
Now, recalling that convolution with $\tilde{\Xi}$ is exact, we have by Lemmas \ref{lem:xitilde stalks} and \ref{lem:xitilde intertwining} a filtered surjection
\begin{equation} \label{eq:strict dual 1}
(\tilde{\Xi} * j_{S*}j_S^*\tilde{\Xi}, F_\bullet^H) \to (\tilde{\Xi} * \tilde{\Delta}'_1, F_\bullet^H) \cong (\tilde{\Xi}, F_\bullet^H).
\end{equation}
Since the kernel $\mc{K}$ of \eqref{eq:strict dual 1} underlies an object in $\widehat{\mhm}(\mc{D}_{\widetilde{0}} \boxtimes \mc{D}_{\widetilde{0}}, G)$, it satisfies
\[ \mrm{H}^i(\mc{B}, F_0\mc{K})^G = 0 \quad \text{for $i > 0$}.\]
Since $(\tilde{\Xi}, F_\bullet^H)$ represents the functor $\Gamma(F_0(-))^G$, \eqref{eq:strict dual 1} must therefore split. So $(\tilde{\Xi}, F_\bullet^H)$ is a summand of $(\tilde{\Xi} * j_{S*}j_S^*\tilde{\Xi}, F_\bullet^H)$ and hence $(\tilde{\Xi}, F_\bullet^H) * (\mc{M}, F_\bullet^H)$ is a summand of $(\tilde{\Xi}, F_\bullet^H) * (\Delta_{\widetilde{\lambda}}\Gamma(\mc{M}), F_\bullet^H)$. Taking global sections, we find that
\[ S(\mf{h}) \otimes_{S(\mf{h})^W} \Gamma(\mc{M}, F_\bullet^H) \]
is a summand of
\[ S(\mf{h}) \otimes_{S(\mf{h})^W} \Gamma(\Delta_{\widetilde{\lambda}}\Gamma(\mc{M}), F_\bullet^H).\]
Since $S(\mf{h})$ is a free filtered $S(\mf{h})^W$-module and strictness is preserved under passing to summands, we deduce the strictness of $\mb{D}\Gamma(\mc{M}, F_\bullet^H)$ from the strictness of $\mb{D}\Gamma(\Delta_{\widetilde{\lambda}}\Gamma(\mc{M}), F_\bullet^H)$ already proved.
\end{proof}

Finally, we will need the following fact relating the associated variety of a $U(\fg)$-module $M$ to the associated variety of its annihilator. In the lemma below, we say that a finitely generated $U(\mf{g})$-module $M$ is \emph{holonomic} if $\mrm{AV}(M) \subset \mc{N}$ and
\[ \mu^{-1}(\mrm{AV}(M)) \subset T^*\mc{B}\]
is an isotropic subvariety of the symplectic variety $T^*\mc{B}$, where $\mu \colon T^*\mc{B} \to \mc{N}$ is the Springer resolution. Note that if $\mc{M} \in \coh(\tilde{\mc{D}})$ is a holonomic $\tilde{\mc{D}}$-module, then $M = \mrm{H}^i(\mc{B}, \mc{M})$ is holonomic for all $i$, since
\[ \mu^{-1}(\mrm{AV}(M)) \subset \mu^{-1}\mu(\mrm{AV}(\mc{M})) = \mrm{pr}_1(T^*\mc{B} \times_{\mc{N}} \mrm{AV}(\mc{M})),\]
which is isotropic since $\mrm{AV}(\mc{M}) \subset T^*\mc{B}$ and $T^*\mc{B} \times_{\mc{N}} T^*\mc{B} \subset T^*\mc{B} \times T^*\mc{B}$ are Lagrangians.

\begin{lem} \label{lem:support}
Let $M$ be a holonomic $U(\mf{g})$-module such that $I := \mrm{Ann}(M)$ is a primitive ideal and let $\mb{O}$ be the unique nilpotent $G$-orbit such that $V(I) = \overline{\mb{O}}$. Then
\[ \mrm{AV}(M) \cap \mb{O} \subset \mb{O} \]
is a non-empty Lagrangian subvariety of $\mb{O}$ with respect to the symplectic structure defined by the Kirillov-Kostant form, and
\[ \dim \mrm{AV}(M) = \frac{1}{2}\dim \mb{O}.\]
\end{lem}
\begin{proof}
Let $\mb{O}' \subset \mf{g}^*$ be a nilpotent $G$-orbit such that $\mrm{AV}(M) \cap \mb{O}'$ has maximal dimension. Let $X \subset \mb{O}'$ be the (open) complement of the closures of $\mrm{AV}(M) \cap \mb{O}''$ for $G$-orbits $\mb{O}'' \neq \mb{O}$. We show below that $\mrm{AV}(M) \cap X \subset X$ is a non-empty Lagrangian and that $X = \mb{O}' = \mb{O}$, from which the result follows.

By \cite[Lemma 8.19]{schmid-vilonen96}, the pullback of the symplectic form on $\mb{O}'$ under the Springer map $\mu \colon T^*\mc{B} \to \mc{N}$ agrees up to sign with the restriction of the symplectic form on $T^*\mc{B}$ to $\mu^{-1}(\mb{O}')$. Since $\mrm{AV}(M) \cap X$ is open in $\mrm{AV}(M)$, we conclude that
\[ \mrm{AV}(M) \cap X = \mu(\mu^{-1}(\mrm{AV}(M)) \cap \mu^{-1}(X)) \subset X\]
is isotropic since $\mu^{-1}(\mrm{AV}(M))$ is isotropic in $T^*\mc{B}$. Since $\mrm{AV}(M) = \operatorname{Supp}\Gr(M)$ is also involutive, it is therefore Lagrangian as claimed. Since $\mrm{AV}(M) \cap \mb{O}'$ has maximal dimension, we have that $\mrm{AV}(M) \cap X$ is non-empty and of dimension $\dim \mrm{AV}(M)$ by construction.

To conclude, we appeal to a theorem of Gabber (see, e.g.\ \cite[Theorem 9.11]{krause-lenagan99}) that
\[ \dim V(I) = \operatorname{GKdim} U(\mf{g})/I \leq 2 \operatorname{GKdim} M = 2\dim \mrm{AV}(M).\]
Since $\dim V(I) = \dim \mb{O}$ and $2\dim \mrm{AV}(M) = 2\dim (\mrm{AV}(M) \cap X) = \dim \mb{O}'$, this implies $\dim \mb{O} \leq \dim \mb{O}'$ and hence $\mb{O} = \mb{O}'$. This implies $X = \mb{O}$ also by construction, so the lemma is proved.
\end{proof}

\begin{proof}[Proof of Theorem \ref{thm:CM}]

Let $\mb{O}$ be the unique nilpotent orbit such that $V(I) = \bar{\mb{O}}$. Since $I$ is maximal, Lemma \ref{lem:support} implies that for every non-zero holonomic $U(\mf{g})$-module $N$ annihilated by $I$, the associated variety $\mrm{AV}(N)$ has dimension $\frac{1}{2}\dim \mb{O}$.

Now, applying this to the holonomic module $M = \Gamma(\mc{M})$, we have that $\dim \mrm{AV}(M) = \frac{1}{2}\dim\mb{O}$. Hence
\begin{equation} \label{eq:CM 1}
\mc{H}^i(\mb{D}M) = 0 \quad \text{for $i < \dim \mc{B} - \frac{1}{2}\dim\mb{O}$};
\end{equation}
by standard commutative algebra, this is true even at the graded level for any good filtration on $M$ (the shift by $\dim \mc{B}$ appears because of our definition of $\mb{D}$). We also have
\[ \dim \mrm{AV}(\mc{H}^i(\mb{D}M)) \leq \dim \mc{B} - i < \frac{1}{2}\dim\mb{O} \quad \text{for $i > \dim \mc{B} - \frac{1}{2}\dim\mb{O}$};\]
this again follows from the analogous fact in commutative algebra. Since $\mrm{AV}(\mc{H}^i(\mb{D}M)) \subset \mrm{AV}(M)$, it follows immediately that $\mc{H}^i(\mb{D}M)$ is a holonomic $U(\mf{g})$-module annihilated by $I$. So by the above, we must have
\begin{equation} \label{eq:CM 2}
 \mc{H}^i(\mb{D}M) = 0 \quad \text{for $i > \dim \mc{B} - \frac{1}{2}\dim\mb{O}$}.
\end{equation}

Finally, consider the Hodge filtration $F_\bullet^H M$. By Lemma \ref{lem:strict dual}, the filtered complex $\mb{D}(M, F_\bullet^H)$ is strict. So by \eqref{eq:CM 1} and \eqref{eq:CM 2}, we have
\[ \mc{H}^i(\mb{D}\Gr^{F^H}M) = \mc{H}^i(\Gr \mb{D}(M, F_\bullet^H)) = \Gr \mc{H}^i(\mb{D} M) = 0 \quad \text{for $i \neq \dim \mc{B} - \frac{1}{2}\dim \mb{O}$}.\]
So $\Gr^{F^H} M$ is Cohen-Macaulay as claimed.

\end{proof}

\subsection{Filtered exactness for partial flag varieties} \label{subsec:partial flags}

We conclude this section with another application of Theorem \ref{thm:strictness} to cohomology vanishing for monodromic mixed Hodge modules on partial flag varieties.

Fix a subset $S$ of the simple roots of $G$ and let $\mc{P}_S$ be the corresponding partial flag variety. We write $H_S$ for the torus with character group
\[ \mb{X}^*(H_S) = \{\mu \in \mb{X}^*(H) \mid \langle \mu, \check\alpha\rangle = 0 \text{ for $\alpha \in S$}\} = \mrm{Pic}^G(\mc{P}_S).\]
Let $\tilde{\mc{P}}_S \to \mc{P}_S$ be the tautological $H_S$-torsor; by construction, we have a commutative diagram
\[
\begin{tikzcd}
\tilde{\mc{B}} \ar[r] \ar[d] & \tilde{\mc{B}} \times^H H_S \ar[r] & \tilde{\mc{P}}_S \ar[d] \\
\mc{B} \ar[rr, "\pi_S"] & & \mc{P}_S.
\end{tikzcd}
\]
The $H_S$-torsor $\tilde{\mc{P}}_S$ defines a sheaf of rings $\tilde{\mc{D}}_{\mc{P}_S}$ on $\mc{P}_S$ with center $S(\mf{h}_S)$, where $\mf{h}_S = \mrm{Lie}(H_S)$.

Recall that for $S = \emptyset$ (i.e.\ for $\mc{P}_S = \mc{B}$), we have been writing
\[ \mc{D}_{\mc{B}, \lambda} = \tilde{\mc{D}}_{\mc{B}} \otimes_{S(\mf{h}), \lambda - \rho} \mb{C}.\]
There is a natural geometric interpretation for the $\rho$-shift here: the line bundle $\mc{O}_\mc{B}(-\rho)$ is a square-root of the canonical bundle of the flag variety. For general $S$, the canonical bundle of $\mc{P}_S$ is $\mc{O}(-2\rho(\mf{g}/\mf{p}_S))$, where
\[ \rho(\mf{g}/\mf{p}_S) := \frac{1}{2}\sum_{\alpha \in \Phi_+ - \mrm{span}(S)} \alpha.\]
So it is natural to write
\[ \mc{D}_{\mc{P}_S, \lambda} = \tilde{\mc{D}}_{\mc{P}_S} \otimes_{S(\mf{h}_S), \lambda - \rho(\mf{g}/\mf{p}_S)} \mb{C},\]
for $\lambda \in \mf{h}^*_S$. With this convention, one has a Beilinson-Bernstein exactness theorem for partial flag varieties:

\begin{thm}[{e.g., \cite[Theorem I.6.3]{Bien1990}}] \label{thm:partial exactness}
Assume that $\lambda \in \mf{h}_S^* \subset \mf{h}^*$ is integrally dominant. Then for $\mc{M} \in \Mod(\mc{D}_{\mc{P}_S, \lambda})$ we have
\[ \mrm{H}^i(\mc{P}_S, \mc{M}) = 0 \quad \text{for $i > 0$}.\]
\end{thm}

Of course, one has the same vanishing for monodromic $\mc{D}$-modules as well by reduction to the twisted case.

Note that if $\mc{M} \in \Mod(\mc{D}_{\mc{P}_S, \lambda})$ then its pullback $\pi_S^{\bigcdot} \mc{M}$ to $\mc{B}$ is naturally an object in $\Mod(\mc{D}_{\mc{B}, \lambda + \rho(\mf{l}_S)})$ satisfying
\[ \mrm{H}^i(\mc{B}, \pi_S^{\bigcdot} \mc{M}) = \mrm{H}^i(\mc{P}_S, \mc{M}) \quad \text{for all $i$},\]
where
\[ \rho(\mf{l}_S) := \rho - \rho(\mf{g}/\mf{p}_S) = \frac{1}{2}\sum_{\alpha \in \Phi_+ \cap \mrm{span}(S)} \alpha.\]
Note that since $\rho(\mf{l}_S)$ is \emph{not} (integrally) dominant in general, Theorem \ref{thm:partial exactness} gives a stronger result than applying Beilinson-Bernstein vanishing for the full flag variety to $\pi_S^{\bigcdot} \mc{M}$.

Now suppose that $\lambda \in \mf{h}^*_{S, \mb{R}}$ and set
\[ \mhm(\mc{D}_{\mc{P}_S, \lambda}) = \mhm_{\lambda - \rho(\mf{g}/\mf{p}_S)}(\tilde{\mc{P}}_S).\]
As usual, we have the functor of underlying Hodge-filtered $\mc{D}_\lambda$-module
\begin{align*}
\mhm(\mc{D}_{\mc{P}_S, \lambda}) &\to \Mod(\mc{D}_{\mc{P}_S, \lambda}, F_\bullet) \\
\mc{M} &\mapsto (\mc{M}, F_\bullet^H).
\end{align*}

We have the following Hodge filtration analog of Theorem \ref{thm:partial exactness}.

\begin{thm} \label{thm:partial filtered exactness}
Assume that $\lambda \in \mf{h}^*_{S, \mb{R}} \subset \mf{h}^*_\mb{R}$ is real dominant. Then for $\mc{M} \in \mhm(\mc{D}_{\mc{P}_S, \lambda})$, we have
\[ \mrm{H}^i(\mc{P}_S, \Gr^{F^H}_p \mc{M}) = 0 \quad \text{for $i > 0$ and all $p$}.\]
\end{thm}
\begin{proof}
Let us first consider the case where $\lambda = 0$. In this case, we claim that the $\tilde{\mc{D}}$-module $\pi_S^{\bigcdot} \mc{M}$ on $\mc{B}$ is very weakly unipotent. To see this, observe that by the Borel-Weil-Bott Theorem, we have
\[ \mrm{R}\pi_{S \bigcdot} \mc{O}(\nu - \rho(\mf{l}_S)) = 0 \quad \text{for $\nu \in C^\circ(\rho(\mf{l}_S))$}.\]
So for $\mu \in D^\circ(\rho(\mf{l}_S))$,
\[ \mrm{R}\Gamma(\mc{B}, \pi_S^{\bigcdot}\mc{M} \otimes \mc{O}(\mu)) = \mrm{R}\Gamma(\mc{P}_S, \mc{M} \otimes \mrm{R}\pi_{S \bigcdot}\mc{O}(\mu)) = 0.\]
So $\mc{M}$ is very weakly unipotent as claimed. So by Theorem \ref{thm:strictness}, the complex
\[ \mrm{R}\Gamma(\mc{B}, (\pi_S^{\bigcdot} \mc{M}, \pi_S^{\bigcdot} F_\bullet^H\mc{M})) = \mrm{R}\Gamma(\mc{P}_S, (\mc{M}, F_\bullet^H\mc{M}))\]
is strict. So by Theorem \ref{thm:partial exactness}, we have
\[ \mrm{H}^i(\mc{P}_S, \Gr^{F^H}_p \mc{M}) = \Gr^{F^H}_p\mrm{H}^i(\mc{P}_S, \mc{M}) = 0 \quad \text{for $i > 0$ and all $p$}.\]

Now consider the general case of $\lambda$ dominant. We follow the strategy of \cite[Theorem 4.2]{DV} and prove the claim by induction on $\dim \supp \mc{M}$.

Since $\lambda \in \mf{h}^*_\mb{R}$ is dominant, there exists a divisor $D \subset \mc{P}_S$ with complement $U = \mc{P}_S - D$ and a positive element $f \in \Gamma_\mb{R}(\tilde U)^{\mathit{mon}}$, where $\tilde{U} = U \times_{\mc{P}_S} \tilde{\mc{P}}_S$. Writing $j \colon \tilde U \to \tilde{\mc{P}}_S$ for the inclusion, as argued in \cite[Theorem 4.2]{DV}, it is enough to prove the claim with $j_!j^*\mc{M}$ in place of $\mc{M}$ since, by induction, the vanishing is known for Hodge modules supported in $D$. Consider the family of morphisms
\begin{equation} \label{eq:partial filtered exactness 1}
j_!f^sj^*\mc{M} \to j_*f^sj^*\mc{M} \quad \text{for $s \in [-1, 0)$};
\end{equation}
this is an isomorphism for $s$ outside a finite subset of $[-1, 0)$. Let $s_0 \in (-1, 0)$ be the largest element such that \eqref{eq:partial filtered exactness 1} is not an isomorphism. Then by \cite[Theorem 3.4]{DV},
\[ \Gr^{F^H} j_!j^*\mc{M} \cong \Gr^{F^H} j_*f^{s_0}j^*\mc{M}.\]
Arguing by induction on dimension again, the vanishing for $\Gr^{F^H} j_*f^{s_0}j^*\mc{M}$ follows from the vanishing for $\Gr^{F^H}j_!f^{s_0}j^*\mc{M}$. Repeating this process, we eventually reduce to showing the vanishing for $\Gr^{F^H} j_*f^{-1}j^*\mc{M}$. But $j_*f^{-1}j^*\mc{M} \in \mhm(\mc{D}_{\mc{P}_S, 0})$, so the desired vanishing in this case is proved above.
\end{proof}

\section{Unipotent representations} \label{sec:unipotent}

In this section, we will apply the Cohen-Macaulay property of \S \ref{subsec:CM property} to a class of Harish-Chandra modules called `unipotent representations' arising from the Orbit Method. In \S\ref{subsec:HCmodules}, we recall the Hodge-theoretic unitarity criterion of \cite{DV} and deduce as a simple corollary (Corollary \ref{cor:indecomposable}) that any Hermitian representation with indecomposable Hodge graded is necessarily unitary. In \S\ref{subsec:AC} and \ref{subsec:W}, we review some necessary background on associated cycles and dagger functors. In \S\ref{subsec:unipotent}, we recall the definitions of unipotent ideals and representations, and show that unipotent ideals are very weakly unipotent. In \S\ref{subsec:unipotent no hodge}, we use dagger functors to prove that small-boundary and birationally rigid unipotent representations are Hermitian and have irreducible associated cycles. Finally, in \S\ref{subsec:unipotent with hodge} we combine this with the Cohen-Macaulay property to deduce our main results on the unitarity and $K$-types of unipotent representations.

\subsection{Harish-Chandra modules and the unitarity criterion}\label{subsec:HCmodules}

Let $G$ be a complex reductive algebraic group and let $\sigma: G \to G$ be an anti-holomorphic involution. Let $G_{\RR}$ be a real Lie group (possibly disconnected) equipped with a finite surjective map onto a finite-index subgroup of $G^{\sigma}$. Choose an anti-holomorphic involution $\sigma_c$ of $G$, commuting with $\sigma$, such that $U_{\RR} := G^{\sigma_c}$ is compact. Then $\theta:=\sigma \circ \sigma_c$ is an algebraic involution of $G$. Let $K_{\RR} \subset G_{\RR}$ denote the preimage of $U_{\RR}$. Then $K_{\RR}$ is a maximal compact subgroup of $G_{\RR}$ and its complexification $K$ admits a finite surjective map onto a finite-index subgroup of $G^{\theta}$. Let $\fg$ and $\fk$ denote the (complex) Lie algebras of $G$ and $K$, and let $\fg_{\RR}$ and $\fu_{\RR}$ denote the (real) Lie algebras of $G_{\RR}$ and $U_{\RR}$. Using the map $G_{\RR} \to G$, we can (and will) regard $\fg_{\RR}$ (resp. $\fk)$ as a real form (resp. subalgebra) of $\fg$. 

A $(\fg,K)$-module is a $U(\fg)$-module $M$ equipped with a locally finite $K$-action, such that
\begin{itemize}
    \item The action map $\fg \otimes M \to M$ is $K$-equivariant.
    \item The differential of the $K$-action on $M$ coincides with the $\fg$-action on $M$, restricted to $\fk \subset \fg$.
\end{itemize}
Let $\Mod(\fg,K)$ denote the category of $(\fg,K)$-modules, and let $\Mod_{fl}(\fg,K)$ (resp. $\Mod_{fg}(\fg,K)$) denote the full subcategory of finite-length (resp. finitely-generated) $(\fg,K)$-modules. 

A Hermitian form $\langle \cdot, \cdot \rangle$ on a $(\fg,K)$-module $M$ is said to be \emph{$(\fg_{\RR},K_{\RR})$-invariant} (resp. \emph{$(\mathfrak{u}_{\RR},K_{\RR})$-invariant}) if $K_{\RR}$ acts on $M$ by unitary operators and $\fg_{\RR}$ (resp. $\mathfrak{u}_{\RR}$) acts $M$ by skew-Hermitian operators. We say that $M$ is \emph{Hermitian} (resp. \emph{unitarizable}) if it admits a non-degenerate (resp. positive-definite) $(\fg_{\RR},K_{\RR})$-invariant Hermitian form. If $M$ is irreducible and $\langle \cdot, \cdot \rangle$ is a non-degenerate $(\fg_{\RR},K_{\RR})$- or $(\fu_{\RR},K_{\RR})$-invariant Hermitian form on $X$, then $\langle \cdot, \cdot \rangle$ is unique up to multiplication by a nonzero real number. 

\begin{prop}[\cite{ALTV}]\label{prop:Hermitiancriterion}
Let $M$ be an irreducible $(\fg,K)$-module of real infinitesimal character. Then the following are true:
\begin{itemize}
    \item[(i)] $M$ admits a non-degenerate $(\fu_{\RR},K_{\RR})$-invariant Hermitian form $\langle \cdot, \cdot\rangle_c$. 
    \item[(ii)] $M$ admits a non-degenerate $(\fg_{\RR},K_{\RR})$-invariant Hermitian form if and only if there is an isomorphism of $(\fg,K)$-modules
    $$M \cong \theta^*M$$
    where $\theta^*M$ is the $(\fg,K)$-module equal to $M$ as a $K$-representation, with $\fg$-action twisted by $\theta$.
    \item[(iii)] If $\theta: M \to \theta^*M$ is $(\fg,K)$-module isomorphism such that $\theta^2=1$, then
    $$\langle x,y \rangle := \langle \theta(x),y\rangle_c, \qquad x,y \in M$$
    defines a non-degenerate $(\fg_{\RR},K_{\RR})$-invariant Hermitian form on $M$.
\end{itemize}
\end{prop}

If $M$ is a finitely-generated $(\fg,K)$-module equipped with a $K$-invariant good filtration $F_{\bullet}M$, then $\gr^F(M)$ has the structure of a finitely-generated $K$-equivariant $S(\fg/\fk)$-module. 

\begin{theorem}[{\cite[Theorem 8.2]{DV}}]\label{thm:unitarity}
Let $M$ be an irreducible $(\fg,K)$-module of real infinitesimal character. Suppose that $M$ is Hermitian, and fix an isomorphism $\theta:M \to \theta^*M$ as in Proposition \ref{prop:Hermitiancriterion}(iii). Then the following are true
\begin{itemize}
    \item[(i)] $\theta$ preserves the Hodge filtration $F^H_{\bullet}M$ and therefore induces an involution $\gr^{F^H}(\theta)$ of $\gr^{F^H}(M)$ (as a $K \times \CC^{\times}$-representation). 
    \item[(ii)] $M$ is unitary if and only if $\gr^{F^H}(\theta)$ acts by $(-1)^p$ on $\gr^{F^H}_p(M)$ for all $p$, or by $(-1)^{p+1}$ for all $p$.
\end{itemize}
\end{theorem}

Theorem \ref{thm:unitarity} has the following easy corollary.

\begin{cor}\label{cor:indecomposable}
Let $M$ be an irreducible Hermitian $(\fg,K)$-module of real infinitesimal character such that $\gr^{F^H}(M)$ is indecomposable as a $K$-equivariant $S(\fg/\fk)$-module. Then $M$ is unitary.
\end{cor}

\begin{proof}
    Since $M$ is Hermitian, Proposition \ref{prop:Hermitiancriterion}(ii) implies that there is a $K$-equivariant involution $\theta: M \to M$ such that
$$\theta(\xi v) = \theta(\xi)\theta(v), \qquad \xi \in \fg, \ v \in M.$$
By Theorem \ref{thm:unitarity}(i), this involution preserves the Hodge filtration $F_{\bullet}^H M$, and therefore induces a $K$-equivariant involution $\gr^{F^H}(\theta)$ of $\mathcal{F} := \gr^{F^H}(M)$ such that
$$\gr^{F^H}(\theta)( \xi m) = - \xi \gr^{F^H}(\theta)(m), \qquad \xi \in \fg, \ m \in \mathcal{F}.$$
Here we have used that $\theta(\xi) = -\xi$ in $\mf{g}/\mf{k}$.
Define a $K$-equivariant involution $T\colon \mathcal{F} \to \mathcal{F}$ by
$$T(m) = (-1)^p\theta(m), \qquad m \in \mathcal{F}_p = \gr_p^{F^H}(M).$$
Note that
$$T(\xi m) = (-1)^{p+1} \theta(\xi m) = (-1)^{p+2} \xi \theta(m) = (-1)^{2p+2} \xi T(m) = \xi T(m), \qquad \xi \in \fg, \ m \in \mathcal{F}_p.$$
So $T$ is an involution of $K$-equivariant $S(\fg/\fk)$-modules. Let $\mathcal{F}_{\pm}$ denote its $\pm 1$-eigenspaces. These are $K$-equivariant $S(\fg/\fk)$-modules, and there is a decomposition
$$\mathcal{F}= \mathcal{F}_+ \oplus \mathcal{F}_-$$
Since $\mathcal{F}$ is indecomposable, one of $\mathcal{F}_{\pm}$ is $0$. Thus, $M$ is unitary by Theorem \ref{thm:unitarity}(ii).
\end{proof}

\subsection{Associated cycles}\label{subsec:AC}

Let $\cN \subset \fg^*$ denote the nilpotent cone and let $\cN_{\theta}$ denote the set-theoretic intersection $\cN \cap (\fg/\fk)^*$. Note that $K$ acts on $\cN_{\theta}$ with finitely many orbits (\cite[Corollary 5.22]{Vogan1991}), denoted $\OO_{\theta,1},...,\OO_{\theta,n}$, and each $K$-orbit is a smooth Lagrangian subvariety in its $G$-saturation (\cite[Corollary 5.20]{Vogan1991}). Let $\Vect^K(\OO_{\theta,i})$ denote the category of $K$-equivariant vector bundles on $\OO_{\theta,i}$, and let $\mrm{K}_+\Vect^K(\OO_{\theta,i})$ denote the semigroup of effective classes in the Grothendieck group $\mrm{K}\Vect^K(\OO_{\theta,i})$ of $\Vect^K(\OO_{\theta,i})$. An \emph{associated cycle} for $(\fg,K)$ is a sum of the form
$$\sum_{i=1}^n [\mathcal{V}_i] \cdot \OO_{\theta,i}, \qquad [\mathcal{V}_i] \in \mrm{K}_+\Vect^K(\OO_{\theta,i})$$
subject to the following condition: if $[\mathcal{V}_i] \neq 0$, then $[\mathcal{V}_j]=0$ for every $j$ such that $\OO_{\theta,j} \subseteq \partial \OO_{\theta,i}$. Let $\mathrm{AC}(\fg,K)$ denote the set of associated cycles for $(\fg,K)$. We regard $\mathrm{AC}(\fg,K)$ as a semigroup, with addition defined by the formula
$$(\sum_{i=1}^n [\mathcal{V}_i] \cdot \OO_i) + (\sum_{i=1}^n [\mathcal{V}'_i] \cdot \OO_i) = (\sum_{i=1}^n [\mathcal{W}_i] \cdot \OO_i)$$
where $[\mathcal{W}_i]=[\mathcal{V}_i]+[\mathcal{V}'_i]$, unless there is a $j$ such that $\OO_i \subseteq \partial \OO_j$ and $[\mathcal{V}_j]+[\mathcal{V}'_j] \neq 0$, in which case we define $[\mathcal{W}_i]=0$.

Now let $\mathcal{F}$ be a $K$-equivariant coherent sheaf on $(\fg/\fk)^*$ with support contained in $\cN_{\theta}$. Write $\OO_{\theta,i_1},...,\OO_{\theta,i_p}$ for the maximal $K$-orbits in $\mathrm{Supp}(\mathcal{F})$. Choose a finite filtration by $K$-invariant subsheaves $0 = \mathcal{F}_0 \subsetneq \mathcal{F}_1 \subsetneq ... \subsetneq \mathcal{F}_n=\mathcal{F}$ such that each successive quotient $\mathcal{F}_k/\mathcal{F}_{k-1}$ is generically reduced along each irreducible component $\overline{\OO}_{\theta,i_j}$ of $\mathrm{Supp}(\mathcal{F})$ (see \cite[Proposition 2.9, Lemma 2.11]{Vogan1991}). Then $(\mathcal{F}_k/\mathcal{F}_{k-1})|_{\OO_{\theta,i_j}}$ is a $K$-equivariant vector bundle on $\OO_{\theta,i_j}$. We define
$$\AC(\mathcal{F}) := \sum_j (\sum_k [(\mathcal{F}_k/\mathcal{F}_{k-1})|_{\OO_{\theta,i_j}}] ) \cdot \OO_{\theta,i_j} \in \AC(\fg,K).$$
By \cite[Theorem 2.13]{Vogan1991}, $\AC(\mathcal{F})$ is independent of the choice of filtration on $\mathcal{F}$. Thus, it defines a function (in fact, a semigroup homomorphism)
$$\AC: \mrm{K}_+\Coh^K(\cN_{\theta}) \to \AC(\fg,K).$$
Now let $M$ be a finite-length $(\fg,K)$-module equipped with a $K$-invariant good filtration $F_{\bullet}M$. Then $\gr^F(M)$ has support contained in $\cN_{\theta}$. So $\gr^F(M)$ determines a class $[\gr^F(M)]$ of $\mrm{K}_+\Coh^K(\cN_{\theta})$. This class is independent of the choice of good filtration (\cite[Proposition 2.2]{Vogan1991}). So the passage from $M$ to $[\gr^F(M)]$ defines a function (in fact, a semigroup homomorphism)
$$\gr: \mrm{K}_+\Mod_{fl}(\fg,K) \to \mrm{K}_+\Coh^K(\cN_{\theta}).$$
If $M \in \Mod_{fl}(\fg,K)$, we define the \emph{associated cycle} of $M$ to be the associated cycle of $[\gr^F(M)]$
$$\AC(M) := \AC([\gr^F(M)])$$
We will need the following facts about associated varieties of irreducible $(\fg,K)$-modules. The first is closely related to Lemma \ref{lem:support}.

\begin{theorem}[{\cite[Theorem 8.4]{Vogan1991}}]\label{thm:V vs AV HC}
Let $M$ be an irreducible $(\fg,K)$-module and let $I=\Ann(M)$, a primitive ideal in $U(\fg)$. Let $\OO$ be the open dense $G$-orbit in $V(I)$. Then the maximal $K$-orbits in $\AV(M)$ belong to $\OO \cap (\fg/\fk)^*$. In particular
$$\dim \AV(M) = \frac{1}{2} \dim \OO.$$
\end{theorem}

\begin{theorem}[{\cite[Theorem 4.6]{Vogan1991}}]\label{thm:AVirred}
Let $M$ be an irreducible $(\fg,K)$-module and suppose there is a maximal $K$-orbit $\OO_{\theta,i}$ in $\AV(M)$ such that
$$\codim(\partial \OO_{\theta,i}, \overline{\OO}_{\theta,i}) \geq 2.$$
Then $\AV(M)$ is irreducible, i.e. $\AV(M) = \overline{\OO}_{\theta,i}$. 
\end{theorem}

\subsection{$\cW$-algebras}\label{subsec:W}

Let $\OO \subset \fg^*$ be a nilpotent co-adjoint orbit. Choose a nilpotent element $e \in \fg$ such that $\chi:=(e,\cdot) \in \OO$, where $(\cdot, \cdot)$ is the (possibly degenerate) Killing form on $\fg$. Choose also an $\mathfrak{sl}_2$-triple $(e,f,h) \in \fg^3$ with $f$ nilpotent and $h$ semisimple. Then the \emph{Slodowy slice} associated to $\OO$ is the affine subspace $S$ of $\fg^*$ defined by
$$S = \chi + (\fg^*)^f.$$
It is known that $S \cap \OO = \{\chi\}$ and that this intersection is transverse, see \cite[\S 2]{GanGinzburg}.

The ring of regular functions $\CC[S]$ has a number of important structures. First, the Poisson bracket on $S(\fg) = \CC[\fg^*]$ induces a Poisson bracket on $\CC[S]$, see \cite[\S 3]{GanGinzburg}. Second, the reductive group $R$ stabilizes $S$ and thus acts on $\CC[S]$, see \cite[\S 2]{LosevICM}. 
This action is Hamiltonian---its moment map is the restriction to $S$ of the natural projection $\fg^* \to \mathfrak{r}^*$. Lastly, there is a positive algebra grading on $\CC[S]$. To define it, consider the co-character $\gamma: \CC^{\times} \to G$ corresponding to $h$, and let $\CC^{\times}$ act on $\fg^*$ by
$$z \cdot \zeta = z^{-2} \Ad^*(\gamma(z))(\zeta), \qquad z \in \CC^{\times}, \ \zeta \in \fg^*.$$
This action stabilizes $S$, fixes $e$, and contracts the former onto the latter. Thus, it defines a positive grading on $\CC[S]$, see \cite[Sec 4]{GanGinzburg}.
The $\cW$-algebra associated to $\OO$ (or $\chi$ or $e$) is a certain filtered associative algebra $\cW$ such that $\gr(\cW) \cong \CC[S]$ (as graded Poisson algebras). For a precise definition of $\cW$, we refer the reader to \cite{Premet2002} or \cite[\S 2.3]{LosevICM}.

Consider the following sets
\begin{align*}
    \mathrm{Prim}_{\overline{\OO}}(U(\fg)) &:= \{\text{primitive ideals $I\subset U(\fg)$ such that $V(I) \subseteq \overline{\OO}$}\}\\
    \mathrm{Id}_{fin}(\cW) &:= \{\text{ideals in $\cW$ of finite codimension}\}
\end{align*}
In \cite[\S 3.4]{Losev3}, Losev defines a map
$$(\bullet)_{\ddagger}: \mathrm{Prim}_{\overline{\OO}}(U(\fg)) \to \mathrm{Id}_{fin}(\cW)$$

\begin{prop}[{\cite[Thm 1.2.2]{Losev3}}]\label{prop:propsofddagger}
The map $(\bullet)_{\ddagger}$ enjoys the following properties:
\begin{itemize}
    \item[(i)] Let $I \in \mathrm{Prim}_{\overline{\OO}}(U(\fg))$. Then
    $$I_{\ddagger} = \cW \iff V(I) \neq \overline{\OO}.$$
    \item[(ii)] Let $I \in \mathrm{Prim}_{\overline{\OO}}(U(\fg))$ and assume $V(I) = \overline{\OO}$. Then $m_{\overline{\OO}}(I) = \codim_{\cW}(I_{\ddagger})$.
    \item[(iii)] Let $I \in \mathrm{Prim}_{\overline{\OO}}(U(\fg))$. Then $I \subseteq (I_{\ddagger})^{\ddagger}$.
\end{itemize}
\end{prop}

Now fix $G_{\RR}$,$K_{\RR}$,$K$, and $\theta$ as in \S \ref{subsec:HCmodules}. Let $\fp \subset \fg$ denote the $-1$-eigenspace of $d\theta$ on $\fg$, so that $\fg = \fk \oplus \fp$. Suppose $\OO \cap (\fg/\fk)^* \neq \emptyset$ and choose a $K$-orbit $\OO_{\theta}$ in this intersection. Assume for the remainder of this subsection that
$$\codim(\partial \OO_{\theta}, \overline{\OO}_{\theta}) \geq 2.$$
Choose an $\mathfrak{sl}_2$-triple $(e,f,h)$ such that $(e,\cdot) \in \OO_{\theta}$, $f \in \fp$, and $h \in \fk$. Then $R = Z_G(e,f,h)$ is stable under $\theta$, and $R \cap K$ is a maximal reductive subgroup of $Z_K(e)$. A \emph{$(\cW,R \cap K)$-module} is a $\cW$-module $M$ equipped with a locally finite $R \cap K$-action such that
\begin{itemize}
    \item The action map $\cW \otimes M \to M$ is $R \cap K$-equivariant;
    \item The differential of the $R \cap K$-action on $M$ coincides with the $\cW$-action on $M$, restricted to $\mathfrak{r} \cap \fk \subset \mathcal{W}$.
\end{itemize}

Consider the following categories
\begin{align*}
    \Mod_{\overline{\OO}_{\theta}}(\fg,K) &:= \{\text{finitely-generated $(\fg,K)$-modules $M$ such that $\AV(M) \subseteq \overline{\OO}_{\theta}$}\}\\
    \Mod_{fin}(\cW,R \cap K) &:= \{\text{finite-dimensional Harish-Chandra $(\cW, R \cap K)$-modules}\}
\end{align*}
In \cite[\S 6]{Losev2014}, Losev defines a functor
$$(\bullet)_{\dagger}: \Mod_{\overline{\OO}_{\theta}}(\fg,K) \to \Mod_{fin}(\cW,R \cap K).$$
\begin{prop}[{\cite[\S 6.1]{Losev2014}, \cite[\S\S 3.3,3.4,4.4]{Losev2011}}]\label{prop:propsofdagger}
The functor $(\bullet)_{\dagger}$ enjoys the following properties:
\begin{itemize}
    \item[(i)] $(\bullet)_{\dagger}$ is exact.
    \item[(ii)] $(\bullet)_{\dagger}$ is full.
    \item[(iii)] The essential image of $(\bullet)_{\dagger}$ is closed under taking arbitrary subquotients.
    \item[(iv)] Let $M \in \Mod_{\overline{\OO}_{\theta}}(\fg,K)$. Then 
    $$M_{\dagger} = 0 \iff \AV(M) \neq \overline{\OO}_{\theta}.$$
    \item[(v)] Let $M \in \Mod_{\overline{\OO}_{\theta}}(\fg,K)$. Then 
    $$\Ann(M)_{\ddagger} \subseteq \Ann(M_{\dagger}).$$
    \item[(vi)] Let $M \in \Mod_{fg}(\fg,K)$ and assume that $\AV(M) = \overline{\OO}_{\theta}$. Let $[\mathcal{V}]$ denote the virtual $K$-equivariant vector bundle on $\OO_{\theta}$ determined by the virtual $R \cap K$-representation $M_{\dagger}$. Then
    $$\AC(M) = [\mathcal{V}] \cdot \OO_{\theta}.$$
\end{itemize}
\end{prop}

\subsection{Unipotent representations: definitions}\label{subsec:unipotent}

Let $\OO \subset \fg^*$ be a nilpotent co-adjoint orbit and let $\widetilde{\OO}$ be a finite connected $G$-equivariant cover of $\OO$ (a \emph{nilpotent cover}, for short). In \cite[Section 6]{LMBM}, the authors associate to $\widetilde{\OO}$ a certain maximal ideal $I(\widetilde{\OO}) \subset U(\fg)$ called the \emph{unipotent ideal} attached to $\widetilde{\OO}$. We will now briefly recall the construction of this ideal and some of its basic properties. For additional details, we refer the reader to \cite[\S\S 4 and 6]{LMBM}.

The ring of regular functions $\CC[\widetilde{\OO}]$ has the structure of a finitely-generated graded Poisson algebra (see \cite[Lemma 2.5]{LosevHC}) with Poisson bracket of degree $-2$. A filtered quantization of $\CC[\widetilde{\OO}]$ is by definition a pair $(\cA,\tau)$ consisting of a filtered associative algebra $\cA=\bigcup_{i \geq 0} \cA_i$ such that
$$[\cA_i, \cA_j] \subseteq \cA_{i+j-2}$$
and an isomorphism $\tau: \gr(\cA) \xrightarrow{\sim} \CC[\widetilde{\OO}]$ of graded Poisson algebras. The set of isomorphism classes of filtered quantizations of $\CC[\widetilde{\OO}]$ is canonically parameterized by points in a complex vector space $\fh_{\widetilde{\OO}}$ (\cite[Theorem 3.4]{Losev4}). For each filtered quantization $(\cA,\tau)$ of $\CC[\widetilde{\OO}]$ there is a unique filtered algebra homomorphism $\Phi: U(\fg) \to \cA$ such that $\tau \circ \gr(\Phi): \CC[\fg^*]=S(\fg) \to \CC[\widetilde{\OO}]$ coincides with the algebra map induced by the natural map $\widetilde{\OO} \to \fg^*$ (\cite[Lemma 4.11.2]{LMBM}).

\begin{definition}[Definition 6.0.1, \cite{LMBM}]\label{def:unipotentideal}
The \emph{canonical quantization} of $\CC[\widetilde{\OO}]$ is the filtered quantization $(\cA_0,\tau_0)$ corresponding to the point $0 \in \fh_{\widetilde{\OO}}$. The \emph{unipotent ideal} attached to $\widetilde{\OO}$ is the kernel of the quantum co-moment map to the canonical quantization
$$I(\widetilde{\OO}) := \ker{(\Phi: U(\fg) \to \cA_0)}$$
\end{definition}

Now let $M \subset G$ be a Levi subgroup of $G$ and let $\OO_M \subset \fm^*$ be a nilpotent co-adjoint $M$-orbit. Choose a parabolic subgroup $Q \subset G$ with Levi factor $M$. The annihilator of $\fq$ in $\fg^*$ is a $Q$-stable subspace $\fq^{\perp} \subset \fg^*$. Form the $G$-equivariant fiber bundle $G \times^Q (\overline{\mathbb{O}}_M \times \mathfrak{q}^{\perp})$ over the partial flag variety $G/Q$. There is a proper $G$-equivariant map
$$G \times^Q (\overline{\mathbb{O}}_M \times \mathfrak{q}^{\perp}) \to \mathfrak{g}^* \qquad (g,\xi) \mapsto g\cdot \xi$$
The image of this map is the closure of a nilpotent co-adjoint $G$-orbit, which we denote by $\mathrm{Ind}^G_M\mathbb{O}_M \subset \fg^*$. This defines a correspondence
$$\mathrm{Ind}^G_M: \{\text{nilpotent } M\text{-orbits}\} \to \{\text{nilpotent } G\text{-orbits}\}$$
called \emph{Lusztig-Spaltenstein} induction. A nilpotent orbit is said to be \emph{rigid} if it cannot be obtained by induction from a proper Levi subgroup. 

Now let $\widetilde{\OO}_M \to \OO_M$ be a finite connected $M$-equivariant cover. Consider the affine variety $\widetilde{X}_M := \Spec(\CC[\widetilde{\OO}_M])$. There is an $M$-action on $\widetilde{X}_M$ and a finite $M$-equivariant surjection $\widetilde{X}_M \to \overline{\mathbb{O}}_M$. Consider the composite map
$$G \times^Q (\widetilde{X}_M \times \mathfrak{q}^{\perp}) \to G \times^Q (\overline{\mathbb{O}}_M \times \mathfrak{q}^{\perp}) \to \fg^*.$$
The image of this map is the closure of $\Ind^G_M \OO_M$, and the preimage of this $G$-orbit is a finite connected $G$-equivariant cover, which we denote by $\mathrm{Bind}^G_M\widetilde{\OO}_M$. This defines a correspondence
$$\mathrm{Bind}^G_M: \{\text{nilpotent covers for $M$}\} \to \{\text{nilpotent covers for $G$}\}$$
called \emph{birational induction}. A nilpotent cover is said to be \emph{birationally rigid} if it cannot be obtained via birational induction from a proper Levi subgroup. Of course, every rigid orbit is birationally rigid (when regarded as a cover), but there are non-rigid orbits which are birationally rigid.

We collect in the proposition below some of the basic properties of unipotent ideals.
\begin{prop}\label{prop:propsofunipotentideals}
Let $I(\widetilde{\OO}) \subset U(\fg)$ be the unipotent ideal attached to a nilpotent cover $\widetilde{\OO}$. Then the following are true:
\begin{itemize}
    \item[(i)] $I(\widetilde{\OO})$ is a maximal (and hence, primitive) ideal in $U(\fg)$;
    \item[(ii)] $V(I(\widetilde{\OO})) = \overline{\OO}$;
    \item[(iii)]  $m_{\overline{\OO}}(I(\widetilde{\OO}))=1$ if the covering map $\widetilde{\OO} \to \OO$ is Galois;
    \item[(iv)] The infinitesimal character of $I(\widetilde{\OO})$ is real;
    \item[(v)] $I(\widetilde{\OO})$ is very weakly unipotent (Definition \ref{def:weaklyunipotentideal}).
\end{itemize}
\end{prop}

\begin{proof}
(i) is \cite[Theorem 5.0.1]{MBM}. (ii) is \cite[Proposition 6.1.2(i)]{LMBM}. (iii) follows from \cite[Proposition 6.1.2(iii)]{LMBM} and \cite[Proposition 5.1.1]{LMBM}. For (iv), we argue as follows. Choose a Levi subgroup $L \subset G$ and a birationally rigid nilpotent cover $\widetilde{\OO}_L$ such that $\widetilde{\OO} = \mathrm{Bind}^G_L \widetilde{\OO}_L$. Then choose a Levi subgroup $M \subset L$ and a birationally rigid nilpotent orbit $\OO_M$ such that $\OO_L = \Ind^L_M \OO_M$. By \cite[Proposition 8.1.3]{LMBM}, there is an equality in $\fh^*/W$
$$\lambda(\widetilde{\OO}) = \lambda(\OO_M) + \delta$$
where $\lambda(\widetilde{\OO}) \in \fh^*/W$ (resp. $\lambda(\OO_M) \in \fh^*/W_M$) denotes the infinitesimal character of $I(\widetilde{\OO})$ (resp. $I(\OO_M)$) and $\delta \in (\fh/[\fm,\fm] \cap \fh)^*$ is the `barycenter parameter' defined in \cite[Section 8.1]{LMBM}. It is clear from its construction that $\delta \in \fh_{\RR}^*$. Thus, we can reduce (iv) to the case of birationally rigid orbits. For such orbits, (iv) follows from the explicit formulas in \cite[Sections 8.2, 8.4]{LMBM}. The proof of (v) is relegated to Appendix \ref{sec:appendix}, in particular Propositions \ref{prop:weaklyunipotentA}, \ref{prop:weaklyunipotentBCD}, and \ref{prop:weaklyunipotentexceptional}. 
\end{proof}

We are now prepared to define our unipotent representations.

\begin{definition}\label{def:birigidunipotent}
Let $\OO \subset \cN$ be a birationally rigid nilpotent $G$-orbit. A \emph{birationally rigid unipotent representation} (attached to $\OO$) is an irreducible $(\fg,K)$-module $M$ such that $\mathrm{Ann}(M) = I(\OO)$. We write $\mathrm{Unip}_{\OO}(\fg,K)$ for the set of equivalence classes of birationally rigid unipotent representations attached to $\OO$.
\end{definition}

\begin{definition}\label{def:smallbdryunipotent}
Let $\OO_{\theta} \subset \cN_{\theta}$ be a nilpotent $K$-orbit such that 
$$\codim(\partial \OO_{\theta},\overline{\OO}_{\theta}) \geq 2,$$
and let $\OO = G \cdot \OO_{\theta}$, a nilpotent orbit in $\fg^*$. A \emph{small boundary unipotent representation} (attached to $\OO_{\theta}$) is an irreducible $(\fg,K)$-module $M$ such that $\mathrm{Ann}(M) = I(\OO)$, and $\mathrm{AV}(M) = \overline{\OO}_{\theta}$. We write $\mathrm{Unip}_{\OO_{\theta}}(\fg,K)$ for the set of equivalence classes of small boundary unipotent representations attached to $\OO_{\theta}$. 
\end{definition}

\begin{example}
The following are examples of small boundary unipotents:
\begin{itemize}
    \item  Let $\OO \subset \cN$ be a nilpotent $G$-orbit such that
    $$\codim(\partial \OO, \overline{\OO}) \geq 4,$$
    (for example, $\OO$ might be the minimal nilpotent orbit for a simple group not of type $A_1$). Let $M$ be an irreducible $(\fg,K)$-module such that $\mathrm{Ann}(M) = I(\OO)$. Then $M$ is a small boundary unipotent attached to $\OO$. 
    \item Suppose $G_{\RR}$ is a complex Lie group (so that $G$ is identified with $G_{\RR} \times G_{\RR}$ and $K$ is identified with a diagonal copy of $G_{\RR}$). Let $\OO \subset \cN$ be a nilpotent $G$-orbit, and let $M$ be an irreducible $(\fg,K)$-module such that $\mathrm{Ann}(M) = I(\OO)$. Then $M$ is a small boundary unipotent attached to $\OO$.
\end{itemize}
\end{example}

\subsection{Unipotent representations: associated cycles and Hermitian forms} \label{subsec:unipotent no hodge}

In this section, we will prove some facts about unipotent representations which \emph{do not} rely on the Cohen-Macaulay property for very weakly unipotent modules. First, we examine the associated cycles of unipotent representations.

Recall the semigroup $\AC(\fg,K)$ defined in \S \ref{subsec:AC}.

\begin{definition}\label{def:irreducibleAC}
An associated cycle $\sum_i [\mathcal{V}_i] \cdot \OO_i \in \AC(\fg,K)$ is \emph{irreducible} if
\end{definition}
\begin{itemize}
    \item There is a unique $i$ such that $[\mathcal{V}_i] \neq 0$.
    \item $[\mathcal{V}_i]$ is the class of an irreducible $K$-equivariant vector bundle.
\end{itemize}

\begin{prop}\label{prop:smallboundaryAC}
Let $\OO_{\theta} \subset \cN_{\theta}$ be a nilpotent $K$-orbit such that
$$\codim(\partial \OO_{\theta}, \overline{\OO}_{\theta}) \geq 2.$$
Then the following are true:
\begin{itemize}
    \item[(i)] If $M \in \mathrm{Unip}_{\OO_{\theta}}(\fg,K)$, then $\AC(M)$ is irreducible.
    \item[(ii)] The map 
    $$\AC: \mathrm{Unip}_{\OO_{\theta}}(\fg,K) \to \AC(\fg,K)$$
    is injective. 
\end{itemize}
\end{prop}

\begin{proof}
Fix the notation of \S \ref{subsec:W}, e.g. $\cW$, $R$, $\Mod_{fin}(\cW,R \cap K)$, $\Mod_{\overline{\OO}_{\theta}}(\fg,K)$, and so on. Recall the map
    $$(\bullet)_{\ddagger}: \mathrm{Prim}_{\overline{\OO}}(U(\fg)) \to \mathrm{Id}_{fin}(\cW)$$
    and the functor
    $$(\bullet)_{\dagger}: \Mod_{\overline{\OO}_{\theta}}(\fg,K) \to \Mod_{fin}(\cW, R \cap K).$$
    Let $\OO = G \cdot \OO_{\theta}$ and let $I=I(\OO)$. By Proposition \ref{prop:propsofunipotentideals}(iii), $m_{\overline{\OO}}(I)=1$. So by Proposition \ref{prop:propsofddagger}(ii), we have $\codim_{\cW}(I_{\ddagger})=1$. 

    Let $M \in \mathrm{Unip}_{\OO_{\theta}}(\fg,K)$ and let $\AC(M) = [\mathcal{V}] \cdot \OO_{\theta}$. Let $\rho$ be the (virtual) representation of $R \cap K$ corresponding to $[\mathcal{V}]$. For (i), we wish to show that $\rho$ is irreducible. 
    
    It is a formal consequence of (i), (ii), and (iii) of Proposition \ref{prop:propsofdagger} that $(\bullet)_{\dagger}$ takes irreducibles to irreducibles. So $M_{\dagger}$ is an irreducible object in $\Mod_{fin}(\cW,R \cap K)$. By Proposition \ref{prop:propsofdagger}(v), we have that $I_{\ddagger} \subseteq \Ann(M_{\dagger})$. But since $I_{\ddagger} \subset \cW$ is of codimension 1, $\cW$ acts by scalars on $X_{\dagger}$. So in fact $M_{\dagger}$ is irreducible as a representation of $R \cap K$. By property (vi) of $(\bullet)_{\dagger}$, we have $\rho \cong X_{\dagger}$ as $R \cap K$-representations. So $\rho$ is irreducible, as asserted.

    For (ii), suppose $M,N \in \mathrm{Unip}_{\OO_{\theta}}(\fg,K)$ and $\AC(M)=\AC(N)$. Then by Proposition \ref{prop:propsofdagger}(vi), there is an isomorphism of $R \cap K$-representations $\phi: X_{\dagger} \cong Y_{\dagger}$. Recall that $\cW$ acts on $M_{\dagger}$, $N_{\dagger}$ by a character (which is determined by $I$). So $\phi$ is an isomorphism in $\Mod_{fin}(\cW,R \cap K)$. Since $(\bullet)_{\dagger}$ is full, there is a morphism $\tilde{\phi}: M \to N$ such that $(\tilde{\phi})_{\dagger}=\phi$. And since $\phi$ is an isomorphism, $(\ker{\tilde{\phi}})_{\dagger}=(\coker{\tilde{\phi}})_{\dagger}=0$. By Proposition \ref{prop:propsofdagger}(iv), this implies $\AV(\ker{\tilde{\phi}}), \AV(\coker{\tilde{\phi}}) \subseteq \partial \OO_{\theta}$. We will show that this implies $\ker{\tilde{\phi}}=\coker{\tilde{\phi}}=0$. Clearly $I$ annihilates $\ker{\tilde{\phi}}$. But if $\Ann(\ker{\tilde{\phi}}) = I$, then $\dim \AV(\ker{\tilde{\phi}}) = \dim \OO_{\theta}$ by Theorem \ref{thm:V vs AV HC}. This contradicts the inclusion $\AV(\ker{\tilde{\phi}}) \subseteq \partial \OO_{\theta}$. So $I \subsetneq \Ann(\ker{\tilde{\phi}})$. Since $I$ is maximal (Proposition \ref{prop:propsofunipotentideals}(i)), this forces $\Ann(\ker{\tilde{\phi}})=U(\fg)$, i.e. $\ker{\tilde{\phi}}=0$. We have $\coker{\tilde{\phi}}=0$ by an identical argument. Thus, $\tilde{\phi}: M \to N$ is an isomorphism. This completes the proof of (ii).
\end{proof}

\begin{prop}[\cite{LosevYu}]\label{prop:birigidAC}
Let $\OO \subset \cN$ be a birationally rigid nilpotent $G$-orbit. Then the following are true:
\begin{itemize}
    \item[(i)] If $M \in \mathrm{Unip}_{\OO}(\fg,K)$, then $\AC(M)$ is irreducible.
    \item[(ii)] The map 
    $$\AC: \mathrm{Unip}_{\OO}(\fg,K) \to \AC(\fg,K)$$
    is injective. 
\end{itemize}
\end{prop}

\begin{proof}
First suppose $\codim(\partial \OO, \overline{\OO}) \geq 4$. Then every $K$-orbit $\OO_{\theta}$ in $\OO \cap (\fg/\fk)^*$  satisfies $\codim(\partial \OO_{\theta},\overline{\OO}_{\theta}) \geq 2$. So by Theorems \ref{thm:V vs AV HC} and \ref{thm:AVirred}, $\mathrm{Unip}_{\OO}(\fg,K)$ is the disjoint union of the sets $\mathrm{Unip}_{\OO_{\theta}}(\fg,K)$ as $\OO_{\theta}$ ranges over all $K$-orbits in $\OO \cap (\fg/\fk)^*$. Now (i) and (ii) are immediate from Proposition \ref{prop:smallboundaryAC}.

In general, the statement easily reduces to the case where $\mf{g}$ is simple. In this setting, there are only six birationally rigid orbits which do not satisfy the condition $\codim(\partial \OO, \overline{\OO}) \geq 4$ (they are: $\widetilde{A}_1$ in type $G_2$, $\widetilde{A}_2+A_1$ in type $F_4$, $(A_3+A_1)'$ in type $E_7$ and $A_3+A_1$, $A_5+A_1$, and $D_5(a_1)+A_2$ in type $E_8$). For these six orbits, we appeal to \cite[Appendix A]{LosevYu}.
\end{proof}
We now turn our attention to the existence of Hermitian forms. 

\begin{prop}\label{prop:unipotentthetastable}
Let $\OO \subset \cN$ be a nilpotent co-adjoint $G$-orbit such that
$$\OO \cap (\fg/\fk)^* \neq \emptyset.$$
Then the unipotent ideal $I(\OO)$ is preserved by the involution $\theta: U(\fg) \to U(\fg)$. 
\end{prop}

\begin{proof}
Under the condition of the proposition, $\theta \OO = \OO$. Hence, $\theta$ induces a graded Poisson involution of $\CC[\OO]$, which we will again denote by $\theta$. Let $(\cA,\tau)$ denote the canonical quantization of $\CC[\OO]$ (here $\tau: \gr(\cA) \xrightarrow{\sim} \CC[\OO]$ is a fixed isomorphism of graded Poisson algebras). By \cite[Proposition 5.1.1]{LMBM}, there is a filtered algebra involution $\widetilde{\theta}$ of $\cA$ such that $\tau \circ \gr(\widetilde{\theta}) = \theta \circ \tau$. It follows that the filtered quantization $(\cA, \theta \circ \tau)$ of $\CC[\OO]$ is isomorphic to $(\cA,\tau)$ (the isomorphism is implemented by the algebra isomorphism $\widetilde{\theta}: \cA \to \cA$). Now let $\Phi: U(\fg) \to \cA$ denote the quantum co-moment map for $(\cA,\tau)$. Then clearly $\Phi \circ \theta: U(\fg) \to \cA$ is a quantum co-moment map for $(\cA,\theta \circ \tau)$. Since $(\cA,\tau)$ and $(\cA,\theta \circ \tau)$ are isomorphic quantizations, we must therefore have $\widetilde{\theta} \circ \Phi = \Phi \circ \theta$ by the uniqueness of quantum co-moment maps (see \cite[Lemma 4.11.2]{LMBM}). Taking kernels, we obtain
$$I(\OO) = \ker{(\widetilde{\theta} \circ \Phi)} = \ker{(\Phi \circ \theta)} = \theta I(\OO).$$
This completes the proof.
\end{proof}

\begin{prop}\label{prop:ACXtheta}
Let $M \in \Mod_{fl}(\fg,K)$. Then 
$$\AC(M) = \AC(\theta^*M).$$
\end{prop}

\begin{proof}
Let $\theta: S(\fg/\fk) \to S(\fg/\fk)$ denote the involution defined by $\xi \mapsto -\xi$ for $\xi \in \fg/\fk$. Twisting by $\theta$ defines an involution $\theta^*$ of the category $\Coh^K(\cN_{\theta})$, and hence of the semigroup $\mrm{K}_+\Coh^K(\cN_{\theta})$. It is easy to check that the following diagram commutes:

\begin{center}
\begin{tikzcd}
\mrm{K}_+\Mod_{fl}(\fg,K) \ar[r,"\theta^*"] \ar[d,"\gr"] & \mrm{K}_+\Mod_{fl}(\fg,K) \ar[d,"\gr"]\\
\mrm{K}_+\Coh^K(\cN_{\theta}) \ar[r,"\theta^*"] & \mrm{K}_+\Coh^K(\cN_{\theta})
\end{tikzcd}
\end{center}

Now let $\OO_{\theta} \subset \cN_{\theta}$ be a $K$-orbit and let $x \in \OO_{\theta}$. Choose $k \in K$ such that $\Ad^*(k)x = -x$. Then $\Ad(k)$ is an automorphism of the centralizer $K_x$. If $\rho$ is a finite-dimensional representation of $K_x$, define $\theta^*\rho = \rho \circ \Ad(k)$. Since $k$ is well-defined up to multiplication by $K_x$, $\theta^*\rho$ is well-defined as a virtual representation of $K_x$. Thus, we obtain an involution $\theta^*$ of the semigroup $\mrm{K}_+ \Rep (K_x)$, which corresponds via the equivalence $\Rep(K_x) \cong \Coh^K(\OO_{\theta})$ to the pullback along $\theta \colon \mb{O}_\theta \to \mb{O}_\theta$. We can then define an involution $\theta^*$ of $\AC(\fg,K)$ via
$$\theta^*(\sum_i [\cV_i] \cdot \OO_i) = \sum_i \theta^*[\mathcal{V}_i] \cdot \OO_i.$$
It is easy to check that the following diagram commutes
\begin{center}
\begin{tikzcd}
\mrm{K}_+\Coh^K(\cN_{\theta}) \ar[r,"\theta^*"] \ar[d,"\AC"] & \mrm{K}_+\Coh^K(\cN_{\theta}) \ar[d, "\AC"] \\
\AC(\fg,K) \ar[r,"\theta^*"] & \AC(\fg,K)
\end{tikzcd}
\end{center}
To complete the proof, it suffices to show that the involution $\theta^*$ of $\AC(\fg,K)$ is trivial---or equivalently, that the involution $\theta^*$ of $\mrm{K}\Rep(K_x)$ is trivial. Choose an $\mathfrak{sl}_2$-triple $(e,f,h)$ such that $x=(e,\cdot)$ and $h \in \mf{k}$, and let $\gamma\colon  \CC^{\times} \to G$ be the co-character corresponding to $h$. Then $z \mapsto \Ad(\gamma(z))$ defines an action of $\CC^{\times}$ on $K_x$ by algebraic group automorphisms. Any such action induces the trivial action on the set of irreducible  characters. So $\rho \circ \Ad(\gamma(z)) = \rho$ as virtual representations for any $z \in \CC^{\times}$. But $\Ad^*(\gamma(i))x = i^2x = -x$, so $\theta^*\rho=\rho \circ \Ad(\gamma(i)) = \rho$.
\end{proof}

\begin{cor}\label{cor:unipotentHermitian}
Let $M$ be a birigid or small-boundary unipotent representation. Then $M$ is Hermitian.
\end{cor}

\begin{proof}
Let $\OO \subset \cN$ be a birationally rigid nilpotent $G$-orbit, and let $M \in \mathrm{Unip}_{\OO}(\fg,K)$. By Theorem \ref{thm:V vs AV HC}, the maximal $K$-orbits in $\AV(M)$ are subsets of $\OO \cap (\fg/\fk)^*$. In particular, $\OO \cap (\fg/\fk)^* \neq \emptyset$. So by Proposition \ref{prop:unipotentthetastable}, $\Ann(\theta^*M) = \theta \Ann(M) = \theta I(\OO) = I(\OO)$, i.e. $\theta^*M \in \mathrm{Unip}_{\OO}(\fg,K)$. And by Proposition \ref{prop:ACXtheta}, $\AC(M) = \AC(\theta^*M)$. So Proposition \ref{prop:birigidAC}(ii) implies that $M \cong \theta^*M$. Also, the infinitesimal character of $M$ is real by Proposition \ref{prop:propsofunipotentideals}(iv). So $M$ is Hermitian by Proposition \ref{prop:Hermitiancriterion}. The proof for small-boundary unipotents is completely analagous.
\end{proof}

\subsection{Unipotent representations: $K$-types and unitarity} \label{subsec:unipotent with hodge}

In this subsection, we prove our main results about unipotent representations. Both are applications of the Cohen-Macaulay property for very weakly unipotent modules (Theorem \ref{thm:CM}).

Our first main result is that all small-boundary unipotent representations are quantizations of vector bundles on nilpotent $K$-orbits. This resolves a longstanding conjecture of Vogan (see \cite[Conjecture 12.1]{Vogan1991} for the original conjecture, and \cite{MB2022}, \cite{LeungYu}, \cite{LMBM} for recent partial results). 

\begin{thm}\label{thm:Kspectraunipotent}
Let $M$ be a small-boundary unipotent representation (Definition \ref{def:smallbdryunipotent}), so that $\AV(M) = \overline{\OO}_{\theta}$ for a small-boundary $K$-orbit $\OO_{\theta} \subset \cN_{\theta}$. Then there is an irreducible $K$-equivariant vector bundle $\mathcal{V}$ on $\OO_{\theta}$ such that 
$$\Gr^{F^H} M \cong \Gamma(\OO_{\theta},\mathcal{V})$$
as $(S(\mf{g}), K)$-modules. In particular, there is an element $x \in \OO_{\theta}$ and an irreducible finite-dimensional representation $\rho$ of the stabilizer $K_{x}$ of $x$ such that
$$M \cong \Ind^K_{K_{x}} \rho$$
as representations of $K$.
\end{thm}

\begin{proof}
Let $\mathcal{F} = \gr^{F^H}(M)$. By Proposition \ref{prop:propsofunipotentideals}(i), $\Ann(M)$ is a maximal ideal, and by Proposition \ref{prop:propsofunipotentideals}(v), it is very weakly unipotent. So by Theorem \ref{thm:CM}, $\mathcal{F}$ is a Cohen-Macaulay $S(\fg)$-module. 

Choose a $K$-invariant open subset $U \subset \fg^*$ such that $U \cap \overline{\OO}_{\theta} = \OO_{\theta}$, and let $j\colon U \hookrightarrow \fg^*$ be the inclusion. Since $\mathcal{F}$ is Cohen-Macaulay and $\codim(\partial \OO_{\theta},\overline{\OO}_{\theta})\geq 2$, the natural map of $K$-equivariant coherent sheaves $\mathcal{F} \to j_*j^*\mathcal{F}$ is an isomorphism (see e.g. \cite[Proposition 1.11]{Hartshorne1992}). 

Choose a finite filtration $0 = \mathcal{G}_0 \subsetneq \mathcal{G}_1 \subsetneq ... \subsetneq \mathcal{G}_n = j^*\mathcal{F}$ by $K$-invariant subsheaves such that each successive quotient is scheme-theoretically supported on $\OO_{\theta}$. Then by definition
$$\AC(M) = (\sum_{i=1}^n [\mathcal{G}_i/\mathcal{G}_{i-1}]) \cdot \OO_{\theta}$$
By Proposition \ref{prop:smallboundaryAC}(i), $\AC(M)$ is irreducible. So $n=1$ and $\mathcal{G}_1=j^*\mathcal{F}$ is (the pushforward to $U$ of) an irreducible $K$-equivariant vector bundle $\mathcal{V}$ on $\OO_{\theta}$. Now as $(S(\mf{g}), K)$-modules,
$$\Gr^{F^H}M \cong \Gamma(\fg^*,\mathcal{F}) \cong \Gamma(\fg^*,j_*j^*\mathcal{F}) \cong \Gamma(U,j^*\mathcal{F}) \cong \Gamma(\OO_{\theta},\mathcal{V}).$$
Finally, choose $x \in \OO_{\theta}$ and let $\rho$ be the fiber of $\mathcal{V}$ at $x$, an irreducible representation of the isotropy group $K_{x}$. Then as representations of $K$
$$M \cong \Gr^{F^H} M \cong \Gamma(\OO_{\theta}, \mathcal{V}) \cong \Ind^K_{K_{x}} \rho.$$
\end{proof}

\begin{rmk}
The representation $\rho$ appearing in Theorem \ref{thm:Kspectraunipotent} will always be \emph{admissible} in the sense of Vogan and Duflo (see \cite[Definition 7.13]{Vogan1991}). This is an immediate consequence of \cite[Theorem 8.7]{Vogan1991}.
\end{rmk}

For our unitarity result, we will need the following lemma.

\begin{lemma}\label{lem:indecomposable}
Let $\mathcal{F}$ be a $K$-equivariant coherent sheaf on $(\fg/\fk)^*$ with support contained in $\cN_{\theta}$. Assume
\begin{itemize}
    \item[(i)] $\mathcal{F}$ is a Cohen-Macaulay $S(\fg)$-module.
    \item[(ii)] $\AC(\mathcal{F})$ is irreducible (Definition \ref{def:irreducibleAC}).
\end{itemize}
Then $\mathcal{F}$ is indecomposable as a $K$-equivariant $S(\fg/\fk)$-module.
\end{lemma}

\begin{proof}
Suppose $\mathcal{F} \cong \mathcal{F}_1 \oplus \mathcal{F}_2$ in the category of $K$-equivariant $S(\fg/\fk)$-modules. We wish to show that one of $\mathcal{F}_i$ is $0$. Let $\AC(\mathcal{F}) = [\cV] \cdot \OO_{\theta}$, $\AC(\mathcal{F}_1) = \sum_j [\cW_1^j] \cdot \OO_1^j$, and $\AC(\mathcal{F}_2) = \sum_k [\cW_2^k] \cdot \OO_2^k$ so that
$$[\cV] \cdot \OO_{\theta} = (\sum_j [\cW_1^j] \cdot \OO_1^j) + (\sum_j [\cW_2^k] \cdot \OO_2^k)$$
in the semigroup $\AC(\fg,K)$ (see \S \ref{subsec:AC} for the definition of addition). This equation implies that $\OO_1^j,\OO_2^k \subseteq \overline{\OO}_{\theta}$ for each $j$ and $k$. Suppose there are indices $j_0$ and $k_0$ such that $\OO_1^{j_0}=\OO_2^{k_0}=\OO_{\theta}$. Then $[\cV] = [\cW_1^{j_0}] + [\cW_2^{k_0}]$, contradicting the assumption that $[\cV]$ is irreducible. So either $\OO_1^j \subseteq \partial \OO_{\theta}$ for all $j$ or $\OO_2^k \subseteq \partial \OO_{\theta}$ for all $k$. Let us assume that $\OO_1^j \subseteq \partial \OO_{\theta}$ for all $j$. Then $\dim \mathrm{Supp}(\mathcal{F}_1) = \max\{\dim \OO_1^j\} < \dim \OO_{\theta} = \dim \mathrm{Supp}(\mathcal{F})$. If $\mathcal{F}_1 \neq 0$, then $\dim \mathrm{Supp}(\mathcal{F}_1) = \dim \mathrm{Supp}(\mathcal{F})$, since $\mathcal{F}_1$ is a direct summand of a Cohen-Macaulay module. So it must be the case that $\mathcal{F}_1=0$. This completes the proof.
\end{proof}

\begin{theorem}\label{thm:unitaritycriterion}
Let $M$ be an irreducible $(\fg,K)$-module such that
\begin{itemize}
    \item[(i)] $M$ is Hermitian.
    \item[(ii)] $\AC(M)$ is irreducible (Definition \ref{def:irreducibleAC}).
    \item[(iii)] $\Ann(M)$ is a maximal ideal.
    \item[(iv)] $\Ann(M)$ is very weakly unipotent (Definition \ref{def:weaklyunipotent}).
\end{itemize}
Then $M$ is unitary.
\end{theorem}

\begin{proof}
By Theorem \ref{thm:CM}, (iii) and (iv) imply that $\gr^{F^H}(M)$ is a Cohen-Macaulay $S(\fg)$-module. This, together with (ii), implies by Lemma \ref{lem:indecomposable} that $\gr^{F^H}(M)$ is indecomposable as a $K$-equivariant $S(\fg/\fk)$-module. This, together with (i), implies that $M$ is unitary by Corollary \ref{cor:indecomposable}.   
\end{proof}

\begin{cor}\label{cor:unitaryunipotent}
Let $M$ be a birationally rigid unipotent representation (Definition \ref{def:birigidunipotent}) or a small-boundary unipotent representation (Definition \ref{def:smallbdryunipotent}). Then $M$ is unitary.
\end{cor}

\begin{proof}
To apply Theorem \ref{thm:unitaritycriterion}, we must verify properties (i)-(iv):
\begin{itemize}
    \item (i) is Corollary \ref{cor:unipotentHermitian}.
    \item (ii) is Propositions \ref{prop:birigidAC}(i) and \ref{prop:smallboundaryAC}(i).
    \item (iii) is Proposition \ref{prop:propsofunipotentideals}(i).
    \item (iv) is Proposition \ref{prop:propsofunipotentideals}(v).
\end{itemize}
\end{proof}

\begin{rmk}\label{rmk:covers}
Our definitions of unipotent $(\fg,K)$-modules (Definitions \ref{def:birigidunipotent} and \ref{def:smallbdryunipotent}) require that the annihilator is a unipotent ideal attached to a nilpotent co-adjoint orbit. The restriction to orbits (rather than arbitrary nilpotent covers) is mostly for convenience. We note that our main results about unipotent representations, i.e.\ the description of their $K$-types and the proof of unitarity, also hold for interesting classes of modules annihilated by unipotent ideals attached to nilpotent covers.  For example, suppose $M$ is an irreducible $(\fg,K)$-module such that $\Ann(M) = I(\widetilde{\OO})$, where $\widetilde{\OO}$ is a nilpotent cover to which the involution $\theta\colon \fg^* \to \fg^*$ lifts. Assume that $\AV(M) = \overline{\OO}_{\theta}$ for a small-boundary $K$-orbit $\OO_{\theta} \subset \cN_{\theta}$. Then the statements of Theorem \ref{thm:Kspectraunipotent} and Corollary \ref{cor:unitaryunipotent} hold also for $M$, and the proofs are essentially the same.
\end{rmk}

Our final result concerns \emph{Harish-Chandra bimodules}.  A Harish-Chandra $U(\fg)$-bimodule is a $(\fg \times \fg, G_{\Delta})$-module (where $G_{\Delta} = G$ acts on $\fg \times \fg$ via the diagonal adjoint action); equivalently, it is a $U(\fg)$-bimodule $M$ such that the adjoint action of $\fg$ on $M$ integrates to a locally finite $G$-action. Writing $\sigma_c: G \to G$ for a compact real form of $G$, we define a real form $\sigma$ of $\fg \times \fg$ by
$$\sigma(X,Y) = (\sigma_c(Y),\sigma_c(X)), \qquad X,Y \in \fg.$$
This commutes with the Cartan involution $\theta$ defined by
$$\theta(X,Y) = (Y,X), \qquad X,Y \in \fg.$$
Since Harish-Chandra bimodules are, in particular, Harish-Chandra modules, it is sensible to talk about \emph{Hermitian} and \emph{unitary} Harish-Chandra bimodules as in \S \ref{subsec:HCmodules}. It is also sensible to talk about associated varieties and cycles for Harish-Chandra bimodules. The associated variety of a Harish-Chandra bimodule $V$ is, a priori, a union of nilpotent $G$-orbits in $\fg^* \cong (\fg \times \fg/\fg_{\Delta})^*$. However, since all co-adjoint orbits are of even dimension, Theorem \ref{thm:AVirred} implies that $\AV(M)$ is the closure of a single nilpotent orbit $\OO \subset \fg^*$, which is also the associated variety of the (left or right) annihilator of $M$.

\begin{theorem}\label{thm:complexunitary}
Let $G$ be a complex connected reductive algebraic group and let $\widetilde{\OO}$ be a finite connected $G$-equivariant cover of a nilpotent co-adjoint $G$-orbit. Let $M$ be an irreducible Harish-Chandra $U(\fg)$-bimodule annihilated on both sides by the unipotent ideal $I(\widetilde{\OO})$. Then $M$ is unitary.
\end{theorem}

\begin{proof}
{\it Step 1.} First, we reduce to the case when $\widetilde{\OO} \to \OO$ is the universal $G$-equivariant cover of $\OO$. 

Let $\widetilde{\OO}'$ be the unique maximal cover in the equivalence class of $\widetilde{\OO}$ (with respect to the equivalence relation defined in \cite[Section 6.5]{LMBM}). Choose a Levi subgroup $L \subset G$ and a birationally rigid nilpotent cover $\widetilde{\OO}_L'$ such that
$$\widetilde{\OO}' = \mathrm{Bind}^G_L \widetilde{\OO}_L',$$
and let $\widetilde{\OO}_L \to \widetilde{\OO}_L'$ be the universal $L$-equivariant cover. Let $\cA'$ denote the canonical quantization of $\CC[\widetilde{\OO}']$, regarded as a Harish-Chandra $U(\fg)$-bimodule, and let $\cA_L$ (resp. $\cA_L'$) denote the canonical quantization of $\CC[\widetilde{\OO}_L]$ (resp. $\CC[\widetilde{\OO}_L']$), regarded as a Harish-Chandra $U(\fl)$-bimodule. 

By \cite[Corollary 6.6.3]{LMBM}, $M$ is a direct summand of $\cA'$. So 
$$\cA' \text{ is unitary} \implies M \text{ is unitary}.$$
By \cite[Proposition 10.3.1]{LMBM}
$$\cA_L' \text{ is unitary} \implies \cA' \text{ is unitary}.$$
Since $\widetilde{\OO}_L'$ is birationally rigid, the covering map $\widetilde{\OO}_L \to \widetilde{\OO}_L'$ induces an almost \'etale map $\Spec(\CC[\widetilde{\OO}_L]) \to \Spec(\CC[\widetilde{\OO}'_L])$, see \cite[Corollary 7.6.1]{LMBM}. So by \cite[Proposition 5.4.5]{LMBM}, $\cA'_L$ is a direct summand of $\cA_L$. So 
$$\cA_L \text{ is unitary} \implies \cA'_L \text{ is unitary}.$$
{\it Step 2.} By Step 1, we can assume that $\widetilde{\OO} \to \OO$ is the universal $G$-equivariant cover. In particular, it is Galois. Hence, by \cite[Theorem 6.6.2]{LMBM}, $\AC(M)$ is irreducible, and the associated cycle map
$$\AC: \{M' \in \Pi(\fg \times \fg, G_{\Delta}) \mid \Ann_L(M')=\Ann_R(M')=I(\widetilde{\OO})\} \to \AC(\fg \times \fg, G_{\Delta})$$
is injective. Since the left and right annihilators of $M$ coincide and $\theta$ acts on $\fg \times \fg$ by transposing factors, it is clear that $\Ann(M)$ is $\theta$-stable. So by the argument of Corollary \ref{cor:unipotentHermitian}, $M$ is Hermitian. To deduce that $M$ is unitary, we apply Theorem \ref{thm:unitaritycriterion}. We have already verified conditions (i) and (ii). Conditions (iii) and (iv) follow from Proposition \ref{prop:propsofunipotentideals}.
\end{proof}

\appendix

\section{Weak unipotence}\label{sec:appendix}

Let $\fg$ be a complex reductive Lie algebra. In this appendix, we collect some facts regarding very weak unipotence (Definition \ref{def:weaklyunipotentideal}). Our main result is that all unipotent ideals (Definition \ref{def:unipotentideal}) are very weakly unipotent. The proof is a case-by-case calculation, involving partition combinatorics in classical types (Propositions \ref{prop:weaklyunipotentA} and \ref{prop:weaklyunipotentBCD}) and computer calculations in exceptional types (Proposition \ref{prop:weaklyunipotentexceptional}). 

Choose a Borel subalgebra $\fb \subset \fg$ and a Cartan subalgebra $\fh \subset \fb$. Let $\fg^{\vee}$ denote the Langlands dual of $\fg$. By definition, $\fg^{\vee}$ contains a distinguished Cartan subalgebra $\fh^{\vee} \subset \fg^{\vee}$ which is canonically identified with $\fh^*$. An element $\lambda \in \fh^* \simeq \fh^{\vee}$ defines two root sybsystems of $\Delta(\fh^{\vee},\fg^{\vee})$:
$$\Delta^{\vee}_{\lambda} := \{\alpha^{\vee} \in \Delta(\fh^{\vee},\fg^{\vee}) \mid \langle \alpha^{\vee},\lambda\rangle \in \ZZ\}, \qquad \Delta^{\vee}_{\lambda,0} := \{\alpha^{\vee} \in \Delta(\fh^{\vee},\fg^{\vee}) \mid \langle \alpha^{\vee},\lambda\rangle =0\}.$$
Denote the corresponding subalgebras by
$$\fg^{\vee}_{\lambda} \subset \fg^{\vee}, \qquad \fg^{\vee}_{\lambda,0} \subset \fg^{\vee}.$$
Note that $\fg^{\vee}_{\lambda,0}$ is in fact a Levi subalgebra of $\fg^{\vee}_{\lambda}$. Let $\OO^{\vee}_{\lambda}$ denote the nilpotent orbit in $(\fg^{\vee}_{\lambda})^*$, defined by
$$\OO^{\vee}_{\lambda} := \Ind^{\fg^{\vee}_{\lambda}}_{\fg^{\vee}_{\lambda,0}} \{0\}.$$
Note that for any $\gamma \in \lambda+\ZZ\Phi$, we have $\fg^{\vee}_{\gamma} = \fg^{\vee}_{\lambda}$. Moreover, the nilpotent orbit $\OO^{\vee}_{\gamma} \subset (\fg^{\vee}_{\lambda})^*$ depends only on the $W$-orbit of $\gamma$. Recall (Definition \ref{def:CandD}) the set 
\begin{equation}\label{eq:D}D^{\circ}(\lambda) = (\mathrm{Conv}(W\lambda) \setminus W\lambda) \cap (\lambda+\ZZ\Phi)\end{equation}
Let $D^{\circ}(\lambda)_+$ denote the set of dominant elements of $\fh_{\RR}^*$ which are conjugate to elements of $D^{\circ}(\lambda)$. 

\begin{definition}\label{def:lambdaweaklyunipotent}
Let $\lambda \in \fh^*_{\RR}$. We say that $\lambda$ is \emph{very weakly unipotent} if for each $\gamma \in D^{\circ}(\lambda)_+$, we have
$$\OO^{\vee}_{\lambda} \not\subseteq \OO^{\vee}_{\gamma}.$$
\end{definition}

The following lemma is probably well-known to the experts, and implicit in McGovern \cite[Theorem 5.6]{McGovern1994}. 

\begin{lemma}\label{lem:weaklyunipotentcriterion}
Let $\lambda \in \fh_{\RR}^*$ and let $I \subset U(\fg)$ be the (unique) maximal ideal of infinitesimal character $\chi_{\lambda}$. Assume that $\lambda$ is very weakly unipotent (Definition \ref{def:lambdaweaklyunipotent}). Then $I$ is very weakly unipotent (Definition \ref{def:weaklyunipotentideal}).
\end{lemma}

\begin{proof}
We begin by recalling some preliminaries on BGG category $\cO$. Let $\cO(\fg)$ denote category $\cO$ for $U(\fg)$. For any $x \in \fh^*$, let $\cO_{\chi_x}(\fg)$ denote the full subcategory of $\cO(\fg)$ consisting of $U(\fg)$-modules of generalized infinitesimal character $\chi_{x}$, so that there is a decomposition
$$\cO(\fg) = \bigoplus_{x \in \fh^*/W} \cO_{\chi_x}(\fg).$$
Write $\mathrm{pr}_{\chi_x}: \cO(\fg) \to \cO_{\chi_x}(\fg)$ for the projection functor.

Recall that the standard objects in $\cO_{\chi_x}(\fg)$ are the Verma modules $\{M(wx)\}_{w \in W}$, where $M(wx) := U(\fg) \otimes_{U(\fb)} \CC_{wx-\rho}$, and the irreducible objects are $\{L(wx)\}_{w \in W}$, where $L(wx)$ is the unique irreducible quotient of $M(wx)$. Let 
$$W_x := \{w \in W \mid wx - x \in \ZZ\Phi\}$$
be the integral Weyl group for $x$. Also recall (\cite[Theorem 4.9]{Humphreys}) that the irreducible objects $\{L(wx)\}_{w \in W_x}$ generate a block in $\cO_{\chi_x}(\fg)$, which we will denote by $\cO_x(\fg)$. 

Taking $\lambda$ to be integrally dominant, we have that the irreducible object $L(\lambda) \in \cO_{\lambda}(\fg)$ has maximal annihilator (\cite{Duflo1977}). So by Proposition \ref{prop:weakly unipotent ideal}, it suffices to show that 
$$\mathrm{pr}_{\chi_{\gamma}}(L(\lambda) \otimes F) = 0,$$
for all $\gamma \in D^{\circ}(\lambda)$ and all finite-dimensional $\fg$-modules $F$. We prove this in two steps.

Step 1.  Let $\fg_{\lambda}$ denote the Langlands dual of $\fg^{\vee}_{\lambda}$. Note that $\fg_{\lambda}$ is not in general a subalgebra of $\fg$, although its root system is contained in the root system of $\fg$. For any $\gamma \in \lambda+ \ZZ\Phi$, there is an isomorphism of Grothendieck groups
$$K_0\cO_{\gamma}(\fg) \simeq K_0\cO_{\gamma}(\fg_{\lambda}),$$
matching Verma modules $M_{\fg}(w\gamma) \leftrightarrow M_{\fg_{\lambda}}(w\gamma)$ and irreducibles $L_{\fg}(w\gamma) \leftrightarrow L_{\fg_{\lambda}}(w\gamma)$ (see \cite[Chapter 1]{Lusztig1984}). Suppose we have matching irreducibles $L \leftrightarrow L'$ in $\cO_{\lambda}(\fg)$ and $\cO_{\lambda}(\fg_{\lambda})$. We will show that if
\begin{equation}\label{eq:translationzero1}\mathrm{pr}_{\chi_{\gamma}}(L' \otimes F')=0 \quad \forall \text{ f.d. $\fg_{\lambda}$-modules $F'$}\end{equation}
then
\begin{equation}\label{eq:translationzero2}\mathrm{pr}_{\chi_{\gamma}}(L \otimes F)=0 \quad \forall \text{ f.d. $\fg$-modules $F$.}\end{equation}
Choose any finite-dimensional $\fg$-module $F$. Since the Weyl group $W_{\lambda}$ of $\fg_{\lambda}$ is a subgroup of $W$, the restriction $F|_{\fh}$ is the character of a virtual $\fg_{\lambda}$-module, which we will denote by $F|_{\fg_{\lambda}}$. We get a homomorphism of Grothendieck groups
$$\mathrm{pr}_{\chi_{\gamma}}(\bullet \otimes F|_{\fg_{\lambda}}): K_0\cO_{\lambda}(\fg_{\lambda}) \to K_0\cO_{\gamma}(\fg_{\lambda}).$$
Moreover, it is clear that the following diagram commutes
\begin{center}
\begin{tikzcd}
K_0\cO_{\lambda}(\fg) \ar[rr, "\mathrm{pr}_{\chi_{\gamma}}(\bullet \otimes F)"] \ar[d,leftrightarrow] &  & K_0\cO_{\gamma}(\fg)\ar[d,leftrightarrow] \\
K_0\cO_{\lambda}(\fg_{\lambda}) \ar[rr,"\mathrm{pr}_{\chi_{\lambda}}(\bullet \otimes F|_{\fg_{\lambda}})"] & & K_0\cO_{\gamma}(\fg_{\lambda})
\end{tikzcd}
\end{center}
For example, this is easy to check in the basis of Verma modules (see \cite[Theorem 3.6]{Humphreys}). Now, (\ref{eq:translationzero1}) implies that $\mathrm{pr}_{\chi_{\lambda}}(L' \otimes F|_{\fg_{\lambda}})=0$. Hence, $\mathrm{pr}_{\chi_{\lambda}}(L \otimes F)=0$ by the commutativity of the diagram.

Step 2. Let $I'_{\lambda}$ denote the (unique) maximal ideal in $U(\fg_{\lambda})$ of infinitesimal character $\chi_{\lambda}$, and let $L'(\lambda)$ denote the irreducible object in $\cO_{\lambda}(\fg_{\lambda})$ corresponding to $\lambda$ so that $\Ann_{U(\fg_{\lambda})}(L'(\lambda)) = I'_{\lambda}$. By Step 1, it suffices to show that 
$$\mathrm{pr}_{\chi_{\gamma}}(L'(\lambda) \otimes F')=0,$$
for all $\gamma \in D^{\circ}(\lambda)$ and all finite-dimensional $\fg_{\lambda}$-modules $F'$. So choose $\gamma \in D^{\circ}(\lambda)$ and a finite-dimensional $F'$.

According to \cite{BarbaschVogan1985}, we have
$$V(I'_{\gamma}) = \overline{d(\OO^{\vee}_{\gamma})},$$
as subsets of $\fg_{\lambda}^*$, where  $d$ denotes the Barbach-Vogan duality on nilpotent co-adjoint orbits. It is known that $d$ is an order-reversing bijection on special nilpotent orbits (and that the orbit $\OO^{\vee}_{\gamma}$ is special). So our assumption $\OO^{\vee}_{\lambda} \not\subseteq \OO^{\vee}_{\gamma}$ implies
$$V(I'_{\gamma}) \not\subseteq V(I'_{\lambda})$$
On the other hand, if we assume that $\mathrm{pr}_{\chi_{\gamma}}(L'(\lambda) \otimes F') \neq 0$, then
$$V(I'_{\gamma}) \subseteq V(\mathrm{Ann}_{U(\fg_{\lambda})}(\mathrm{pr}_{\chi_{\gamma}}(L'(\lambda) \otimes F'))) \subseteq V(\mathrm{Ann}_{U(\fg_{\lambda})}(L'(\lambda)) = V(I'_{\lambda})$$
(this first inclusion follows from the maximality of $I'_{\gamma}$; the second is a basic property of the functor $\mathrm{pr}_{\chi_{\gamma}}(\bullet \otimes F')$). This is a contradiction. So $\mathrm{pr}_{\chi_{\gamma}}(L'(\lambda) \otimes F')=0$. This completes the proof.
\end{proof}

Since unipotent ideals in $U(\fg)$ are maximal (\cite[Theorem 5.0.1]{MBM}), Lemma \ref{lem:weaklyunipotentcriterion} implies the following.

\begin{prop}\label{prop:weaklyunipotentcriterion}
Let $\widetilde{\OO}$ be a nilpotent cover and let $I(\widetilde{\OO})$ be the corresponding unipotent ideal (Definition \ref{def:unipotentideal}). Let $\lambda \in \fh_{\RR}^*$ be the dominant representative of the infinitesimal character of $I(\widetilde{\OO})$. Suppose that $\lambda$ is very weakly unipotent (Definition \ref{def:lambdaweaklyunipotent}). Then $I(\widetilde{\OO})$ is very weakly unipotent (Definition \ref{def:weaklyunipotentideal}).
\end{prop}

Finally, we record the following lemma, which will be useful for computations. The (trivial) proof is left to the reader.

\begin{lemma}\label{lem:integralrootssplit}
Let $\lambda \in \fh^*$. Suppose there are Lie subalgebras $\fa^{\vee}_1, \fa^{\vee}_2 \subset \fg^{\vee}_{\lambda}$ with Cartan subalgebras $\fh^{\vee}_1 \subset \fa^{\vee}_1$, $\fh^{\vee}_2 \subset \fa^{\vee}_2$ such that
$$\fg^{\vee}_{\lambda} = \fa^{\vee}_1 \times \fa^{\vee}_2, \qquad \fh^{\vee} = \fh^{\vee}_1 \times \fh^{\vee}_2.$$
Write $\lambda=\lambda_1+\lambda_2$ for $\lambda_i \in \fh^{\vee}_i$. Then $\lambda$ is very weakly unipotent (with respect to $\fg$) if and only if $\lambda_1$ and $\lambda_2$ are very weakly unipotent (with respect to $\fa_1$ and $\fa_2$, respectively).
\end{lemma}

\begin{comment}
\begin{proof}
Let $\gamma \in \lambda+\ZZ\Phi$ and write $\gamma=\gamma_1+\gamma_2$ with $\gamma_i \in \fh_i^{\vee}$. Then clearly
%
$$\OO^{\vee}_{\gamma} = \OO^{\vee}_{\gamma_1} \times \OO^{\vee}_{\gamma_2} \subset (\fa_1^{\vee})^* \times (\fa_2^{\vee})^*.$$
%
So
%
$$\OO^{\vee}_{\gamma} \subseteq \OO^{\\vee}_{\lambda} \iff \OO^{\vee}_{\gamma_1} \subseteq \OO^{\vee}_{\lambda_1} \text{ and } \OO^{\vee}_{\gamma_2} \subseteq \OO^{\vee}_{\lambda_2} $$
%
The lemma follows at once from Proposition \ref{prop:weaklyunipotentcriterion}.
\end{proof}
\end{comment}

\subsection{Notation for partitions}

For $n \in \ZZ_{\geq 0}$, let $\mathcal{P}(n)$ denote the set of partitions of $n$. Elements of $\mathcal{P}(n)$ are written as $p=[p_1,p_2,....,p_r]$, where $p_1 \geq p_2 \geq ... \geq p_r$ and $\sum_{i=1}^r p_r = n$. We use exponents to indicate multiplicities, i.e. 
$$[p_1^{(m_1)},p_2^{(m_2)},...,p_r^{(m_r)}] = [\underbrace{p_1,...,p_1}_{m_1},\underbrace{p_2,...,p_2}_{m_2},...,\underbrace{p_r,...,p_r}_{m_r}].$$
There is a partial order on $\mathcal{P}(n)$ defined by the relation
\begin{equation}\label{eq:dominanceorder}p \geq p' \iff \sum_{i=1}^j p_i \geq \sum_{i=1}^j p'_j, \ \forall j \geq 1.\end{equation}
The transpose of a partition $p$ is denoted by $p^t$. Note that $p \leftrightarrow p^t$ is an order-reversing involution on $\mathcal{P}(n)$. If $p=[p_1,...,p_r]$ , $q=[q_1,...,q_s]$, and $p_r \geq q_1$, we write
$$p \sqcup q = [p_1,...,p_r,q_1,...,q_s].$$
We call this partition the `concatenation' of $p$ and $q$.

For any subset $S \subset \RR$, we write $S^r_+$ for the set of weakly decreasing $r$-tuples in $S$. If $x = (x_1,....,x_r) \in S^r$, we write $x_+ \in S^r_+$ for the sequence obtained by arranging the entries of $x$ in weakly decreasing order. If $S = \ZZ_{\geq 0}$ and $\sum_{i=1}^r x_i = n$, then we can also regard $x_+$ as an element of $\mathcal{P}(n)$. For $s \in S$ and $x \in S^n$, let $m_s(x)$ denote the multiplicity of $s$ in $x$.

\subsection{Combinatorics in type $A$}

\begin{definition}
For each $\epsilon \in (-\frac{1}{2},\frac{1}{2})$, define maps
\begin{align*}
v_{\epsilon}: \mathcal{P}(n)& \to (\epsilon + \ZZ)^n_+,\\
 v_{\epsilon}([\mu_0,\mu_1,\mu_1',\mu_2,\mu_2',...])& = \begin{cases}
		(\epsilon^{(\mu_0)},(\epsilon-1)^{(\mu_1)},(\epsilon+1)^{(\mu_1')},....)_+, & \text{if $\epsilon \geq 0$,}\\
            (\epsilon^{(\mu_0)},(\epsilon+1)^{(\mu_1)},(\epsilon-1)^{(\mu_1')},....)_+, & \text{if $\epsilon <0$}
		 \end{cases} 
\end{align*}
and
$$p_{\epsilon}: (\epsilon+\ZZ)_+^n \to \mathcal{P}(n), \qquad p_{\epsilon}(v) = [m_v(\epsilon), m_v(\epsilon-1), m_v(\epsilon+1),m_v(\epsilon-2),m_v(\epsilon+2),...]_+.$$
We say that $v \in (\epsilon+\ZZ)^n_+$ is $\epsilon$-triangular if it lies in the image of $v_{\epsilon}$ or, equivalently, if $v = v_\epsilon(p_\epsilon(v))$.
\end{definition}

\begin{lemma}\label{lem:monotonicityA}
For any $\epsilon \in (-\frac{1}{2},\frac{1}{2})$, the map $v_{\epsilon}: \mathcal{P}(n) \to (\epsilon + \ZZ)^n_+$ is order-reversing, i.e. for $p,p' \in \mathcal{P}(n)$ we have
$$p' \leq p \implies |v_{\epsilon}(p)| \leq |v_{\epsilon}(p')|,$$
where $|\bullet|$ denotes the Euclidean norm on $\RR^n$.
\end{lemma}

\begin{proof}
We can reduce to the case when $(p,p')$ is a `minimal degeneration', i.e. $p' < p$ and there is no $q \in \mathcal{P}(n)$ such that $p' < q < p$. Every such pair $(p,p')$ is of the following form
$$p = [p_1,p_2,...,p_r], \qquad p' = [p_1,...,p_{i-1},p_i-1,p_{i+1}+1,p_{i+2},...p_r]$$
where $p_i \geq p_{i+1}+2$. First assume that $\epsilon \geq 0$. Then
\begin{align*}
|v_{\epsilon}(p')|^2 - |v_{\epsilon}(p)|^2 &= \left(\epsilon - (-1)^{i}\left\lfloor \frac{i}{2} \right\rfloor\right)^2(p_i-1) + \left(\epsilon - (-1)^{i+1}\left\lfloor \frac{i+1}{2}\right\rfloor\right)^2 (p_{i+1}+1)\\ 
&- \left(\epsilon - (-1)^{i}\left\lfloor \frac{i}{2} \right\rfloor\right)^2 p_i - \left(\epsilon - (-1)^{i+1}\left\lfloor \frac{i+1}{2}\right\rfloor\right)^2 p_{i+1}  \\
&= \left(\epsilon - (-1)^{i+1}\left\lfloor \frac{i+1}{2}\right\rfloor\right)^2 - \left(\epsilon - (-1)^{i}\left\lfloor \frac{i}{2} \right\rfloor\right)^2\\
&\geq 0
\end{align*}
(since the sequence $|\epsilon|, |\epsilon-1|, |\epsilon+1|, |\epsilon-2|, |\epsilon+2|,...$ is weakly increasing for $\epsilon \in [0,\frac{1}{2})$). 

On the other hand, if $\epsilon <0$, then
\begin{align*}
    |v_{\epsilon}(p')|^2 - |v_{\epsilon}(p)|^2 &= \left(\epsilon + (-1)^{i}\left\lfloor \frac{i}{2} \right\rfloor\right)^2(p_i-1) + \left(\epsilon + (-1)^{i+1}\left\lfloor \frac{i+1}{2}\right\rfloor\right)^2 (p_{i+1}+1)\\ 
&- \left(\epsilon + (-1)^{i}\left\lfloor \frac{i}{2} \right\rfloor\right)^2 p_i - \left(\epsilon + (-1)^{i+1}\left\lfloor \frac{i+1}{2}\right\rfloor\right)^2 p_{i+1}  \\
&= \left(\epsilon + (-1)^{i+1}\left\lfloor \frac{i+1}{2}\right\rfloor\right)^2 - \left(\epsilon + (-1)^{i}\left\lfloor \frac{i}{2} \right\rfloor\right)^2\\
&> 0
\end{align*}
(since the sequence $|\epsilon|, |\epsilon+1|, |\epsilon-1|, |\epsilon+2|, |\epsilon-2|,...$ is increasing for $\epsilon \in (-\frac{1}{2},0)$). 
\end{proof}

\begin{lemma}\label{lem:weaklyunipotentA}
Let $\epsilon \in (-\frac{1}{2},\frac{1}{2})$ and let $v,v' \in (\epsilon+\ZZ)^n_+$. Assume that $v$ is $\epsilon$-triangular and $|v'| < |v|$. Then $p_{\epsilon}(v') \not\leq p_{\epsilon}(v)$.
\end{lemma}

\begin{proof}
Note that $|v'| \geq |v_{\epsilon}(p_{\epsilon}(v'))|$, so that
$$|v_{\epsilon}(p_{\epsilon}(v'))| < |v| = |v_{\epsilon}(p_{\epsilon}(v))|.$$
Now Lemma \ref{lem:monotonicityA} implies that $p_{\epsilon}(v') \not\leq p_{\epsilon}(v)$. 
\end{proof}

\begin{prop}\label{prop:weaklyunipotentA}
Let $G$ be a simple group of type $A$ and let $\widetilde{\OO}$ be a nilpotent cover. Then $I(\widetilde{\OO})$ is very weakly unipotejnt.
\end{prop}

\begin{proof}
Let $\lambda = \lambda(\widetilde{\OO}) \in \fh_{\RR}^*$ denote the dominant representative of the infinitesimal character of $I(\widetilde{\OO})$. We will show that $\lambda$ is very weakly unipotent (Definition \ref{def:lambdaweaklyunipotent}). This will imply that $I(\widetilde{\OO})$ is very weakly unipotent by Proposition \ref{prop:weaklyunipotentcriterion}.

Identify $\fh^{\vee} \simeq \CC^n$ using standard (Bourbaki) coordinates. Positive coroots for $\fg$ (i.e. positive roots for $\fg^{\vee}$) are $\{e_i - e_j\}_{i < j} \subset \fh$.  The Weyl group $W$ is the symmetric group $S_n$ acting on $\CC^n$ by permutation of entries. It is clear from the explicit formulas in \cite[Proposition 9.2.8 and Subsection 9.2.3]{LMBM} that $\lambda$ is $W$-conjugate to a sequence of the form
$$v = \bigsqcup_{i=1}^t v_i,$$
where
\begin{itemize}
    \item $[n_1,n_2,...,n_t]$ is a partition of $n$;
    \item $\epsilon_1,...,\epsilon_t$ are distinct real numbers in $(-\frac{1}{2},\frac{1}{2})$;
    \item each $v_i$ is an $\epsilon_i$-triangular sequence in $(\epsilon_i+\ZZ)^{n_i}_+$;
    \item $\bigsqcup$ denotes concatenation.
\end{itemize}
For positive integers $a \leq b \leq n$, let $\Delta_A^{\vee}(a,b) = \{e_i-e_j\}_{a \leq i \neq j, \leq b}$. Since $\epsilon_i-\epsilon_j \notin \ZZ$ for $i\neq j$, the integral co-roots for $\lambda$ are
$$\Delta^{\vee}_{v} = \Delta^{\vee}_A(1,n_1) \sqcup \Delta^{\vee}_A(n_1+1,n_1+n_2) \sqcup ... \sqcup \Delta^{\vee}_A(n-n_t+1,n).$$
So by Lemma \ref{lem:integralrootssplit}, we can reduce to the case when $\lambda=v$ is an $\epsilon$-triangular $n$-tuple, for some $\epsilon \in (-\frac{1}{2},\frac{1}{2})$. We will assume this for the remainder of the proof. In the notation explained in the beginning of this Appendix, we have
$$D^{\circ}(\lambda)_+ \subseteq \{v' \in (\epsilon +\ZZ)^n_+ \mid |v'| < |v|\}.$$
Nilpotent orbits in $\fg^{\vee}=\fg^{\vee}=\mathfrak{sl}(n,\CC)$ are parameterized by $\mathcal{P}(n)$ (\cite[Section 5.1]{CM}). Under this parameterization, the closure order on orbits corresponds to the dominance order on partitions (\ref{eq:dominanceorder}). Moreover, it is easy to see, using the formula for induction in \cite[Secton 7.2]{CM}, that the nilpotent orbit $\OO^{\vee}_{v'} \subset \mathfrak{sl}(n,\CC)^*$, for any $v' \in (\epsilon+\ZZ)^n_+$, corresponds to the partition $p_{\epsilon}(v')^t \in \mathcal{P}(n)$ and that the Killing form on $\fh^*$ corresponds to a constant multiple of the dot product on $\CC^n$. So by Proposition \ref{prop:weaklyunipotentcriterion}, it suffices to show that
$$p_{\epsilon}(v') \not\leq p_{\epsilon}(v), \qquad \forall v' \in (\epsilon+\ZZ)^n_+ \text{ such that } |v'| < |v|.$$
This is the content of Lemma \ref{lem:weaklyunipotentA}.
\end{proof}

\subsection{Combinatorics in types $B$, $C$, and $D$}

\begin{definition}\label{def:triangular}
\leavevmode
\begin{itemize}
    \item Let $\mathcal{P}^*(2n)$ denote the set of partitions of $2n$ of the form 
    $$\mu = [\mu_1^{(2)},\mu_2^{(2)},...,\mu_{i-1}^{(2)},2\mu_0,\mu_{i}^{(2)},...,\mu_r^{(2)}]$$
    for $1 \geq i \geq r$ and let $\mathcal{P}^{**}(2n) \subset \mathcal{P}^*(2n)$ denote the subset consisting of partitions of the same form but with  $2\mu_0 \geq \mu_1$. Define maps
    \begin{align*}
    v_{\ZZ}: \mathcal{P}^*(2n) &\to (\ZZ_{\geq 0})^n_+, \\
    v_{\ZZ}([\mu_1^{(2)},...,\mu_{i-1}^{(2)},2\mu_0,\mu_{i}^{(2)},...,\mu_r^{(2)}]) &= (r^{(\mu_r)}, (r-1)^{(\mu_{r-1})}, ..., 1^{(\mu_1)},0^{(\mu_0)})
    \end{align*}
    and
    $$p_{\ZZ}: (\ZZ_{\geq 0})^n_+ \to \mathcal{P}^*(2n), \qquad p_{\ZZ}(v) = [2m_v(0),m_v(1)^{(2)},m_v(2)^{(2)}, ...]_+.$$
    We say that $v \in (\ZZ_{\geq 0})^n_+$ is \emph{$\ZZ$-triangular} if it lies in the set $v_{\ZZ}(\mathcal{P}^{**}(2n))$.

    \item Define maps
    $$v_{\frac{1}{2}\ZZ}: \mathcal{P}(n) \to  (\frac{1}{2} + \ZZ_{\geq 0})^n_+, \qquad v_{\frac{1}{2}\ZZ}([\nu_1,\nu_2,...,\nu_s]) = \left(\frac{2s-1}{2}^{(\nu_s)},\frac{2s-3}{2}^{(\nu_{s-1})},...,\frac{1}{2}^{(\nu_1)}\right),$$
    and
    $$p_{\frac{1}{2}\ZZ}: (\frac{1}{2} + \ZZ_{\geq 0})^n_+ \to \mathcal{P}(n), \qquad p_{\frac{1}{2}\ZZ}(v) = \left[m_v\left(\frac{1}{2}\right),m_v\left(\frac{3}{2}\right),...\right]_+.$$
    We say that a $v \in (\frac{1}{2} + \ZZ_{\geq 0})^n_+$ is \emph{$\frac{1}{2}\ZZ$-triangular} if it lies in the image of $v_{\frac{1}{2}\ZZ}$.
    
    \item Let $(\frac{1}{2} + \frac{1}{4}\ZZ_{\geq 0})^n_+$ denote the set of weakly decreasing $n$-tuples in $\frac{1}{2} + \frac{1}{4}\ZZ_{\geq 0}$. Define maps
    $$v_{\frac{1}{4}\ZZ}: \mathcal{P}(n) \to  (\frac{1}{4} + \frac{1}{2}\ZZ_{\geq 0})^n_+, \qquad v_{\frac{1}{4}\ZZ}([\nu_1,\nu_2,...,\nu_s]) = \left(\frac{2s-1}{4}^{(\nu_s)},\frac{2s-3}{2}^{(\nu_{s-1})},...,\frac{1}{4}^{(\nu_1)}\right),$$
    and
    $$p_{\frac{1}{4}\ZZ}: (\frac{1}{4} + \frac{1}{2}\ZZ_{\geq 0})^n_+ \to \mathcal{P}(n), \qquad p_{\frac{1}{4}\ZZ}(v) = \left[m_v\left(\frac{1}{4}\right),m_v\left(\frac{3}{4}\right),...\right]_+.$$
    We say that a sequence $v \in (\frac{1}{4} + \frac{1}{2}\ZZ_{\geq 0})^n_+$ is \emph{$\frac{1}{4}\ZZ$-triangular} if it lies in the image of $v_{\frac{1}{4}\ZZ}$.
\end{itemize}
\end{definition}

\begin{lemma}\label{lem:monotonicity}
\begin{itemize}
    \item Suppose $p \in \mathcal{P}^{**}(2n)$, $p' \in \mathcal{P}^*(2n)$, and $p' \leq p$. Then 
    $$|v_{\ZZ}(p)| \leq |v_{\ZZ}(p')|.$$
    \item Suppose $p,p' \in \mathcal{P}(n)$ and $p' \leq p$. Then
    $$|v_{\bullet}(p)| \leq |v_{\bullet}(p')|, \qquad \bullet \in \left\{\frac{1}{2}\ZZ,\frac{1}{4}\ZZ\right\}.$$
\end{itemize}
\end{lemma}

\begin{proof}
For the first claim, we can reduce to the case when $(p,p')$ is a minimal degeneration in $\mathcal{P}^*(2n)$ of one of the following two types: 
\begin{itemize}
    \item[(i)] %
$$p = [\mu_1^{(2)},\mu_2^{(2)},...,\mu_{i-1}^{(2)},2\mu_0,\mu_{i}^{(2)},...,\mu_r^{(2)}]$$
and
$$p' = [\mu_1^{(2)},(\mu_j-1)^{(2)},(\mu_{j+1}+1)^{(2)},...,\mu_{i-1}^{(2)},2\mu_0,\mu_{i}^{(2)},...,\mu_r^{(2)}]$$
for $1 \leq j \leq i-2$ or $i \leq j \leq r-1$;
    
    \item[(ii)] %
$$p = [\mu_1^{(2)},\mu_2^{(2)},...,\mu_{i-1}^{(2)},2\mu_0,\mu_{i}^{(2)},...,\mu_r^{(2)}]$$
and
$$p' = [\mu_1^{(2)},...,\mu_{i-1}^{(2)},2\mu_0-2,(\mu_{i}+1)^{(2)},...,\mu_r^{(2)}].$$
\end{itemize}
In case (i), we have
$$|v_{\ZZ}(p')| - |v_{\ZZ}(p)| = j^2(\mu_j-1) + (j+1)^2(\mu_{j+1}+1) - j^2(\mu_j) - (j+1)^2 \mu_j > (j+1)^2 - j^2 = 2j+1 > 0$$
In case (ii), we have
$$|v_{\ZZ}(p')| - |v_{\ZZ}(p)| = i^2 (\mu_i+1) - i^2 \mu_i = i^2 >0$$
For the second claim, there is only one type of minimal degeneration, namely

\begin{itemize}
    \item[(iii)] $p = [\nu_1,\nu_2,...,\nu_s]$ and $p' = [\nu_1,...,\nu_t-1,\nu_{t+1}+1,....\nu_s]$. 
\end{itemize}
In this case, we have
\begin{align*}
|v_{\frac{1}{2}\ZZ}(p')| - |v_{\frac{1}{2}}(p)| &= (\nu_t-1)\left(\frac{2t-1}{2}\right)^2 + (\nu_{t+1}+1)\left(\frac{2t+1}{2}\right)^2 - \nu_t\left(\frac{2t-1}{2}\right)^2 - \nu_{t+1}\left(\frac{2t+1}{2}\right)^2\\
&= \left(\frac{2t+1}{2}\right)^2 - \left(\frac{2t-1}{2}\right)^2\\
&>0
\end{align*}
and
\begin{align*}
|v_{\frac{1}{4}\ZZ}(p')| - |v_{\frac{1}{4}}(p)| &= (\nu_t-1)\left(\frac{2t-1}{4}\right)^2 + (\nu_{t+1}+1)\left(\frac{2t+1}{4}\right)^2 - \nu_t\left(\frac{2t-1}{4}\right)^2 - \nu_{t+1}\left(\frac{2t+1}{4}\right)^2\\
&= \left(\frac{2t+1}{4}\right)^2 - \left(\frac{2t-1}{4}\right)^2\\
&>0
\end{align*}
\end{proof}

\begin{lemma}\label{lem:combinatoricsBCD}
\leavevmode
\begin{itemize}
    \item Let $v, v' \in (\ZZ_{\geq 0})^n_+$. Assume that $v$ is $\ZZ$-triangular and $|v'| < |v|$. Then $p_{\ZZ}(v') \not\leq p_{\ZZ}(v)$.

    \item Let $v, v' \in (\frac{1}{2}+\ZZ_{\geq 0})^n_+$. Assume that $v$ is $\frac{1}{2}\ZZ$-triangular
 and $|v'| < |v|$. Then $p_{\frac{1}{2}\ZZ}(v') \not\leq p_{\frac{1}{2}\ZZ}(v)$.

   \item Let $v, v' \in (\frac{1}{4}+\frac{1}{2}\ZZ_{\geq 0})^n_+$. Assume that $v$ is $\frac{1}{4}\ZZ$-triangular
 and $|v'| < |v|$. Then $p_{\frac{1}{4}\ZZ}(v') \not\leq p_{\frac{1}{4}\ZZ}(v)$.
\end{itemize}
\end{lemma}

\begin{proof}
We will prove the first claim (the others are analogous). First, observe that
$$|v'| \geq |v_{\ZZ}(p_{\ZZ}(v'))|$$
so that
$$|v_{\ZZ}(p_{\ZZ}(v'))| < |v_{\ZZ}(p_{\ZZ}(v))|$$
Note that $p_{\ZZ}(v') \in \mathcal{P}^*(2n)$ and $p_{\ZZ}(v) \in \mathcal{P}^{**}(2n)$. So by Lemma \ref{lem:monotonicity} we have $p_{\ZZ}(v') \not\leq p_{\ZZ}(v)$.
\end{proof}

A partition of $2n$ is said to be of `type $C$' if every odd part has even multiplicity. For any $p \in \mathcal{P}(2n)$, let $p_C$ denote the largest type $C$ partition such that $p_C \leq p$. Note that the map $p \mapsto p_C$ is order-preserving. 

\begin{lemma}\label{lem:typeC}
Let $p,p' \in \mathcal{P}^{**}(2n)$. Suppose $(p^t)_C \leq ((p')^t)_C$. Then $p' \leq p$.
\end{lemma}

\begin{proof}
Since $p \mapsto p^t$ is order-reversing and $p \mapsto p_C$ is order-preserving, we have that $p \mapsto (p^t)_C$ is order-reversing. It remains to show that $p \mapsto (p^t)_C$ is injective when restricted to the set $\mathcal{P}^{**}(2n)$. Since $p \mapsto p^t$ is injective on all of $\mathcal{P}(2n)$, it is enough to show that $p \mapsto p_C$ is injective on $\mathcal{P}^{**}(2n)^t$. Let $\mathcal{P}'(2n)$ denote the set of type $C$ partitions $q$ with the following properties:
\begin{itemize}
    \item If $i$ is odd and $p_i$ is even, then $p_{i+1}$ is even;
    \item If $i$ is even and $p_i$ is even, then $p_{i} \geq p_{i+1}+1$.
\end{itemize}
If $q \in \mathcal{P}^{**}(2n)^t$ then all parts of $q$ are odd. It follows that $q_C \in \mathcal{P}'(2n)$. For each $r \in \mathcal{P}'(2n)$, define a partition $\tilde{r} \in \mathcal{P}(2n)$ by
$$\tilde{r}_i = \begin{cases}
  r_i+1  & \text{$i$ is odd and $r_i$ is even} \\
  r_i-1 & \text{$i$ is even and $r_i$ is even} \\
  r_i & \text{else}
\end{cases}$$
It is easy to see that the map $r \mapsto \tilde{r}$ is left-inverse to $(\bullet)_C: \mathcal{P}^{**}(2n)^t \to \mathcal{P}'(2n)$. This completes the proof.
\end{proof}

\begin{prop}\label{prop:weaklyunipotentBCD}
Let $G$ be a simple group of type $B$, $C$, or $D$, and let $\widetilde{\OO}$ be a nilpotent cover. Then $I(\widetilde{\OO})$ is very weakly unipotent.
\end{prop}

\begin{proof}
We will prove the proposition in the type $B$ case. The other cases are completely analagous. 

Let $\lambda=\lambda(\widetilde{\OO}) \in \fh_{\RR}^*$ denote the dominant representative of the infinitesimal character of $I(\widetilde{\OO})$. We will show that $\lambda$ is very weakly unipotent (Definition \ref{def:lambdaweaklyunipotent}). This will imply that $I(\widetilde{\OO})$ is very weakly unipotent by Proposition \ref{prop:weaklyunipotentcriterion}.

We have $\fg = \mathfrak{so}(2n+1,\CC)$ and $\fg^{\vee}=\mathfrak{sp}(2n,\CC)$.  We will identify $\fh^{\vee} \simeq \CC^n$ using standard (Bourbaki) coordinates. The Weyl group $W$ is generated in $GL_n(\CC)$ by permutations and sign changes. Positive coroots for $\fg$ (i.e positive roots for $\fg^{\vee}$) are $\{e_i \pm e_j\}_{1 \leq i < j \leq n} \cup \{2e_i\}_{i=1}^n$.

It is clear from the explicit formulas in \cite[Proposition 4.2.6 and Equation (A.1.3)]{MBM} that $\lambda$ is $W$-conjugate to a sequence of the form
$$v = x \sqcup y \sqcup z,$$ 
where
\begin{itemize}
    \item $r,s,t$ are nonnegative integers such that $n=r+s+t$;
    \item $x=(x_1,...,x_r)$ is a $\ZZ$-triangular sequence in $(\ZZ_{\geq 0})^r_+$;
    \item $y=(y_1,...,y_s)$ is a  $\frac{1}{2}\ZZ$-triangular sequence in $(\frac{1}{2}+\ZZ_{\geq 0})^s_+$;
    \item $(z_1,...,z_t)$ is a $\frac{1}{4}\ZZ$-triangular sequence in $(\frac{1}{4}+\frac{1}{2}\ZZ_{\geq 0})_+^t$;
    \item $\sqcup$ denotes concatenation of sequences.
\end{itemize}
For positive integers $a \leq b \leq n$, let $\Delta_C^{\vee}(a,b) = \{\pm 2e_i\}_{i=a}^b \cup \{\pm e_i \pm e_j\}_{a \leq i < j \leq b}$. Then the integral co-roots for $v$ are
$$\Delta^{\vee}_v = \Delta^{\vee}_C(1,r) \sqcup \Delta_C^{\vee}(r+1,r+s) \sqcup \{\alpha^{\vee} \in \Delta_C^{\vee}(r+s+1,n) \mid \alpha^{\vee}(0 \sqcup 0 \sqcup z) \in \ZZ\}.$$
The three pieces appearing in the disjoint union above are root systems of type $C_r$, $C_s$, and $A_{t-1}$, respectively. Thus, by Lemma \ref{lem:integralrootssplit}, we can reduce to the case when $v$ is purely $\ZZ$-, $\frac{1}{2}\ZZ$-, or $\frac{1}{4}\ZZ$-triangular, i.e. $n \in \{r,s,t\}$.

Let us first assume that $v$ is $\ZZ$-triangular. Then, in the notation explained in the beginning of this appendix, we have
$$D^{\circ}(v)_+ \subseteq \{v' \in (\ZZ_{\geq 0})^n_+ \mid |v'| < |v|\}.$$
Nilpotent orbits in $\fg^{\vee}_v = \fg^{\vee} = \mathfrak{sp}(2n)$ are parameterized by type $C$ partitions of $2n$, see \cite[Section 5.1]{CM}. Under this parameterization, the closure order on orbits corresponds to the dominance order on partitions (\ref{eq:dominanceorder}). Moreover, it is easy to see, using the formula for induction in \cite[Section 7.3]{CM}, that the nilpotent orbit $\OO^{\vee}_{v'} \subset \mathfrak{sp}(2n,\CC)^*$, for $v' \in (\ZZ_{\geq 0})^n_+$, corresponds to the partition $[p_{\ZZ}(v')^t]_C$ of $2n$, and the Killing form on $\fh^*$ corresponds to a constant multiple of the dot product on $\CC^n$. So by Proposition \ref{prop:weaklyunipotentcriterion}, it suffices to show that 
$$[p_{\ZZ}(v)^t]_C \not\leq [p_{\ZZ}(v')^t]_C, \qquad \forall v' \in (\ZZ_{\geq 0})^n_+ \text{ such that } |v'| < |v|.$$
By the first part of Lemma \ref{lem:combinatoricsBCD}, we have
$$p_{\ZZ}(v') \not\leq p_{\ZZ}(v), \qquad \forall v' \in (\ZZ_{\geq 0})^n_+ \text{ such that } |v'| < |v|.$$
This implies the desired relation by Lemma \ref{lem:typeC}.

The argument for $v$ $\frac{1}{2}\ZZ$-triangular is similar, except we use the second part of Lemma \ref{lem:combinatoricsBCD}.

Finally, let us assume that $v$ is $\frac{1}{4}\ZZ$-triangular. Then
$$D^{\circ}(v)_+ \subseteq \{v' \in (\frac{1}{4}+\frac{1}{2}\ZZ_{\geq 0})^n_+ \mid |v'| < |v|\}.$$
Nilpotent orbits in $\fg^{\vee}_v = \mathfrak{gl}(n,\CC)^*$ are parameterized by partitions of $n$, and the closure ordering on orbits corresponds to the dominance ordering on partitions. It is easy to see using \cite[Section 7.3]{CM} that the nilpotent orbit $\OO^{\vee}_{v'} \subset \mathfrak{gl}(n,\CC)^*$ corresponds to the partition $p_{\frac{1}{4}\ZZ}(v')^t$. So by Proposition \ref{prop:weaklyunipotentcriterion}, it suffices to show that 
$$p_{\frac{1}{4}\ZZ}(v') \not\leq p_{\frac{1}{4}\ZZ}(v), \qquad \forall v' \in (\frac{1}{4}+\frac{1}{2}\ZZ_{\geq 0})^n_+ \text{ such that } |v'| < |v|.$$
This is the third part of Lemma \ref{lem:combinatoricsBCD}.
\end{proof}

\subsection{Exceptional groups}

For simple exceptional groups, we check very weak unipotence using computer calculation and the tables in \cite[Appendix]{MBM}. The validity of our algorithm relies on several preliminary lemmas. 

\begin{lemma}\label{lem:translation1}
Let $\lambda, \gamma \in \fh^*$ such that $\lambda-\gamma \in \ZZ\Phi$, and let $M \in \cO_{\chi_{\lambda}}(\fg)$. Suppose
$$\mathrm{pr}_{\chi_{\gamma}}(M \otimes F_{w\gamma-\lambda}) = 0, \qquad \forall w \in W_{\lambda},$$
where $F_{w\gamma-\lambda}$ is the irreducible finite-dimensional $G^{ad}$-representation of extremal weight $w\gamma-\lambda$. Then
$$\mathrm{pr}_{\chi_{\gamma}}(M \otimes F) = 0,$$
for any finite-dimensional $G^{ad}$-module $F$.
\end{lemma}

\begin{proof}
This follows immediately from \cite[Theorem 3.3]{BernsteinGelfand}.
\end{proof}

For $\lambda,\gamma \in \fh^*$ with $\lambda-\gamma \in \ZZ\Phi$, the functor 
$$\mathrm{pr}_{\chi_{\lambda}}(\bullet \otimes F_{\gamma-\lambda}): \cO_{\lambda}(\fg) \to \cO_{\gamma}(\fg)$$
is denoted $T^{\gamma}_{\lambda}$ and is called the \emph{translation functor} from $\lambda$ to $\gamma$. Since for any $\lambda \in \fh^*$, the set $D^{\circ}(\lambda)$ is $W_{\lambda}$-stable, Lemma \ref{lem:translation1} immediately implies the following criterion for very weak unipotence. 

\begin{lemma}\label{lem:translation2}
Let $\lambda \in \fh_{\RR}^*$ and let $M$ be an irreducible object in $\cO_{\chi_{\lambda}}(\fg)$. Then $M$ is very weakly unipotent if and only if
$$T^{\gamma}_{\lambda}(M) = 0, \qquad \forall \gamma \in D^{\circ}(\lambda).$$
\end{lemma}
Lemma \ref{lem:translation2} suggests the following \emph{finite} algorithm for very weak unipotence:

\vspace{5mm}

Let $\lambda \in \fh_{\RR}^*$ be integrally dominant, and let $I \subset U(\fg)$ be the unique maximal ideal of infinitesimal character $\chi_{\lambda}$. 

\begin{enumerate}
    \item  Calculate the finite set $D^{\circ}(\lambda)$.
    \item For each $\gamma \in D^{\circ}(\lambda)$, check that $T^{\gamma}_{\lambda}(L(\lambda))=0$.
\end{enumerate}

\vspace{5mm}

We have implemented a version of this algorithm in the {\tt atlas} software.

For simple exceptional groups, a complete list of unipotent infinitesimal characters can be extracted from \cite[Appendix A]{MBM}. For each infinitesimal character on this list, we check very weak unipotence in {\tt atlas} using the algorithm above. In this way, we prove

\begin{prop}\label{prop:weaklyunipotentexceptional}
Let $G$ be a simple exceptional group and let $\widetilde{\OO}$ be a nilpotent cover. Then $I(\widetilde{\OO})$ is very weakly unipotent. 
\end{prop}

\begin{sloppypar} \printbibliography[title={References}] \end{sloppypar}

\end{document}